\documentclass[12pt, letterpaper]{amsart}
\usepackage{amsmath}

\usepackage[svgnames]{xcolor}

\usepackage{tikz}
\usetikzlibrary{matrix,arrows,decorations.pathmorphing} 
\usepackage{graphicx}
\usepackage{amssymb}
\usepackage{mathrsfs}
\usepackage{latexsym}
\usepackage[hypertexnames=false, colorlinks, citecolor=DarkBlue, linkcolor=DarkBlue]{hyperref}
\hypersetup{bookmarksdepth=3}

\usepackage[obeyFinal]{todonotes}

	\normalbaselines						  %

\newcommand{\fdv}{finite-dimensionally-valued }

\newcommand{\R}{{\mathbb  R}}

\newcommand{\D}{{\mathbb  D}}

\newcommand{\T}{\mathbb{T}}

\newcommand{\Z}{{\mathbb  Z}}
\newcommand{\N}{{\mathbb  N}}
\newcommand{\C}{{\mathbb  C}}

\newcommand{\Leb}{\mathfrak{L}}

\newcommand{\OZ}{{\mathbf{0}}}
\newcommand{\dd}{{\mathrm{d}}}

\newcommand{\OID}{{\mathbf{I}}}
\newcommand{\bI}{\mathbf{I}}
\newcommand{\fL}{\mathfrak{L}}

\newcommand{\fdot}{\,\cdot\,}

\newcommand{\Cay}{\operatorname{Cay}}
\newcommand{\jap}[1]{\langle#1\rangle}
\newcommand{\Jap}[1]{\left\langle#1\right\rangle}

\newcommand{\fX}{\mathfrak{X}}

\newcommand{\ciG}{\ci\Gamma}

\newcommand{\cC}{\mathcal{C}}
\newcommand{\cT}{\mathcal{T}}
\newcommand{\cS}{\mathcal{S}}

\newcommand{\tup}[1]{\textup#1}

\newcommand{\I}{\mathbf{I}}

\newcommand{\bO}{\mathbf{0}}

\newcommand{\bp}{\mathbf{p}}
\newcommand{\bA}{\mathbf{A}}
\newcommand{\bC}{\mathbf{C}}
\newcommand{\bB}{\mathbf{B}}
\newcommand{\bE}{\mathbf{E}}
\newcommand{\bM}{\mathbf{M}}
\newcommand{\bN}{\mathbf{N}}
\newcommand{\be}{\mathbf{e}}

\newcommand{\bx}{\mathbf{x}}
\newcommand{\by}{\mathbf{y}}

\newcommand{\wt}{\widetilde}

\newcommand{\cM}{\mathcal{M}}
\newcommand{\cK}{{\mathcal K}}

\newcommand{\cH}{\mathcal{H}}

\newcommand{\cB}{\mathcal{B}}
\newcommand{\cX}{\mathcal{X}}

\newcommand{\cA}{\mathcal{A}}

\newcommand{\cF}{\mathcal{F}}
\newcommand{\1}{\mathbf{1}}

\newcommand{\f}{\varphi}

\newcommand{\e}{\varepsilon}

\DeclareMathOperator{\cspn}{\overline{span}}

\DeclareMathOperator{\id}{id}
\DeclareMathOperator{\tr}{tr}
\DeclareMathOperator{\Ker}{Ker}
\DeclareMathOperator{\Ran}{Ran}
\DeclareMathOperator{\cRan}{\overline{Ran}}

\DeclareMathOperator{\im}{Im}
\DeclareMathOperator{\re}{Re}
\DeclareMathOperator{\supp}{supp}

\DeclareMathOperator{\rank}{rank}

\newcommand{\ci}[1]{_{_{\scriptstyle #1}}}

\newcommand{\ti}[1]{_{\scriptstyle \text{\rm #1}}}

\newcommand{\ut}[1]{^{\scriptstyle \text{\rm #1}}}

\renewcommand{\labelenumi}{(\roman{enumi})}

\newenvironment{entry}
{\begin{list}{X}%
		{%
			\setlength{\labelwidth}{55pt}%
			\setlength{\leftmargin}{\labelwidth}%\labelsep}%
			\addtolength{\leftmargin}{\labelsep}%
			\setlength{\itemsep}{.4pc}%
		}%
	}%
	{\end{list}}  

\newcounter{vremennyj}

\newcommand\cond[1]{\setcounter{vremennyj}{\theenumi}\setcounter{enumi}{#1}\labelenumi\setcounter{enumi}{\thevremennyj}}

\numberwithin{equation}{section}

\theoremstyle{plain}
\newtheorem{theo}{Theorem}[section]
\newtheorem*{theo*}{Theorem}
\newtheorem{cor}[theo]{Corollary}
\newtheorem{lem}[theo]{Lemma}
\newtheorem{prop}[theo]{Proposition}

\theoremstyle{definition}
\newtheorem{defn}[theo]{Definition}

\theoremstyle{remark}
\newtheorem*{ex*}{Example}
\newtheorem*{exs*}{Examples}
\newtheorem{rem}[theo]{Remark}
\newtheorem*{rem*}{Remark}
\newtheorem*{rems*}{Remarks}

\title[Singular Spectrum for General Perturbations]
{Analysis of the Singular Spectrum for General Perturbations}

\author{Constanze~Liaw}
\address{C.~Liaw: Department of Mathematical Sciences, University of Delaware, Newark, 
DE 19716, USA; and 
CASPER, Baylor University, One Bear Place \#97328,      
Waco, TX  76798, USA}
\email{liaw@udel.edu}
\author{Eero~Saksman}
\address{E.~Saksman: Department of Mathematics and Statistics, University of Helsinki,  
P.O. Box 68, FI-00014 University of Helsinki, Finland}
\email{eero.saksman@helsinki.fi}

\author{Sergei~Treil}
\address{S.~Treil: Department of Mathematics, Brown University   
151 Thayer
Str./Box 1917,      
Providence, RI  02912, USA}

\email{treil@math.brown.edu}

\thanks{The work of C.~Liaw is supported by the National Science Foundation under 
the  DMS grant 2452894. Work of E.~Saksman has been supported by the Finnish Academy Center of 
Excellence FiRST.  Work of S.~Treil was 
supported by the National Science Foundation under 
the grants DMS-2154321,  DMS-2452407. All three authors were supported by BIRS workshop 26rit044.
}
\keywords{Spectral theory, perturbation theory, singular spectrum}
\subjclass[2010]{47A55, 47A56, 30E20, 28B05}

\begin{document}

\begin{abstract}
We investigate the behavior of the singular spectrum of self-adjoint operators under families of 
Hermitian perturbations, 
even non-compact ones. Under mild assumptions we have the ``shift'' of the singular spectrum for 
almost all values of the 
parameter. Moreover, if we consider trace class perturbations, we observe the ``shift'' outside a 
countable set of the 
values of the parameter. 
While similar results were known for finite rank perturbations, the extensions to trace class 
perturbations 
are far from easy, and the proofs require significant new ideas. 

In addition,  some of our results hold  (surprisingly enough!) even for all positive and bounded 
perturbations. 
This is a consequence of our generalized Aleksandrov disintegration theorem established in this 
paper.   We use some advanced techniques, involving Sz.-Nagy--Foia\c s theory and the operator 
$\bA_2$ condition.   
\end{abstract}

\maketitle
\setcounter{tocdepth}{1}
\tableofcontents

\setcounter{tocdepth}{2}

\setcounter{section}{-1}

%%%%%%%%%%%%%%%%%%%%%%%%%%%%%%%%%%%%%%%%%%%%%%%%%
%%%%%%%%%%%%%%%%%%%%%%%%%%%%%%%%
\section{Notation and conventions}
%%%%%%%%%%%%%%%%%%%%%%%%%%%%%%%%
In this paper, unless otherwise specified, all Hilbert spaces are assumed to be separable; 
finite-dimensional Hilbert spaces are allowed.  

Unless otherwise specified, all measures are finite (or of bounded variation for signed and 
complex-valued measures), and all measures on topological spaces are assumed to be Borel.

Here are some notations we use throughout:
\begin{entry}
\item[${\bf 1}\ci{E}$] the characteristic function of the set $E$; we also use ${\bf 1}={\bf 
1}\ci{\R}$;
\item[$\id$] the identity function, usually on $\R$,  $\id(s)\equiv s$; 
\item[$\bE$] usually stands for a projection-valued spectral measure; the super-/sub-scripts 
(e.g.~$\bE^\Gamma$, $\bE\ci K$ and $\bE^t$) indicate the perturbation operators or parameters;
\item[$\bM$] operator-valued measures are upper case and bold face; the super-/sub-scripts 
(e.g.~$\bM^\Gamma$ and $\bM^t$) are to clarify the perturbation operators or parameters; for 
scalar-measures we use lower case Greek letters without bold face;
\item[$L^2(\bM)$] Lebesgue space for operator-valued 
measures $\bM$ (see Section \ref{s-SpMeas});
\item[$\bM(z)$] Poisson extension of an operator-valued measure $\bM$ to the point 
$z\in \C\setminus \R$, see Section \ref{s-reptoa2}; we use $\mu(z)$ analogously; 
\item[$\cM_\f$] multiplication operator in a function space by the scalar function  $\f$, 
\linebreak $\cM_\f f 
:= \f f$; in particular $\cM_{\id}$ is the multiplication by the independent variable, 
$\left(\cM_{\id} f\right)(s) = s f(s)$; 
\item[$\cRan$] closure of the range;
\item[$\cspn$] closed linear span;
\item[$p_a$] Poisson kernel, see \eqref{d-poisson}; in \eqref{def:p-oper} we denote by $\bp\ci{R}$ 
an analog of the Poisson kernel for operator $R$; in Appendix \ref{s-apppoiss} we also use upper 
case $P_\delta = p_{i\delta}.$
\end{entry}

Recall that for an operator $T$ in a Hilbert space, its imaginary part $\im T$ is given by 
$\im T:= \frac{T-T^*}{2i}$. 

%%%%%%%%%%%%%%%%%%%%%%%%%%%%%%%%%%%%%%%%%%%%%%%

%%%%%%%%%%%%%%%%%%%%%%%%%%%%%%%%%%%%%%%%%%%%%%%%%%
\section{Introduction and main results}
%%%%%%%%%%%%%%%%%%%%%%%%%%%%%%%%%%%%%%%%%%%%%%%%%%%
The classical Aronszajn--Donoghue theorem says that for a family of  perturbations $A_t:= A+tK$ of a
self-adjoint operator $A$ by a rank one operator $K=K^*$ with cyclic range, the singular parts of
the spectral measures of $A_t$ and $A_s$, $t\ne s$ are mutually singular to each other
\cite{Aronszajn, Donoghue}\footnote{We refer the reader  to the classical text \cite{katobook} for
general perturbation theory, and to \cite{SIMREV} for  rank one perturbations.}. While this theorem
fails even for  rank 2 perturbations, it was shown in  \cite{JST} that a version of this theorem
``with countably many exceptions'' holds for families of finite rank perturbations. It was also
shown there that  a version of the Aronszajn--Donoghue theorem where one replaces mutual singularity
by the so called ``vector mutual singularity'',  holds for finite rank perturbations as well.

In this paper we investigate the behavior of the singular spectrum for wide classes of infinite rank
perturbations. Since the results for finite rank perturbations are already known, one would expect
the results to  extend to the trace class perturbations. While, as it is proved in this paper, such
an extension is indeed possible, see Theorems \ref{t:GA-D_trace}, \ref{t:two except dim},
\ref{t:many except dim} and \ref{t:GA-D_02} below, they are far from trivial.

For trace class perturbations we also obtain an analogue of the Aronszajn--Donoghue theorem, where
we replace the notion of mutual singularity by the so-called  vector mutual singularity,  see
Theorem \ref{t:VMS} below.

Moreover, for more general (non-compact) perturbations, we observe the ``shift'' of singular
spectrum for almost all parameters (exceptional set having measure $0$), see Theorems
\ref{t:GA-D_01}, \ref{t-two} and \ref{t-many} below.

The proofs of the last mentioned results rely on the generalization of the Aleksandrov
Disintegration Theorem\footnote{This result is also known as the Aleksandrov spectral averaging
theorem.} \cite{Aleksandrov} to the case of infinite rank perturbations, see Theorem
\ref{t:AlAver01} below. However, the proof of Theorem \ref{t:AlAver01} is significantly more
involved than the proof of the original Aleksandrov Disintegration Theorem or its matrix-valued
version \cite{Elliott, JST, Martin}. In particular, it  involves Sz.-Nagy--Foia\c s theory of the
functional model in an essential way.

In order to avoid unnecessary technical details, we consider here only bounded self-adjoint
operators. However, we note that most of our results can be stated with minor modifications in the
proofs for bounded perturbations of unbounded operators. Moreover, we expect some of the results to
ultimately hold in the case of so-called form bounded perturbations of unbounded operators.

In  this section by \emph{spectral measure} we always mean a scalar spectral measure of maximal
spectral type, see Section \ref{s:MaxType} for details. 

\subsection{Singular spectrum of perturbations}\label{s:one}

\begin{theo}\label{t:GA-D_01}
Let $A$, $K$ be self-adjoint operators in a separable Hilbert space, $K\ge\bO$ be bounded, and let
$\sigma$ be an arbitrary singular measure on $\R$.  If $\Ran K$ is cyclic for $A$, then for almost
all $t\in\R$ \tup(with respect to the Lebesgue measure\tup) the spectral measure of $A+tK$
\tup(equivalently, its singular part\tup) is mutually singular with $\sigma$.
\end{theo}

Theorem \ref{t:GA-D_01} is a corollary of a general Aleksandrov spectral disintegration theorem, 
see Theorem \ref{t:AlAver01} below. Explanation of how Theorem \ref{t:AlAver01} implies Theorem 
\ref{t:GA-D_01} will be given in Section \ref{s:Proof a.e. except}.
In fact we will prove even a bit stronger result, see Theorem \ref{t:mu_t(Omega)=0} there.

For trace class perturbations  we can say even more. 

\begin{theo}\label{t:GA-D_trace}
Let $A$, $K$ be self-adjoint operators in a separable Hilbert space, and let us assume that
$K\ge\bO$ belongs to the trace class.  Let $\sigma$ be an arbitrary singular measure on $\R$.  If
$\Ran K$ is cyclic for $A$, then for all $t\in\R$ except maybe countably many, the spectral measure
of $A+tK$ \tup(equivalently, its singular part\tup) is mutually singular with $\sigma$.
\end{theo}

The classical Aronszajn--Donoghue theorem states that if $K$ is a rank one operator (with $\Ran K$
still being cyclic), then the singular parts of the spectral measures of $A$ and of $A+tK$ are
mutually singular for all $t\ne0$. It is not hard to see considering direct sums of operators that
the classical Aronszajn--Donoghue theorem cannot hold for all $t\ne0$ even for rank $2$
perturbations. Thus Theorem  \ref{t:GA-D_01} can be considered as an \emph{almost everywhere}
generalization of the Aronszajn--Donoghue theorem for pretty general one-parameter families of 
bounded operators. Similarly, Theorem \ref{t:GA-D_trace} can be treated as a generalization of the 
Aronszajn--Donoghue theorem  to the trace class perturbations.

In the Aronszajn--Donoghue theorem, $\sigma$ is the singular part of the scalar spectral measure (of
maximal type) of $A$, but Theorem  \ref{t:GA-D_trace}  holds for arbitrary singular measures. Thus,
our Theorem  \ref{t:GA-D_trace} is new even for rank one perturbations. Note, that 
Aronszajn--Donoghue theorem and Theorem \ref{t:GA-D_trace}  do not follow from each other.

The requirement  that $K\ge\bO$ looks a bit restrictive, but we should remind the reader that any
rank one operator is sign definite, so changing $t$ by $-t$ if necessary we can always assume that
for rank one perturbations $K>\bO$.  We do not know if Theorems  \ref{t:GA-D_01} and/or
\ref{t:GA-D_trace} hold without the assumption $K\ge\bO$.

However, we can get rid of this assumption if we consider two parameter families of perturbations.
Namely, let us decompose the perturbation $K=K^*$ as $K=K_+-K_-$, where $K_\pm = \pm K \bE\ci{K}
(\R_\pm)$ are positive and negative parts of $K$.

\begin{theo}\label{t-two}
Let $A$, $K$ be self-adjoint operators in a separable Hilbert space, and $ K=K_+-K_-$ be the 
spectral decomposition of $K$ into positive and negative parts as above. Let $\sigma$ be an 
arbitrary singular measure on $\R$. 

If $\Ran K$ is cyclic for $A$ then the spectral measure of the operator
\begin{align*}
A_{\vec t} = A+ t_1 K_+ + t_2 K_-, \qquad \vec t = (t_1, t_2)\in \R^2, 
\end{align*}
is mutually singular with $\sigma$ for \tup(Lebesgue\tup) almost all $\vec t\in \R^2.$
\end{theo}

In fact, a more general result for $n$-parameter perturbations holds.

\begin{theo}\label{t-many}
For given self-adjoint operators $A$, $K_1, \dots, K_n$ on $\cK$ consider an $n$-parameter family of
perturbations
\begin{align*}
A_{\vec t} = A+ K_{\vec t}  \text{ with }\vec t := (t_1, \dots, t_n)\in \R^n\text{ and 
} K_{\vec t}=\sum_{k=1}^n t_k K_k,   
\end{align*}
and let $\sigma$ be an arbitrary singular Radon measure on $\R^n$. If for some $\vec \tau\in\R^n$
the operator $K_{\vec \tau} $ is non-negative and $\Ran K_{\vec \tau}$ is cyclic for $A$, then the
spectral measure of $A_{\vec t}$ is mutually singular with $\sigma$ for \tup(Lebesgue\tup) almost
all $\vec t\in\R^n$. 
\end{theo}
Theorem \ref{t-two} trivially follows from Theorem \ref{t-many}, just choose $\vec\tau = (1,1)$. We
will prove Theorem \ref{t-many} in Section \ref{s-parameters}.

Again, for $K$ in the trace class we can say even more: 

\begin{theo}\label{t:two except dim}
Let $A$, $K$ be self-adjoint operators in a separable Hilbert space, and let  $K$ be in the trace
class.  Let $\sigma$ be an arbitrary singular measure on $\R$.  Let $K=K_+-K_-$ be the spectral
decomposition of $K$ into positive and negative parts as above.

If $\Ran K$ is cyclic for $A$ then the spectral measure of the operator
\begin{align*}
A_{\vec t} = A+ t_1 K_+ + t_2 K_-, \qquad \vec t = (t_1, t_2)\in \R^2,
\end{align*}
is mutually singular with $\sigma$ for  all $\vec t\in \R^2\setminus E$, where the exceptional set 
$E$ has Hausdorff dimension at most $1$. 
\end{theo}

To state the theorem about the dimension drop for $n$-parameter perturbations, we use the same 
notation as in Theorem \ref{t-many}. 
\begin{theo}\label{t:many except dim}
Let $A$, $K_1, \dots, K_n$, be self-adjoint operators and let all $K_k$ be in the trace class.
Assume that there exists a set of directions $D\subset S^{n-1}\subset \R^n$ of positive surface
measure, such that for all $\vec \tau\in D$ the  operators $K_{\vec \tau}$ are non-negative and 
$\Ran K_{\vec \tau}$ is cyclic for $A$.

Then, given an arbitrarily singular Radon measure $\sigma$ on $\R$, the spectral measure of 
$A_{\vec t}$ is mutually singular with $\sigma$ for all $\vec t\in \R^n\setminus E$, where the 
exceptional set $E=E_\sigma$ has Hausdorff dimension at most $n-1$.
\end{theo}

Again, Theorem \ref{t:two except dim} is a particular case of Theorem \ref{t:many except dim}: we 
just take for $D\subset S^1$ the collection of all vectors $\vec t\,=(t_1, t_2)\in S^1$ such that 
$t_1, t_2>0$.  

We will prove Theorem \ref{t:many except dim} in Section \ref{s:exceptional}.

\subsection{Factorization of perturbations and Aleksandrov Disintegration Theorem}\label{s:two}
The main tool used to prove Theorems \ref{t:GA-D_01} and \ref{t-two} is the so-called Aleksandrov 
Disintegration Theorem, see Theorem \ref{t:AlAver01} below. This theorem is of independent 
interest, and can also  be considered as a main result of this paper. 

To state this theorem we need to introduce some notation. 

It is not hard to show, and will be shown later, see Lemma \ref{l:factorization} below, that any 
operator 
$K=K^*$ 
in a Hilbert space $\cH$ 
can be represented as 
$K = \bB \Gamma \bB^*$, where $\bB:\cK\to \cH$ is a bounded operator with $\Ker \bB=\{\bO\}$ and 
$\Gamma=\Gamma^*$ is an invertible bounded operator in an auxiliary Hilbert space $\cK$. Note that 
$\Gamma\ge\bO$ if and only if $K\ge\bO$.

Moreover, if $K$ is a trace class operator, it can be factorized as $K = \bB \Gamma \bB^*$, where 
$\bB:\cK\to \cH$ is a Hilbert--Schmidt operator with $\Ker\bB=\{\bO\}$, and $\Gamma$ is an 
invertible operator. 

It is also easy to see that in the above factorization $\Ran K$ is cyclic for $A$ if and only if 
$\Ran \bB$ is 
cyclic for $A$.

So,  factorizing $K$ we can rewrite the family of perturbations $A_t = A+tK$ as 
\begin{align*}
A_t = A + t \bB \Gamma \bB^*, \qquad t\in \R, 
\end{align*}
where $\Gamma=\Gamma^*$ is invertible and $\Ker \bB=\{\bO\}$. 

For a self-adjoint operator $A$ let $\bE=\bE\ci{A}$ denote its projection-valued spectral measure.

Let $\bE_t$ be the projection-valued spectral measure of the operator $A_t$, and let $\Leb$ denote 
the Lebesgue--Borel measure on $\R$.

\begin{theo}[Aleksandrov disintegration]\label{t:AlAver01}
Let in the above assumptions $\Gamma=\Gamma^*\ge \bO$ be invertible, and let $\Ker \bB=\{\bO\} $. 
Then 
\begin{align*}
\int_\R \bB^* \bE_t \bB \dd t = \Gamma^{-1}\otimes \Leb , 
\end{align*}
meaning that for any Borel $E\subset \R$ 
\begin{align*}
\int_\R \bB^* \bE_t (E) \bB \dd t = \Leb(E) \Gamma^{-1}.
\end{align*}
\end{theo}
The last integral is to be understood in the weak sense (see Section \ref{ss-integration}).

Note that in the above theorem we did not assume that $\Ran \bB$ is cyclic for $A$.

For finite rank perturbations ($\rank K<\infty$) Theorem \ref{t:AlAver01} was proved in \cite{JST}.
It was reasonable to expect that the result from \cite{JST} could be extended to trace class
perturbations; the fact that we were able to get is for general perturbations was a real surprise.

This result is essentially Theorem \ref{t-disintegration1},  applied to the function $f:=\1\ci{E}.$ 
We carry out this argument in Section \ref{s-disintegration}.

\subsection{Families of trace class perturbation along  curves}\label{s:threeee}

The next main result is a generalization of Theorem \ref{t:GA-D_trace} to ``perturbations along a 
curve''. Let $t\mapsto\Gamma(t)=\Gamma(t)^*\in B(\cK)$ be a $C^1$ map from an (open) interval 
$I\subset \R$ to the self-adjoint operators on $\cK$; $C^1$ here means $C^1$ in the norm of 
$B(\cK)$.

\begin{theo}\label{t:GA-D_02}
Let $A_t= A+ \bB \Gamma(t)\bB^*$, where $\bB$ is a Hilbert--Schmidt operator, and  $t\mapsto
\Gamma(t)$ be a $C^1$ map as above, such that $\Gamma'(t)$ is positive definite and invertible for
all $t\in I$, and let $\sigma$ be an arbitrary singular Radon measure.

If $\Ran \bB$ is cyclic for $A$, then for all $t\in I$ except maybe countably many, the spectral
measure of $A_t$ is mutually singular with $\sigma$.
\end{theo}

Theorem \ref{t:GA-D_trace} is a particular case of the above theorem. One can see that by
factorizing $K$ from Theorem  \ref{t:GA-D_trace} as $K=\bB\Gamma\bB^*$, as discussed in Section
\ref{s:two}, see Lemma \ref{l:factorization}, and defining $\Gamma(t) =t\Gamma$.

%%%%%%%%%%%%%%%%%%%%%%%%%%%%%%%%%%%%%%%%%%%%%%%%%%
\section{Operator-valued measures $\bM$ and associated $L^2(\bM)$ spaces}\label{s-prelim}
%%%%%%%%%%%%%%%%%%%%%%%%%%%%%%%%%%%%%%%%%%%%%%%%%%
%%%%%%%%%%%%%%%%%%%%%%%%%%%%%%%%%%%%%%%%%%%%%%%%%%
Operator-valued measures are one of the main technical tools of the present paper. For the 
convenience of the reader, we present some well-known and some less known facts about them in this 
section.

\subsection{Measurability  of vector and operator-valued functions}\label{s-measOVF}
Let us recall main definitions. Let $(\fX, \cA)$ be a measurable space, i.e.~a set $\fX$ endowed 
with a sigma-algebra $\cA\subset 2^{\fX}$. Let $X$ be a Banach space, and $f:\fX\to X$. We say that 
$f$ is a simple function if it can be represented as $f=\sum_{k=1}^n1_{E_k}x_k$, where $x_k\in X$ 
and $E_k\in \cA$ are disjoint. 

A function $f:\fX\to X$ is called \emph{strongly measurable} if it is a pointwise limit of simple 
functions, and it is \emph{weakly measurable} if the scalar function $s\mapsto \langle f(s), 
h\rangle$ 
is 
measurable for any functional $h\in X^*$. 

It follows from Pettis' measurability theorem, see \cite[Theorem 1.1.6]{Hytoenen-book2016} that 
if $X$ is separable, then weak and strong continuity coincide. 

An operator-valued function $f:\fX\to B(X, Y)$ is called weakly (resp.~strongly) measurable if for 
any $\bx\in X$ the vector-valued function $s\mapsto f(s)\bx$ is weakly (resp.\ strongly) 
measurable. Note that if $Y$ is separable, then both notions coincide and we will simply  call 
such  functions \emph{measurable}. 

\begin{rem}
The terminology above is a bit confusing and imprecise, because one can confuse it with the notions
of measurability when $B(X,Y)$ is treated as a Banach space. The more appropriate terms for
operator-valued functions would be something like ``weak operator measurable'' (resp.\ ``strong
operator measurable''). But this  ``weakly (strongly) measurable'' terminology  for operator-valued
functions is widely used in the literature, and we do not feel comfortable fighting against
established terminology.
\end{rem}

\subsection{Operator-valued measures}\label{s:OVM-01} 
In this paper we always deal only with positive operator-valued measures.

Let us recall the main definitions. Let $\fX$ be a measurable space with a  $\sigma$-algebra of
$\cA$. For a Hilbert space $\cK$ let $B(\cK)$ be the set of bounded linear operators on $\cK$ and
$B_+(\cK)$ be the cone of non-negative operators in $B(\cK)$. A (positive) operator-valued measure
on a $\fX$ is a function $\bM:\cA\to B_+(\cK)$ which is weakly countably additive. Weak countable
additivity simply means that for every $\bx\in\cK$ the function
\begin{align}\label{e:POVM-01}
E\mapsto (\bM(E)\bx,\bx)\ci\cK, \qquad E\in\cA
\end{align}
is countably additive, i.e.~it is a (positive) measure. Of course, the polarization identity
immediately implies that for every $\bx,\by\in\cK$  the function $\mu_{\bx,\by}$
\begin{align}\label{e:mu_{xy}}
\mu_{\bx,\by}(E) = (\bM(E)\bx,\by)\ci\cK, \qquad \forall E\in\cA
\end{align}
is a complex-valued measure. For brevity of notation we will denote $\mu_\bx:=\mu_{\bx,\bx}.$

Note that we always assume that the measure $\bM$ is finite, i.e.~that $\bM(\fX)$ is a bounded 
operator.

If $\fX$ is a topological space (usually a locally compact Hausdorff), the $\sigma$-algebra $\cA$ 
is the Borel $\sigma$-algebra.  In this case an operator-valued measure is called \emph{regular} if 
for any $\bx\in\cK$ the measure \eqref{e:POVM-01} is regular. 

The term semi-spectral measure is sometimes used in the literature. The abbreviation POVM (Positive 
Operator-Valued Measure) is also used in the literature, although usually it is assumed that 
$\bM(\fX)=\bI$. 

A positive operator-valued measure is called  \emph{spectral}  if it is ``multiplicative'' 
\begin{align*}
\bM(E_1)\bM(E_2) = \bM(E_1\cap E_2)\qquad \forall E_1, E_2\in\cB.  
\end{align*}
It is easy to see that the measure $\bE$ is spectral if and only if the values $\bM(E)$ are
orthogonal projections, and for $E_1\cap E_2=\varnothing$ the ranges of $\bM(E_1)$ and $\bM(E_2)$
are orthogonal. To emphasize this fact we will often use the term \emph{projection-valued spectral
measure}.

Sometimes in the literature it is also required that a spectral measure is a probability one,
i.e.~that $\bM(\fX)=\bI$. 

We defined an operator-valued measure to be only weakly countably additive, but  in fact any 
operator-valued measure is strongly countably additive i.e.~for disjoint sets $E_k\in\cA$ 
\begin{align*}
\bM\left(\bigcup_k E_k\right)   = \sum_k \bM(E_k) 
\end{align*}
with the convergence in the sum is understood as the convergence in the strong operator topology. 

Indeed, it is easy to see that for the  projection-valued spectral measures weak countable
additivity implies strong countable additivity. As for the general semi-spectral measures the strong
countable additivity follows from the Naimark Dilation Theorem which states, see \cite[Theorem
4.6]{Paulsen}, that for any regular operator-valued Borel measure on a Hausdorff compact $\fX$ with
values in $B_+(\cK)$ there exists a spectral (projection-valued) measure $\bE$ whose values are
orthogonal projections in some auxiliary Hilbert space $\cH$ and a bounded operator $\bC:\cH\to \cK$
such that
\begin{align}\label{e:comp-01}
\bM= \bC^*\bE \bC. 
\end{align}
Here \emph{Borel} means that the measure $\bM$ is defined on Borel subsets of $\fX$, and 
\emph{regular} means that  for any $\bx\in\cK$ the scalar measure \eqref{e:POVM-01} is regular.

This immediately implies strong countable additivity for regular measures on Hausdorff compacts, but
these conditions can easily be relaxed. Namely, the strong countable additivity we need is a
statement about countable collection of non-negative operators, i.e.~about operator-valued measures
on $\N$. But the set $\N$ can be of course embedded in the compact set $[0,1]$, so we are dealing
with operator-valued measure on a (Hausdorff) compact. But since any finite measure on $\R$ is
regular, all assumptions of the Naimark Dilation Theorem are satisfied, so we can apply it to get
countable additivity.

\subsubsection{Any operator-valued measure on $\R$ is compression of a projection-valued measure.}
Note that the Naimark Dilation Theorem holds for Borel measures on any locally compact Hausdorff
(LCH) space $\fX$, in particular for measures on the real line $\R$. Indeed, we can consider
one-point compactification $\hat \fX=\fX\cup\{\infty\}$ of $\fX$, and treat $\bM$ as a measure on
$\hat\fX$. Since any Borel measure on a (Hausdorff) compact is regular, we can apply the Naimark
Dilation Theorem stated above to get a projection-valued measure $\bE$ on $\hat \fX$, such that
\eqref{e:comp-01} holds. Since $\bM(\{\infty\}) =\bO$, we can always get that $\bE(\{\infty\})
=\bO$, so $\bE$ can be treated as a measure on $\fX$.

In this paper  the operator-valued measures we use are defined as compressions of projection-valued
measures. However, any Borel POVM on $\R$ is a compression of a projection-valued measure, so there
is nothing special about the measures we are working with.

\subsection{Integration with respect to operator-valued measures}\label{ss-integration}
All integration with respect to operator-valued measures will be defined in the \emph{weak sense},
using the scalar measures $\mu_{\bx,\by}$ introduced above in \eqref{e:mu_{xy}}.  Namely, for a
scalar bounded $\cA$-measurable function $\f$ on $\fX$ and POVM $\bM$ on $(\fX, \cA)$ we define the
operator $\int \f\bM =\int \f(s)\bM(\dd s)$
by 
\begin{align*}
\Jap{\Bigl( \int \f(s)\bM (\dd s) \Bigr)\bx, \by} =\int \f(s)  \mu_{\bx,\by}(\dd s), \qquad 
\forall \bx,\by\in \cK . 
\end{align*}
For a vector-valued function $f$, $f(s)=\f(s)\bx$ where, as above, $\f$ is a bounded 
$\cA$-measurable scalar function and $\bx\in\cK$ the vector $\int \bM(\dd s)  f(s)$, is defined by
\begin{align*}
\Jap{\Bigl( \int \bM (\dd s) f(s) \Bigr), \by}=
\Jap{\Bigl( \int \bM (\dd s) \f(s) \Bigr)\bx, \by} =\int \f(s)  \mu_{\bx,\by}(\dd s), \qquad 
\forall  \by\in \cK ; 
\end{align*}
This definition trivially extends to finite linear combinations $f=\sum_k \f_k\bx_k$, and it is not 
hard to check that the integral does not depend on the representation.

\begin{defn}
We say that $f:\R\to \cK$ is a \emph{\fdv}function, if all the values $f(s)$ belong to some 
finite-dimensional subspace $\cK'\subset \cK$. 
\end{defn}

Note, that  $f$ is a \fdv function if and only if it can be represented as a finite linear 
combination
\begin{align}\label{e:fdr01}
f=\sum_k \f_k \be_k
\end{align}
where $\be_k\in\cK$ and $\f_k$ are scalar  functions. 

Note that we can always choose the vectors  $\be_k$ to be linearly independent (just pick a basis
and ignore the rest), so for a bounded finite-dimensionally-valued function $f$, the functions $\f_k
$ in \eqref{e:fdr01} can be chosen to be bounded.

Also, if a \fdv function $f$ is measurable, by the same reasoning the functions $\f_k$ can be 
chosen to be measurable.

Finally, for finite-dimensionally-valued bounded measurable functions $f = \sum_k \f_k\bx_k$ and 
$g=\sum_k \psi_k\by_k$ one can define a sesqui-linear form
\begin{align}\label{e:inner prod}
\int\Jap{ \bM(\dd s) f(s), g(s)}\ci\cK := \sum_{j,k} \int \f_k(s) \overline{\psi_j(s)} \mu_{\bx_k, 
\by_j}(\dd s).
\end{align}

Note that it is not hard to show that the result does not depend on the representation of $f$ and 
$g$, and that the form is positive semidefinite. One way to see that is to notice that everything 
happens in the finite-dimensional space $\cspn_{k}\{\bx_k, \by_k\}$, so essentially we are dealing 
with matrix-valued measures. And for matrix-valued measures the above facts are trivial.

%%%%%%%%%%%%%%%%%%%%%%%%%%%%%%%%%%
\subsection{Weighted spaces $L^2(\bM)$ for operator-valued measures}\label{s-SpMeas}

%%%%%%%%%%%%%%%%%%%%%%%%%%%%%%%%%%
Let $\bM$ be a finite operator-valued measure on $(\fX, \cA)$ with values in $B_+(\cK)$. We want to
define the weighted space $L^2(\bM)$.  For finite-dimensionally-valued functions $f$ and $g$ define
the (semi) inner product $\jap{f,g}\ci{L^2(\bM)}$ as 
\begin{align*}
\jap{f,g}\ci{L^2(\bM)} = \int\Jap{ \bM(\dd s) f(s), g(s)}\ci\cK , 
\end{align*}
where the integral is defined above in \eqref{e:inner prod}, and the (semi)norm 
$\|f\|\ci{L^2(\bM)}$,  $\|f\|\ci{L^2(\bM)}^2:=  \Jap{f,f}\ci{L^2(\bM)} $. Taking the quotient over
the set of functions $f$, $\|f\|\ci{L^2(\bM)}=0$, we get an inner product space. Taking completion
we get a Hilbert space we will denote by $L^2(\bM)$.

\begin{lem}\label{l-finiteapprox}
Let $\cspn\ci\cK\{\be_k: k\ge1\}=\cK$.  Then the finite linear combinations
\begin{align*}
\sum_{k=1}^N f_k \be_k , 
\end{align*}
where $f_k$ are bounded $\cA$-measurable functions are dense in $L^2(\bM)$. 
\end{lem}

\begin{proof}
First recall that, by the definition of $L^2(\bM)$ finite linear combination of vectors $\sum f_k
\bx_k$, $\bx_k\in \cK$, $f_k\in L^\infty$, are dense in $L^2(\bM)$. Assume  that 
$\|f_k\|_\infty\le C$. Let $\wt\bx_k$ be finite linear combinations
\begin{align*}
\wt\bx_k= \sum_n c_{n,k} \be_n
\end{align*}
approximating $\bx_k$ in $\cK$ (we will see necessary precision in a moment). We can estimate 
\begin{align*}
\left\|  \sum_{k=1}^N f_k \bx_k - \sum_{k=1}^N f_k \wt\bx_k \right\|_{L^2(\bM)}  & =  \left\|  
\sum_{k=1}^N f_k (\bx_k -  \wt\bx_k) \right\|_{L^2(\bM)} \\
&\le C \sum_{k=1}^N \left\|  \1 \cdot (\bx_k-\wt\bx_k) \right\|_{L^2(\bM)} \\
& \le C \|\bM(\R)\|^{1/2} \sum_{k=1}^N \left\|  \bx_k-\wt\bx_k\right\|_{\cK}.
\end{align*}
Taking arbitrary $\e>0$ and requiring $\left\|  \bx_k-\wt\bx_k\right\|_{\cK} <\e N^{-1} 
\|\bM(\R)\|^{-1/2}$ we get that 
\[
\left\|  \sum_{k=1}^N f_k \bx_k - \sum_{k=1}^N f_k \wt\bx_k \right\|_{L^2(\bM)} <\e,
\]
which completes the proof.
\end{proof}

\subsection{Absolute continuity: indefinite integrals}\label{s-KurodaIndef}
For a moment, let us consider operator-valued measures that are absolutely continuous in a quite
strong way. Following Kuroda \cite{Kuroda1967} we call an operator-valued measure $\bM$ a \emph{weak
indefinite integral} if there exists a weakly (and so strongly) measurable operator-valued function
$W$ with values in $B_+(\cK)$ and a scalar-valued measure $\sigma$ (called \emph{a base} of $\bM$)
over a measure space $\cX$ such that
\begin{align*}
\bM(E)  = \int_{E} W(s) \sigma(\dd s)
\end{align*}
where the integral is understood in the weak sense. 

The main example of such measures is given by the following lemma. 

\begin{lem}\label{l:TraceClass_IndInt}
Let $\bM$ be an operator-valued measure such that $\bM(\fX)$ belongs to the trace class. 
Then $\bM$ is a weak indefinite integral. 
\end{lem}

This result is well known, see for example \cite{Kuroda1967}, Corollary to Lemma 2.3. For 
convenience of the reader we present a proof below; we should mention that it is different from a 
more abstract proof presented in \cite{Kuroda1967}. 

\begin{proof}[Proof of Lemma \ref{l:TraceClass_IndInt}]
Define a scalar measure $\mu$ by $\mu=\tr \bM$, or more precisely $\mu(E)=\tr \bM(E)$ for all 
measurable $E\subset \fX$; note that $\tr \bM(E)\le \tr \bM(\fX)$, so everything is well defined.   
If $\{\be_k\}_{k\in\N}$ is an orthonormal basis   in $\cK$, we can write 
\begin{align*}
\mu= \sum_k \mu_{\be_k},  
\qquad \mu_{\be_k}(\fdot) = \Jap{\bM(\fdot)\be_k, \be_k}, 
\end{align*}
so the measures $\mu_{\be_k}$ are absolutely continuous with respect to $\mu$, and we can write
$\mu_{\be_k}= w_{k,k} \mu$.  We  then conclude that the complex-valued measures
$\mu_{j,k}:=\mu_{\be_j, \be_k}$, where $\mu_{\be_j, \be_k}(\fdot) : = \Jap{\bM(\fdot)\be_j, \be_k}$
are also absolutely continuous with respect to $\mu$. Indeed, since for any measurable $E\subset\fX$
the operator $\bM(E)$ is non-negative, we see that
\begin{align*}
\begin{pmatrix}
\mu_{j,j}(E) & \mu_{j,k}(E) \\
\mu_{k,j}(E) & \mu_{k,k}(E)
\end{pmatrix}
\end{align*}
must be positive semidefinite. So $\mu_{j,k}(E) =0$ whenever $\mu(E)=0$. 
Thus $\mu_{j,k}=w_{j,k}\mu$. 

Therefore, if $P_n$ is an orthogonal projection onto 
\[
\cspn\{ \be_k:1\le k\le n\}
\]
then trivially $P_n\bM P_n$ is an indefinite integral, $P_n\bM (E)P_n = \int_E W_n(s) \mu(\dd s)$ 
for any measurable $E\subset \fX$.  The entries of $W_n(s)$ are given by 
\begin{align*}
\Jap{W_n(s) \be_j, \be_k} = \left\{
\begin{array}{ll}
w_{j,k}(s)&\text{ if }1\le j,k\le n,\\
0 &\text{ otherwise.}
\end{array}
\right.
\end{align*}
Clearly, $W_n(s) \nearrow$ as $n$ increases, and $ W_n(s)\to W(s)$ in norm topology (in fact in the 
trace class norm, although we do not need it here). 

Therefore, by the Monotone Convergence Theorem $\bM(E)  = \int_{E} W(s) \sigma(\dd s)$. 
\end{proof}

\subsubsection{Spaces $\fL^2(\bM)$ for weak indefinite integrals}\label{s:L^2 Wmu}

The weighted space $L^2(\bM)$ was defined in Section \ref{s-SpMeas} as the completion of the bounded
measurable \fdv functions in $L^2(\bM)$-norm. However, in the case when $\bM=W\sigma$ is a weak
indefinite integral, an equivalent, more transparent definition can be given.

For a Borel measurable function $f:\cX\to \cK$ we can define its $\fL^2(W\sigma)$-norm    as 
\begin{align}\label{e:L^2(W sigma)}
\|f\|\ci{\fL^2(W\sigma)}^2 := \int_{\cX} \Jap{W(s)f(s), f(s)}\ci{\cK}\sigma(\dd s); 
\end{align}
note that for Borel measurable $f$ the function $s\mapsto \Jap{W(s)f(s), f(s)}\ci{\cK}$  is always 
Borel measurable, so the integral exists for all such functions, although it could be infinite. 

Here we use $\fL^2(W\sigma)$ instead of $L^2(W\sigma)$  (or $L^2(\bM)$)  to emphasize that formally
we get a different space; we will show a bit later that the spaces coincide, so after that we will
be only using $L^2(W\sigma)$ or $L^2(\bM)$.

Define the weighted space $\fL^2(W\sigma)$ by taking all Borel measurable 
functions with $\|f\|\ci{\fL^2(W\sigma)}<\infty$, considering the quotient space and taking the 
completion.  Note that only considering functions is not sufficient, and we still need to have a 
completion: a weight $W$, $W(s)\equiv A$, where $A=A^*$ is a non-invertible operator with trivial 
kernel gives an example. 

However --- unlike the case of general operator-valued measures --- we can always say whether a 
function 
$f$ belongs to $\fL^2(W\sigma)$ by simply evaluating the integral \eqref{e:L^2(W sigma)}; the 
integrand is always a non-negative measurable function. 

Note that the above definition does not depend on the choice of the base $\sigma$, namely if 
$W\sigma = \wt W \wt \sigma$, then $\|f\|\ci{\fL^2(W\sigma)} = \|f\|\ci{\fL^2(\wt W \wt \sigma)}$. 

\begin{rem}
If $\bM$ is a matrix-valued measure ($\rank \bM(\fX)<\infty$), then $\bM$ is a weak indefinite 
integral $\bM= W\sigma$, and it is an easy exercise to show that in this case the space $L^2(\bM)$ 
is 
exactly the space of 
(equivalence classes of) functions  $f$ for which the norm \eqref{e:L^2(W sigma)} is finite, so 
in this case $L^2(W\sigma)=\fL^2(W\sigma)$.

As we just mentioned above, if $\bM$ is an indefinite integral we can always say whether a
measurable function belongs to $\fL^2(\bM)$, but to get the whole space $\fL^2(\bM)$ we need to take
the completion.

For general operator-valued measures we can verify that $f\in L^2(\bM)$ for \fdv functions, because 
in this case the question can be reduced to the question about matrix-valued measures. However, for 
general measurable functions it is not easy to say whether $f$ is in $L^2(\bM)$ or not. We know 
that there exists a sequence of simple measurable functions $f_n$ that converges to $f$ pointwise, 
but we cannot just say that $\|f\|\ci{L^2(\bM)} = \lim_{n\to \infty}\|f_n \|\ci{L^2(\bM)} $; taking 
the limit requires justification which is generally not available. 
\end{rem}

\subsubsection{Comparing $L^2(\bM)$ with $\fL^2(W\sigma)$}

A priori, for $\bM=W\sigma$  the space $L^2(W\sigma)$ defined in Section \ref{s-SpMeas} could
potentially differ from $\fL^2(W\sigma)$, because in the definition of $\fL^2(W\sigma)$ we take a
completion of a bigger set. But as one could expect, $L^2(W\sigma) = \fL^2(W\sigma)$. Let us show
that.

First, it is easy to see (using identity \eqref{e:inner prod}, for example) that  the norms
$\|\fdot\|\ci{L^2(W\sigma)} $ and $\|\fdot\|\ci{\fL^2(W\sigma)}$  coincide on  bounded measurable
\fdv functions.

Let us agree that we are using the standard construction of the completion via equivalence classes
of Cauchy sequences. Then the equivalence of the 2 definitions (i.e.~equality
$L^2(\bM)=L^2(W\sigma)=\fL^2(W\sigma)$) follows from Lemma \ref{l:L^2(M) = L^2(W)} below, which is
essentially Lemma 3.5 from \cite{Kuroda1967}.

\begin{lem}\label{l:L^2(M) = L^2(W)}
Let $\bM=W\sigma$ be a weak indefinite integral, and let $f:\fX\to \cK$ be a measurable function 
such that
\begin{align*}
\|f\|\ci{\fL^2(W\sigma)}^2 = \int_{\fX} \Jap{W(s) f(s),f(s)} \sigma(\dd s) <\infty . 
\end{align*}
Then there exists a sequence $\{f_n\}_{n\in\N}$ of bounded measurable \fdv functions such that 
$\|f-f_n\|\ci{\fL^2(W\sigma)}\to0$ as $n\to\infty$. 
\end{lem}

Let us first explain why Lemma \ref{l:L^2(M) = L^2(W)} implies the equality
$L^2(W\sigma)=\fL^2(W\sigma).$ Since $\|f-f_n\|\ci{\fL^2(W\sigma)}\to0$ as $n\to\infty$, the
sequence $\{f_n\}_{n\in\N}$ is a Cauchy sequence in $\fL^2(W\sigma)$, and therefore in
$L^2(W\sigma)$ (both norms coincide on measurable bounded \fdv functions). Identifying $f$ with an
equivalence class of Cauchy sequences containing $\{f_n\}_{n\in\N}$, we get that $f$ is canonically
identified with an element of $L^2(W\sigma)$.

Since double completion coincides with completion, we conclude that $L^2(W\sigma)=\fL^2(W\sigma)$. 

From now on we will use the standard notation $L^2(\bM)=L^2(W\sigma)$ when $\bM=W\sigma$.

\begin{proof}[Proof of Lemma \ref{l:L^2(M) = L^2(W)}]
Assume that $\varepsilon >0$ is given, and denote 
\[
E_\lambda:=\{ s\,:\, \max(\|W(s)\|, \|f(s)\|)>\lambda\}.
\]
We note that $\sigma(E_\lambda)\to 0$ as $\lambda\to\infty$, so we may choose $\lambda >0$ so large 
that 
\begin{align}\label{e:part1}
\int_{E_\lambda} \Jap{W(s)f(s), f(s)}\ci{\cK}\sigma(\dd s)\; <\; \varepsilon/2.
\end{align}
Let $\be_1,\be_2,\ldots $ be  an orthonormal basis of $\cK$ and write
$f(s)=\sum_{j=1}^\infty\f_j(s)\be_j.$ By denoting $f_n(s):= \sum_{j=1}^n\f_j(s)\be_j$, we may next
use the dominated convergence theorem to pick large enough $n$ so that 
\begin{align}\label{e:part2}
 \int_{X\setminus E_\lambda} \Jap{W(s)(f(s)-f_n(s), f(s)-f_n(s)}\ci{\cK}\sigma(\dd s)\; <\; 
 \varepsilon/2.
\end{align}
Put together, \eqref{e:part1} and \eqref{e:part2}  verify that 
$\|f-\widetilde f_n\|\ci{L^2(\mathbf{M})}< \varepsilon,$ where 
$\widetilde f_n(s):= \sum_{j=1}^n\widetilde\f_j(s)\be_j$ with 
$\widetilde\f_j(s):=\1\ci{X\setminus E_\lambda}(s)\f_j(s)$;
note that by the construction the functions $\wt f_j$ are bounded.
\end{proof}

%%%%%%%%%%%%%%%%%%%%%%%%%%%%%
\section{Preliminaries in Spectral Theory}

We introduce and discuss the spectral theory toolkit, which we use in this paper. Some of the
results are known to the experts. But striving for a self-contained presentation tailored for our
purposes, we include them here. Some of the less known  facts recalled below are von Neumann direct
integral representation, spectral invariants, relevant cyclicity facts, compressed spectral
measures, and vectors and sets of maximal spectral type.

%%%%%%%%%%%%%%%%%%%%%%%%%%%%%%%%%%%%%%%%
%%%%%%%%%%%%%%%%%%%%%%%%%%%%%%%%%%%%%%%%
%%%%%%%%%%%%%%%%%%%%%%%%%%%%%%%%%%%%%%%%
\subsection{Spectral Theorems}\label{s:Sp Thm}
Let us recall the classical spectral theorem. 
\begin{theo}
Let $A$ be a self-adjoint \tup(possibly unbounded\tup) operator on a Hilbert space $\cH$. Then there
exists a unique operator-valued measure $\bE$ \tup(supported on $\sigma(A)$\tup) with values in
$B(\cH)$ such that for any $\bx\in\cH$ and for any bounded continuous function $\f$ on $\R$
\begin{align}\label{e:SpMeas E 01}
\Jap{\f(A) \bx, \bx} =\int_\R \f(s)\Jap{\bE(\dd s)\bx, \bx} =: \int_\R \f(s)\mu_\bx(\dd s) . 
\end{align}
Moreover, the values of $\bE$ are orthogonal projections, and $\bE$ is a probability measure, i.e.\ 
$\bE(\R)=\bI$. 
\end{theo}

\begin{rem*}
It is sufficient to check \eqref{e:SpMeas E 01} on an appropriate subclass of functions, for 
example on rational functions $\f_a$, $\f_a(s)=1/(s-a)$, $a\in\C\setminus\R$; very often this class 
is used in the statement of the spectral theorem.  
\end{rem*}

Applying polarization identity to \eqref{e:SpMeas E 01} we can see that  for any bounded 
continuous function $\f$ on $\R$ and for any $\bx, \by\in\cH$ 
\begin{align}\label{e:SpMeas E 02}
\Jap{\f(A) \bx, \by} =\int_\R \f(s)\Jap{\bE(\dd s)\bx, \by} =: \int_\R \f(s)\mu_{\bx,\by}(\dd s) ,  
\end{align}
where 
\begin{align*}
\mu_{\bx,\by} = \frac{1}{4} \sum_{\alpha:\alpha^4=1}\alpha \mu_{\bx+\alpha\by}.
\end{align*}

\begin{rem*}
Note that the right hand side of \eqref{e:SpMeas E 02} is well defined for any bounded Borel 
measurable function $\f$, so this identity can be used to define $\f(A)$ for any such function. 

It is also well-known that the maps $\f\mapsto \f(A)$ is a homomorphism between algebras of 
functions and operators. 
\end{rem*}

If the operator $A$ is  \emph{cyclic}, i.e.~if it has a cyclic vector, see definitions in Section 
\ref{s:cycl}, the statement of the Spectral Theorem can be simplified. In this case, the operator 
$A$ is unitarily equivalent to a simple model.  

\begin{theo}\label{t:SpTh cycl}
If $\bx$ is a cyclic vector for $A$, then the operator $A$ is unitarily equivalent to the 
multiplication $\cM_{\id}$  by the independent variable, 
\begin{align*}
\cM_{\id} f(s) = s f(s), \qquad \forall s\in \R
\end{align*}
in $L^2(\mu_\bx)$, where $\mu_\bx(\fdot)= \Jap{\bE(\fdot)\bx,\bx}_\cH$; the operator 
$U:L^2(\mu_\bx)\to\cH$, defined on the dense subset $C\ti{c}(\R)\subset 
L^2(\mu_\bx) $ by
\begin{align}\label{e:operU}
U f = f(A) \bx, \qquad f\in C\ti{c}(\R),
\end{align}
is a unitary operator such that 
\begin{align}\label{e:UE-01}
A = U\cM_{\id} U^* .
\end{align}
\end{theo}
\begin{rem}
A unitary operator $U$ satisfying \eqref{e:UE-01} is not unique, but the canonical operator $U$ 
given by \eqref{e:operU} is uniquely defined by the condition $U\1=\bx$. 
\end{rem}

%%%%%%%%%%%%%%%%%%%%%%%%%%%%%
\subsection{Von Neumann Direct Integral and Spectral Invariants}\label{s-proj}
%%%%%%%%%%%%%%%%%%%%%%%%%%%%%
Let us present maybe a less known version (interpretation) of the Spectral Theorem, which states
that a self-adjoint operator in a separable Hilbert space is unitarily equivalent to the
multiplication operator $\cM_{\id}$ by the independent variable,  
$\left(\cM_{\id} f\right) (s) = sf(s)$ (compare with Theorem \ref{t:SpTh cycl}), in the von Neumann 
direct integral
\begin{align}\label{e:VN_int}
\cH= \int_\R \oplus \, \cH(s) \nu(\dd s).
\end{align}
Let us recall the simplest construction of the direct integral \eqref{e:VN_int}. Suppose we have a 
$\sigma$-finite Borel measure $\nu$ on $\R$ and a Borel measurable function 
$N:\R\to\Z_+\cup\{\infty\}$ (the dimension function). 

For $s\in\R$ define $\cH(s)\subset \ell^2 =\ell^2(\N)$ as 
\begin{align*}
\cH(s) := \cspn \{ e_k: 1\le k \le N(s)    \};
\end{align*}
here $e_k$, $k\in \N$ is the standard basis in $\ell^2$. If $N(s)=0$ we put $\cH(s):=\{\bO\}.$

Then the von Neumann direct integral \eqref{e:VN_int} is defined as 
\begin{align}\label{e-cHDef}
\cH :=\{f\in L^2(\nu;\ell^2): f(t)\in \cH(t ) \ \nu \text{-a.e.} \}.
\end{align}

For a measure $\nu$ its \emph{type} is the set of all measures mutually absolutely continuous 
with respect to $\nu$.

\begin{theo}[{see e.g.~\cite[Ch.~7, Theorem 5.2]{BirmanSol-book_1987}}]\label{t:SpectrInv}

The type of the  measure $\nu$ and the dimension function $N$ \tup(up to $\nu$-a.e.\tup) completely
define a self-adjoint operator up to unitary equivalence.
	
Namely, two self-adjoint operators \tup(represented in the von Neumann direct integrals with
measures $\nu$ and $\nu_1$, and the dimension functions $N$ and $N_1$ respectively\tup) are
unitarily equivalent if and only if the measures $\nu$ and $\nu_1$ are mutually absolutely
continuous and $N(t) = N_1(t)$ $\nu$-a.e.
\end{theo}

Note that in the direct integral representation if $\bE$ is the projection-valued measure  for 
$\cM\ti{id}$, then the operator $\bE(E)$ is just the multiplication by $\1\ci{E}$ (i.e.~the 
characteristic function of $E$).

\subsection{Cyclicity}\label{s:cycl} Let us recall some definitions. Let $A$ be a  
self-adjoint operator. For a subset $S\subset \cH$ let   $[S]\ci{A}$ be the minimal closed 
$A$-invariant subspace containing $S$. 

It is easy to see from the spectral theorem that 
\begin{align}\label{e-CycSpan}
[S]\ci{A} & =  \cspn \{ (A-z\bI)^{-1} f: z\in\C\setminus \R, \, f\in S\}  \\ \notag
& =  \cspn \{ \f(A) f: \f\in C\ti{c}(\R), \, f\in S\} \\\notag
& =  \cspn \{ \f(A) f: \f\in C_0(\R), \, f\in S\} \\\notag
& =  \cspn \{ \f(A) f: \f \text{ bounded Borel measurable function on }\R, \, f\in S\}. 
\end{align}
Here $C\ti{c}(\R)$ are compactly supported continuous functions, and $C_0(\R)$ refers to continuous 
functions that approach zero at $\pm\infty$.

If $A$ is bounded one can also show that 
\begin{align*}
[S]\ci{A} = \cspn\{A^n f: n \text{ non-negative integer}  \}. 
\end{align*}

A set $S\subset \cH$ is called \emph{cyclic} for $A=A^*$ if $[S]\ci{A}=\cH$. 

If a cyclic set $S$ consists of one vector $v$, we call $v$ a \emph{cyclic vector}.

Trivially a set $S$ is cyclic if and only if $\cspn S $ is cyclic. For  a separable Hilbert space  
$\cH$ one can always find a countable set $\wt S\subset S$ dense in $S$, or just a countable set 
$\wt S$ such that $\cspn \wt S =\cspn S$, so cyclicity in a separable Hilbert space can always be 
verified on a dense set.

Let $(\fX, \cA)$ be a measurable space, and let $\cH$, $\cK$ be  separable Hilbert spaces. 
\begin{lem}\label{l:meas proj}
Let $F:\fX\to B(\cH, \cK)$ be a measurable operator-valued function. Then the function 
$s\mapsto P_{\cRan F(s)}$, where $P_{\cRan F(s)}$ is the orthogonal projection onto $\cRan F(s)$, 
is measurable. 
\end{lem}
\begin{proof}
Note that
\begin{align*}
\cRan F(s) = \cRan \left(F(s) F(s)^*\right). 
\end{align*}
Denoting $\Phi(s):= F(s) F(s)^*$, and $\f:=\1_{(0,\infty)}$, we can see that 
\begin{align*}
P_{\cRan F(s)}= \f(\Phi(s)). 
\end{align*}
The function $\Phi$ is weakly (and so strongly) measurable, and therefore so is the function 
$F^*$ ($s\mapsto F(s)^*$). Therefore the product $\Phi=FF^*$ is also strongly measurable, see 
Lemma \ref{l:prod meas funct} in the Appendix \ref{ap:measurability} below. Therefore by Lemma 
\ref{lem:measurable} in the same appendix, the function $\f(\Phi)$ is measurable.
\end{proof}

\begin{lem}\label{lem:a}
Let $f_k:\fX\to \cK$, $k\in \N$  be $\cA$-measurable functions, and let 
$\cK(s):=\cspn\{f_k(s):k\in\N\}$. Then 
\begin{enumerate}
\item The function $s\mapsto P\ci{\cK(s)}$, where $P\ci{\cK(s)}$ is the orthogonal projection onto 
$\cK(s)$ is measurable. 
\item There exists a measurable function $g:\fX\to\cK$ such that 
\begin{align*}
&g(s)\perp\cK(s) \qquad \forall s\in \fX;\\
&\|g(s)\|=1 \qquad \text{whenever } \cK(t)\ne \cK. 
\end{align*}

\end{enumerate}
\end{lem}

\begin{proof}
Renormalizing, if necessary, we can assume without loss of generality that for all $s\in\fX$ we 
have $\|f_k(s)\|\le 2^{-k}$. 

Define a function $F:\fX\to B(\ell^2;\cK)$ by 
\begin{align}\label{e:F(s)}
F(s)\bx=\sum_{k\in\N} x_k f_k(s), \qquad \bx = (x_1, x_2, \ldots, x_n, \ldots)\in \ell^2; 
\end{align}
the normalization condition $\|f_k(s)\|\le 2^{-k}$ guarantees convergence in norm on $\cK$. 

Partial sums in \eqref{e:F(s)} are clearly measurable functions of $s$, so the function $s\mapsto
F(s)\bx$ is measurable for all $\bx\in\cK$. Thus $F$ is measurable. Note that
\begin{align*}
\cK(s) =\cRan F(s) .   
\end{align*}
Applying the above Lemma \ref{l:meas proj} we get Statement \cond1.

To prove Statement \cond2, the complementary projection $P\ci{\cK(s)^\perp}=\bI-P\ci{\cK(s)}$; 
clearly the function $s\mapsto P\ci{\cK(s)^\perp}$ is measurable.  

Fix an orthonormal basis  $\{\be_n\}_{n\ge 1}$ in $\cK$ and define the sets 
$\fX_n\subset \fX$, 
\begin{align*}
\fX_n &:= \{ s\in \fX: P\ci{\cK(s)^\perp}\be_n \ne\bO \text{ and }P\ci{\cK(s)^\perp}\be_k =\bO 
\text{ for }k=1, 2, \ldots n-1\}  \\
&\phantom{:} = \{ s\in \fX:  P\ci{\cK(s)^\perp}\be_n \ne\bO\}\cap\bigcap_{k=1}^{n-1}\{ s\in \fX: 
P\ci{\cK(s)^\perp}\be_k =\bO \}  .
\end{align*}
Notice that for a measurable function $f:\fX\to \cK$ the function $s\mapsto \|f(s)\|$ is trivially 
measurable. Therefore the sets $\fX_k$ are measurable as finite intersections of measurable sets. 
Note also that the sets $\fX_n$ are disjoint, and 
\begin{align*}
\{s\in\fX : \cK(s)\ne \cK\} = \bigcup_{n\ge 1}\fX_k, 
\end{align*}
and $P\ci{\cK(s)^\perp}\be_n \ne\bO$ for all $s\in\fX_n$. Therefore defining 
\begin{align*}
\wt g := \sum_{n\ge1} \1\ci{\fX_n}  P\ci{\cK(s)^\perp}\be_n ,
\end{align*}
and then normalizing it on the set $\{s\in\fX : \cK(s)\ne \cK\}$, where $\wt g$ does not vanish, we 
get the desired function $g$. 
\end{proof}

\begin{lem}\label{l-cyclicset}
Let $A$ be the multiplication $\cM\ti{id}$ by the independent variable $s$ in the von Neumann
integral $\cH$ given by \eqref{e:VN_int}.  Then a countable set $S =\{f_k:k\in\N\}\subset\cH$ is
cyclic for $A$ if and only if
\begin{align*}
\cspn\{ f_k(s) : k\in\N  \} =\cH(s) \qquad \nu\text{-a.e.}
\end{align*}
\end{lem}

\begin{proof}
Let $S$ not be cyclic. By \eqref{e-CycSpan} this means the existence of a function $ f\in\cH$,
$f\ne\bO$ which is orthogonal to functions $\f f_k$, for all bounded (scalar) Borel measurable
functions $\varphi$, i.e.~that
\begin{align*}
\int_\R \f(s) \Jap{f_k(s), f(s)}\ci{\cH(s)} \nu(\dd s)=0 
\end{align*}
for all bounded  Borel measurable functions $\varphi$.

This implies that for all $k$ 
\begin{align}\label{e:orthog_03}
\Jap{f_k(s), f(s)}\ci{\cH(s)}=0 \qquad \nu\text{-a.e.~on }\R. 
\end{align}
Let
\[
E:=\{ s\in\R: f(s) \ne 0 \};
\]
since $f\ne\bO$ we conclude that $\nu(E)>0$. Orthogonality condition \eqref{e:orthog_03} then
implies that
\begin{align}\label{e:spn ne H}
\cspn \{ f_k(s): k\in\N   \} \ne \cH(s) \qquad \nu\text{-a.e.~on }E .  
\end{align}

To prove the other direction, assume that there exists a set $E$, $\nu(E)>0$ such that 
\eqref{e:spn ne H} holds. 

Then for at least one of the sets  $E_k:=\{s:\dim\cH(s)=k\}$, $k\in \N\cup\{\infty\}$ we have 
\begin{align*}
\nu(E\cap E_k)>0.
\end{align*}
Applying the previous Lemma \ref{lem:a} to functions $\1\ci{E\cap E_k}f_n$, $n=1, 2, \ldots$ with 
$\cK=\cH(s)$, $s\in E_k$, we get a function $g= \1\ci{E\cap E_k} g \in \cH$ such that $\nu$-a.e.~on 
$\R$
\begin{align*}
g(s)\perp f_n(s), \qquad n=1, 2, \ldots
\end{align*}
which immediately implies that the collection $S$ is not cyclic. 
\end{proof}

\subsection{Factorization of perturbations} We will need the following lemma. 
\begin{lem}\label{l:factorization}
Let $K=K^*\in B(\cH)$. Then $K$ can be represented as $K=\bB\Gamma\bB^*$, where $\bB:\cK\to\cH$, 
$\Ker\bB=\{\bO\}$ \tup($\cK$ is an auxiliary Hilbert space\tup), and $\Gamma =\Gamma^*\in B(\cK)$ 
is invertible. Note that $\Gamma\ge\bO$ if and only if $K\ge\bO$. 

Moreover, if $K$ belongs to the trace class, then $\bB$ is Hilbert--Schmidt. 
\end{lem}
\begin{proof}
Put $\cK:=\cRan K$ and define  $\bB:\cK\to\cH$ as the restriction of $| K|^{1/2}$ onto $\cK$. To
compute $\Gamma$ consider the (unique) polar decomposition $K=U|K|$, where $U:\cH\to\cH$ is a
partial isometry with $\Ker U=\Ker K$. It is easy to see from the spectral theorem that
$U=\bE((0,\infty)) -\bE((-\infty,0))$, and that $U$ commutes with $|K|$. Therefore, if we define
$\Gamma$ to be a restriction of $U$ to $\cRan K=(\Ker K)^\perp$, one can easily see that 
$K= U|K|=|K|^{1/2} U |K|^{1/2} =\bB\Gamma\bB^*$, which gives us the desired factorization.

Note that in  this case $\Gamma$ is a \emph{sign} operator, i.e.\ $\Gamma=\Gamma^*=\Gamma^{-1}$. 

Replacing $\bB$ and $\Gamma$ by $\bB T^{-1}$ and $T \Gamma T^*$, where $T:\cK\to \cK'$ is 
invertible we get all possible factorizations. 

Finally, it is easy to see that $K$ belongs to the trace class if and only if $\bB$ is 
Hilbert--Schmidt. 
\end{proof}

%%%%%%%%%%%%%%%%%%%%%%%%%%%%%%%
\subsection{Families of perturbations and compressed spectral measures} \label{s:Fam Pert}
Many of the results in this paper deal with families of perturbations 
\begin{align}\label{e-AGamma}
A\ci\Gamma = A+ \bB\Gamma  \bB^*, 
\end{align}
where  $A$ is a self-adjoint   operator on a separable Hilbert space $\cH$,  $\bB:\cK\to\cH$ is a 
fixed bounded operator ($\cK$ is an auxiliary Hilbert space) and $\Gamma$ belongs to a suitable 
family of self-adjoint operators  on $\cK$.

Let   $\bE$ and $\bE\ci\Gamma$ be the projection-valued spectral measures for $A$ and $A\ci\Gamma$
respectively. Define the compressed spectral measures $\bM$ and $\bM^\Gamma$ with values in
$B_+(\cK)$ as $\bM=\bB^*\bE\bB$, i.e.\
\begin{align}
\label{dM 01}
\bM(E) = \bB^* \bE(E) \bB\qquad \text{for all Borel } E\subset \R. 
\end{align}
(it is the same equation as \eqref{e:comp-01} with $\bC$ from \eqref{e:comp-01} denoted $\bB^*$), 
and similarly  $\bM^\Gamma=\bB^*\bE\ci\Gamma\bB$.%
\footnote{For typographical reasons, that will be clear later, we use $\Gamma$ as a superscript 
here.}

Recall that the Cauchy transform $\cC\mu$ of a (scalar) measure $\mu $ on $R$ is defined as 
\begin{align*}
\cC\mu(z):= \int_\R \frac{\mu(\dd t)}{t-z}, \qquad z\in\C\setminus \R. 
\end{align*}
Using \eqref{e:SpMeas E 01} we can compute the  Cauchy transforms  of the measures $\bM$ and
$\bM^\Gamma$,
\begin{alignat}{3}\label{d-M}
F(z)&:= \cC\bM(z) := \int_\R \frac{ \bM(\dd t)}{t-z}= {\bf B}^* (A-z\OID)^{-1} {\bf B}, \qquad 
 &&z\in\C\setminus \R , \\  \label{e-DefFGamma} F\ci\Gamma(z)
&:= \cC\bM^\Gamma(z) := \int_\R \frac{ \bM^\Gamma(\dd t)}{t-z} =
{\bB}^* (A\ci\Gamma-z\OID)^{-1} {\bB}, 
\qquad&&z\in \C\setminus\R.
\end{alignat}

Since for $z\in\C_+$ and $t\in\R$
\[
\im \frac{1}{t-z}  = \frac{\im z}{|t-z|^2} >0,
\]
we conclude that for $z\in\C_+$
\begin{align}\label{e-herglotzF}
\im F(z) =  \int_\R \frac{\im z}{|t-z|^2}\bM(\dd t)\ge \bO.
\end{align}

We will need the following simple fact: 

\begin{lem}\label{l:cycl triv kernel}
Let $A=A^*$, $\bB\in B(\cK;\cH)$, and let $F=\cC\bM$ be defined by \eqref{d-M}. 
If $\Ker \bB=\{\bO\}$, then $\Ker F(z) = \{\bO\}$  for all $z\in\C\setminus \R$.
\end{lem}
We leave the proof as an exercise for the reader.

%%%%%%%%%%%%%%%%%%%%%%%%%%%%%%%%%%%%%%%%%%%%%%%%%%
\subsubsection{Families of perturbations and cyclicity}\label{s-perturbations}
%%%%%%%%%%%%%%%%%%%%%%%%%%%%%%%%%%%%%%%%%%%%%%%%%%

\begin{lem}\label{l-cycAGamma}
If $\Ran\bB$ is cyclic for $A$ then it is  cyclic for all $A\ci{\Gamma}$ from \eqref{e-AGamma}.
\end{lem}

\begin{proof}
For the readers' convenience, we recall the argument from \cite[Lemma 2.5]{JST}, where the lemma 
was proven for finite rank perturbations.

Denote by $R(z)$ and $R_\Gamma(z)$ the resolvents, 
\[
	R (z)= (A-z \OID)^{-1} 
	\qquad\text{and}\qquad
R\ci\Gamma(z) = (A\ci\Gamma-z \OID)^{-1}.
\]

The subspace $\Ran\bB$ is cyclic for $A$ if and only if the (finite) linear combinations of $R(z_k)
b_k$, $z_k\in\C\setminus\R$, $b_k \in \Ran \bB$  are dense in $\cH$, and similarly for the cylicity
for $A\ci\Gamma$.

Therefore to prove that  $\Ran\bB$ is cyclic for $A\ci\Gamma$  it suffices to show that for each
$z\in \C\setminus\R$, $b\in \Ran \bB$ the vector $R(z) b$ belongs to $R\ci\Gamma(z) \Ran\bB$.

To see this, we apply the resolvent identity
\begin{align}\label{e-ResId}
R(z) = R\ciG(z) + R\ciG(z) {\bf B}\Gamma {\bf B}^* R(z)
\end{align}
to a vector $b\in\Ran\bB$ and notice that on the right-hand side we have 
$[\OID + {\bf B}\Gamma {\bf B}^* R(z)]b\in \Ran\bB$.  
\end{proof}

%%%%%%%%%%%%%%%%%%%%%%%%%%%%%
\subsubsection{Relations between Cauchy transforms of compressed spectral 
measures}\label{s-FamilyBorel}
%%%%%%%%%%%%%%%%%%%%%%%%%%%%%

We will need the classical Aronszajn--Krein type relationship between the Cauchy transforms 
$F:=\cC\bM$ and $F\ci\Gamma :=\cC\bM^\Gamma$.

The following lemma is well-known to experts, see e.g.~\cite{katokuroda, Kuroda1967, Yafaev1992}. 
We provide complete proofs for the convenience of the reader.

\begin{lem}\label{l-AK}
For all $z\in \C\setminus\R$ and all self-adjoint bounded $\Gamma$ the operators $\bI+F(z)\Gamma$, 
$\bI + \Gamma F(z)$ are invertible. Moreover, we have 
\begin{align}
\label{F_Gamma}
F\ciG (z)
&=
F(z)(\OID + \Gamma F(z))^{-1}
=
(\OID +  F(z) \Gamma)^{-1} F(z),  \text{ and}
\\
\label{e-AKIm2}
\im F\ciG(z)
&
=
(\bI+F(z)^*{\Gamma})^{-1}
\left(\im F(z)\right)
(\bI+{\Gamma} F(z))^{-1}  \\ \notag
&=
(\bI+F(z){\Gamma})^{-1}
\left(\im F(z)\right)
(\bI+{\Gamma} F(z)^*)^{-1}.
\end{align}
\end{lem}

\begin{proof}
We apply $\bB^*$ to the left and $\bB$ to the right of resolvent identity \eqref{e-ResId}. By the
identities \eqref{d-M}  and \eqref{e-DefFGamma} we have 
$F(z)= \bB^* R (z) \bB$ and $F\ci\Gamma(z) = \bB^* R\ci\Gamma(z) \bB$, so we get for all 
$z\in\C\setminus\R$
\begin{align}\label{e-FGamma}
F(z) = F\ci\Gamma(z) + F\ci\Gamma(z)\Gamma F(z). 
\end{align}
With this,  we calculate, skipping the argument $z$ for the brevity of notation
\[
(\OID -  F\ci\Gamma \Gamma)(\OID +  F \Gamma) = 
\OID - F\ci\Gamma \Gamma + F \Gamma - F\ci\Gamma \Gamma F \Gamma
\equiv
\OID
\]
to conclude that $\OID +  F(z) \Gamma$ is left invertible for all $z\in \C\setminus\R$.

In analogy, the resolvent identity
$
R = R\ci\Gamma + R \bB\Gamma\bB^* R\ci\Gamma
$ 
can be used to prove that $(\OID +  F \Gamma) (\OID -  F\ci\Gamma \Gamma) = \OID$ on 
$\C\setminus\R$.

The first equality of \eqref{F_Gamma} now follows from \eqref{e-FGamma}, and the second equality of 
\eqref{F_Gamma} is proven in analogy, i.e.~by showing that $\OID -  \Gamma F\ci\Gamma (z) $ is the 
inverse of $\OID +  \Gamma F(z)$ for all $z\in\C\setminus\R$.

Using the first identity from \eqref{F_Gamma}, and skipping again the argument $z$ for the brevity 
of notation, we obtain 
\begin{align*}
\im F\ci\Gamma & = ( F\ci\Gamma - F\ci\Gamma^*  )/(2i)
= ( F(\bI + \Gamma F)^{-1}  - (\bI + F^*\Gamma)^{-1}F^* )/(2i) \\
&=  (\bI + F^*\Gamma)^{-1}[(\bI + F^*\Gamma)F - F^* (\bI + \Gamma F) ] (\bI + \Gamma F)^{-1}/(2i)\\
& = (\bI + F^*\Gamma)^{-1} \left(\im F\right) (\bI + \Gamma F)^{-1},  
\end{align*}	
which is the first identity in \eqref{e-AKIm2}. 

The second identity in \eqref{e-AKIm2} is obtained similarly from the first identity in 
\eqref{F_Gamma}. 
\end{proof}

\subsection{A Generalized Spectral Theorem}

If $\Ran \bB$ is cyclic for $A$ then the operator $A$ is unitarily equivalent to the multiplication
operator $\cM_{\id}$ in $L^2(\bM)$; recall that the definition of the weighted space $L^2(\bM)$ was
established in Section \ref{s-SpMeas}.

\begin{theo}[Generalized spectral theorem]\label{t:SpThm-OVM}
Let $A$ be a self-adjoint operator in a separable Hilbert space $\cH$ and let $\bB:\cK\to\cH$ be a 
bounded operator with $\Ran \bB$ being cyclic for $A$.  Let $\bE:=\bE\ci{A}$ be the 
projection-valued spectral measure for $A$, and let $\bM := \bB^*\bE \bB$ be the compressed 
spectral measure.  

Then $A$ is unitarily equivalent to the multiplication  $\cM_{\id}$  by the independent variable in
the weighted space $L^2(\bM)$,  $\left(\cM_{\id} f\right)(s) = s f(s)$.

More precisely, the operator $U:L^2(\bM) \to \cH$ defined on a total in $L^2(\bM)$ set of functions 
$\f \be$ \tup(where $\f$ is a bounded Borel measurable scalar function, $\be\in\cK$\tup) by 
\begin{align}\label{e:U fe}
U(\f\be) = \f(A) \bB\be 
\end{align}
extends to a unitary operator such that 
\begin{align}\label{e:UM=AU}
U\cM_{\id} = A U. 
\end{align}
\end{theo}

\begin{proof} Let $f=\sum_{k=1}^n \f_k \be_k$, where $\be_k\in\cK$ and $\f_k$ are bounded Borel 
measurable functions. Then $Uf=\sum_{k=1}^n \f_k(A) \bB \be_k$ and by the Spectral Theorem, see 
\eqref{e:SpMeas E 01}
\begin{align*}
\|Uf\|_\cH^2 & = \sum_{j,k=1}^n\Jap{\f_j(A)^*\f_k(A) \bB\be_k, \bB\be_j}_\cH \\
& = \sum_{j,k=1}^n \int_\R \Jap{ \overline{\f_j(s)}\f_k(s) \bE(\dd s) \bB\be_k, \bB\be_j } \\
& = \sum_{j,k=1}^n \int_\R \Jap{ \overline{\f_j(s)}\f_k(s) \bM(\dd s) \be_k, \be_j }
=\|f\|\ci{L^2(\bM)}^2 . 
\end{align*}
Thus, $U$ is an isometry between dense linear subsets of $L^2(\bM)$ and $\cH$ respectively, so it 
extends by continuity to a unitary operator. 

To prove  \eqref{e:UM=AU} let us check it on functions $\f\be$, where $\be\in\cK$ and $\f$ is a 
bounded Borel measurable function. Since the spectral measure $E$ is supported on $\sigma(A)$, we 
can assume without loss of generality that $\f\equiv 0$ outside of $\sigma(A)$. Define 
$\psi(s):=s\f(s)$. Then trivially $\cM_{\id}\f\be=\psi\be$, and therefore 
\begin{align*}
U \cM_{\id}\f\be = U\psi\be = \psi(A)\bB\be =A\f(A)\bB\be =A U\f\be;
\end{align*}
the last identity follows from \eqref{e:U fe}. 
\end{proof}

This theorem might look silly and too complicated, but it will be one the main tools we use to 
prove Theorems \ref{t:GA-D_trace} and \ref{t:GA-D_02}. 

\begin{rem}
It is easy to see that %%
\begin{align}\label{e:Ue}
U\be = \1\be;
\end{align}
here we put $\1$ on the right-hand side to emphasize that this is a function identically equal to 
$\be$ on $\R$. 

Note that while a unitary operator satisfying \eqref{e:UM=AU} is not unique, the above condition 
\eqref{e:Ue} defines it uniquely.  We leave the details as an exercise for the reader. 
\end{rem}

\subsection{Vectors and measures of maximal spectral type}\label{s:MaxType}  

Let $A$ be a self-adjoint operator in a separable Hilbert space $\cH$. We say that a vector
$\bx\in\cH$ has \emph{maximal spectral type} if for any $\by\in\cH$ the scalar spectral measure
$\mu_\by$ is absolutely continuous with respect to the spectral measure $\mu_\bx$. (Recall that the
scalar spectral measure is defined in Section \ref{s-SpMeas}.) We call the corresponding measure
$\mu_\bx$ a (scalar) spectral measure of maximal type.

We can also state this property in terms of the projection-valued spectral measure $\bE$ of $A$. 
Namely, a vector $\bx$ has \emph{maximal spectral type} if and only if  $\bE(E)\bx\ne\bO$ for any 
Borel set $E\subset \R$ such that $\bE(E) \ne\bO$.

\begin{lem}\label{l-maxtype1}
Let $A$ be the multiplication $\cM\ti{id}$ by the independent variable $s$ in the von Neumann direct
integral $\cH$ given by \eqref{e:VN_int}. A vector $f\in \cH$ has maximal spectral type if and only
if $f(s)\ne\bO$ $\nu$-a.e.
\end{lem}

\begin{proof}
In the direct integral representation of $A,$ the projection-valued spectral measure $\bE(E)$ 
equals multiplication by the characteristic function $\1\ci E$ for every Borel set $E$. So, the 
lemma follows immediately from the definition of a vector having maximal spectral type.
\end{proof}

\begin{rem}
If $\mu$ is a (finite) measure mutually absolutely continuous with a spectral measure of maximal 
type, then it is also a spectral measure of maximal type, meaning that there exists a  vector $\bx$ 
of maximal spectral type such that $\mu=\mu_\bx$. 

Indeed, if we consider the representation in which $A$ is the multiplication $\cM_{\id}$ in the von
Neumann direct integral $\cH$ given by \eqref{e:VN_int}, the measure $\nu$ is the spectral measure
corresponding to the vector $\1 \be_1 \in \cH$. Since $\mu$ is mutually absolutely continuous with
respect to $\nu$, it can be represented as $\mu = w\nu$, where $w>0$ $\nu$-a.e. Taking a scalar
measurable function  $f$, $|f|^2=w$ (for example, $f:=w^{1/2}$), we see that $\mu$ is the spectral
measure corresponding to the vector $f \be_1\in\cH$.
\end{rem}

\begin{cor}\label{c:max type}
Let $A=A^*$, and let $\bB:\cK\to\cH$ be such that $\Ran \bB$ is cyclic for $A$.  Let $\bE=\bE\ci{A}$
be the projection-valued spectral measure for $A$, and let $\bM=\bB^* \bE \bB$ be the corresponding
compressed spectral measure.

Then for any complete system $\{\bx_k\}_{k\in\N}$ in $\cK$ and any sequence 
$\{\alpha_k\}_{k\in\N}$, $\alpha_k>0$, $\sum_k \alpha_k \|\bB\bx_k\|^2 <\infty$, the measure $\mu$ 
given by
\begin{align*}
\mu(\fdot) :=\sum_{k\in\N} \alpha_k \Jap{\bM(\fdot) \bx_k, \bx_k}
\end{align*}
is a spectral measure of maximal type for $A$. 
\end{cor}

\begin{proof}
Without loss of generality we assume that $A$ is the multiplication $\cM_{\id}$ by the independent
variable in the von Neumann direct integral \eqref{e:VN_int}. Define $f_k = \bB \bx_k\in \cH$. Since
$\Ran \bB$ is cyclic for $A$, and finite linear combinations of vectors $f_k$ are dense in $\Ran
\bB$, we see that
\[
\cF=\{f_k\;:\;k\in\N\}
\]
is a countable cyclic set for $A$.  
Therefore, by Lemma \ref{l-cyclicset}, we have
\[
\cspn\{ f_k(s) \;:\; k\in\N  \} =\cH(s)
\]
a.e.~with respect to $\nu$. This implies that 
\begin{align}\label{e:w>0}
\sum_{k} \alpha_k \|f_k(s)\|_{\cH(s)}^2 > 0 \qquad \nu\text{-a.e.}
\end{align}
(the sum could be infinite). 

Now consider the weights $w_k$,  $w_k(s):=\|f_k(s)\|^2\ci{\cH(s)}$ and define the measures  
$\mu_k := w_k\nu$. Let $w:=\sum_k \alpha_k w_k$.   Since 
\begin{align*}
\sum_{k} \alpha_k \mu_k(\R) =\sum_k \alpha_k \|f_k\|_\cH^2 < \infty, 
\end{align*}
the measure $\mu$ defined by
$$
\mu = \sum \alpha_k\mu_k =\left(\sum \alpha_k w_k\right)\nu = w\nu = \sum_{k\in\N} \alpha_k 
\Jap{\bM 
(\fdot)
\be_k, \be_k} 
$$
is a finite measure on $\R$. Therefore, $w(s)<\infty$ $\nu$-a.e., and we already know, see 
\eqref{e:w>0}, that $w(s)>0$ $\nu$-a.e. 

Thus, $\mu$ and $\nu$ mutually absolutely continuous, and $\nu$ is a scalar spectral measure of 
maximal spectral type for $A$. Then, by  Lemma \ref{l-maxtype1} the measure $\mu$ is also a 
spectral measure of maximal type. 
\end{proof}

%%%%%%%%%%%%%%%%%%%%%%%%%%%%%%%%%%%%%%%%%%%%%%%%%%%%%
\subsection{Measures of maximal spectral type and indefinite integrals}
%%%%%%%%%%%%%%%%%%%%%%%%%%%%%%%%%%%%%%%%%%%%%%%%%%%%%

\begin{prop}\label{p:max type base}
Let $A=A^*\in B(\cH)$ and let $\bE$ be its projection-valued spectral measure. Let $\bB:\cK \to \cH$
be such that the compressed spectral measure $\bM=\bB^* \bE \bB$ is an indefinite integral,
i.e.~$\bM=W\sigma$, see Section \ref{s-KurodaIndef}.

Then any spectral measure $\mu$ of $A$ of maximal spectral type is a base of $\bM$, i.e.~$\bM$ can 
be represented as $\bM=\wt W\mu.$\end{prop}

\begin{proof}
Let $\mu$ be a spectral measure of maximal type for $A$. Consider the Radon--Nikodym decomposition 
of $\sigma$ with respect to $\mu$, 
\begin{align*}
\sigma = w\mu + \sigma\ti s, \qquad \sigma\ti s \perp \mu. 
\end{align*}
Take $\bx \in\cK$ and define $\nu_{\bx }$ as 
\begin{align*}
\nu_{\bx}(\fdot) = \Jap{\bM (\fdot)\bx,\bx}\ci{\cK};
\end{align*}
thus, $\nu_{\bx}= w_{\bx} \sigma$, where $w_{\bx } (s) =\Jap{ W(s) \bx, \bx}\ci\cK$. 

Since $\nu_{\bx }$ is the scalar spectral measure of $A$ corresponding to the  vector  $\bB\bx$, it
is absolutely continuous with respect to $\mu$, so $w_\bx=0$ $\sigma\ti s$-a.e. Taking $\bx$ from a
countable dense set in $\cK$ we conclude that $W(s)=\bO$ $\sigma\ti s$-a.e., so
\begin{align*}
\bM = Ww\mu, 
\end{align*}
i.e.~$\mu$ is a base for $\bM$. 
\end{proof}

%%%%%%%%%%%%%%%%%%%%%%%%%%%%%%%%%%%%%%%%%%%%%%%%%%%
\section{Proof of the Aleksandrov disintegration theorem (Theorem 
\ref{t:AlAver01})}\label{s-disintegration}
%%%%%%%%%%%%%%%%%%%%%%%%%%%%%%%%%%%%%%%%%%%%%%%%%%%
Let $A$ and $A\ciG$ be operators given in Section \ref{s-perturbations}, also recalling the role of
$\cK$ and $\bB$. Further, let $\bM^{\Gamma} := \bB^* \bE\ci\Gamma \bB$ be the compressed spectral
measure of $A\ciG$ as defined in Section \ref{s:Fam Pert}.

Applying the following version of the Aleksandrov disintegration theorem to the function $\1\ci{E}$ 
we immediately obtain Theorem \ref{t:AlAver01}.

\begin{theo}[Aleksandrov disintegration]\label{t-disintegration1}
Let $\Gamma=\Gamma^*\ge0$ be a bounded and invertible operator on $\cK$ and let $\Ran \bB$ be 
cyclic for $A$. Then for any scalar-valued Borel function $f\in L^1(\R)$ we have
\begin{align}
\label{Disintegration 011}
\int_\R\left(\int_\R f(s) \bM^{t\Gamma}(\dd s) \right)\dd t
=
\Gamma^{-1} \int_\R f(s) \dd s ,
\end{align}
where the integral in the left side is understood in weak operator sense \tup(see Section 
\ref{ss-integration}\tup). 
\end{theo}

\subsection{Overview of the proof of Theorem \ref{t-disintegration1}}

In Sections \ref{s-pfdisintegrationPoisson} and \ref{s:CNSA}, we prove Theorem
\ref{t-disintegration1} for the case when $f$ is a Poisson kernel. For rank one and finite rank
perturbations ($\rank \bB=1$ \cite{Aleksandrov, SIMREV}, $\rank\bB<\infty$ \cite{JST}) this is a
pretty standard and easy reasoning. However, in the general case the new difficulties arise and we
have to apply more advanced techniques involving the   Sz.-Nagy--Foia\c s model theory
\cite{SzNF2010}. In fact, this is the hardest part of the proof.

Since linear combinations of Poisson kernels are dense in $C_0(\R)$, see Appendix \ref{s-apppoiss}
below,  we immediately get Theorem \ref{t-disintegration1} for $f\in C_0(\R)$. Finally, a routine
application of the monotone class lemma/Dynkin $\pi$--$\lambda$ theorem allows us to extend the
result to $f\in L^1(\R)$, see Section \ref{s-AleksGenProof}.

\subsection{Proof of Theorem \ref{t-disintegration1} for Poisson kernels} 
\label{s-pfdisintegrationPoisson}

Let $p_a$, $a\in\C_+$ be the Poisson kernel, 
\begin{align}\label{d-poisson}
p_a(s) = \frac1\pi \im \frac{1}{s-a} = \frac{1}{2\pi i} \left( \frac{1}{s-a}-\frac{1}{s-\bar a}  
\right), \qquad s\in \R . 
\end{align}
Trivially for $f=p_a$ the right-hand side of \eqref{Disintegration 011} computes to 
\begin{align*}
\Gamma^{-1} \int_\R p_a(s) \dd s = \Gamma^{-1}. 
\end{align*}

Denote 
\begin{align*}
h_a(t):= \int_\R p_a(s) \bM^{t\Gamma}(\dd s). 
\end{align*}
Recalling that $F$ and $F_{t\Gamma}$ are the Cauchy transforms of the measures $\bM$ and 
$\bM_{t\Gamma}$  respectively, see \eqref{d-M}, \eqref{e-DefFGamma}, we get that  
\begin{align*}
h_a(t)  = \frac1\pi \im 
    F\ci{t\Gamma}(a) =\frac1\pi \im \left( F(a)(\OID + t\Gamma F(a))^{-1}\right);
\end{align*}
note that by Lemma \ref{l-AK} the operator $\OID + t\Gamma F(a)$ is invertible, and we used  
\eqref{e-FGamma} for the second identity.

To prove Theorem \ref{t-disintegration1} for $f=p_a$ we need to show that $\int_\R h_a(t) \dd t 
=\Gamma^{-1}$.  

The substitution $\tau=1/t$ results in
\begin{align*}
\int\ci\R h_a(t)\dd t=
\frac1\pi\int\ci\R 
\im \left[F(a)(\tau\OID + \Gamma F(a))^{-1}\frac{1}{\tau}\right] \dd \tau ;
\end{align*}
note that since $\OID + t\Gamma F(a)$ is invertible, the operator $\tau \OID +\Gamma F(a)$ is 
invertible for all $\tau\in \R\setminus\{0\}$. The partial fraction decomposition 
\[
\Gamma F(a) \left(\tau\OID + \Gamma F(a)\right)^{-1}\frac{1}{\tau}
=
-(\tau\OID + \Gamma F(a))^{-1}
+\frac{1}{\tau}\OID
\]
now yields
\begin{align*}
\int\ci\R h_a(t)\dd t
&=
-\frac1\pi\int\ci\R \im\left[
\Gamma^{-1}(\tau\OID + \Gamma F(a))^{-1}
\right]\dd \tau  \\
&=
\frac1\pi\int\ci\R \im\left[
(\tau\OID +  F(a)^* \Gamma )^{-1} \Gamma^{-1}
\right]\dd \tau.
\end{align*}
To compute this integral let us note that 
\begin{align}\label{e:sim01}
\Gamma(\tau\OID +  F(a)^* \Gamma )  &= \Gamma^{1/2} (\tau\OID + \Gamma^{1/2} F(a)^* 
\Gamma^{1/2} )  \Gamma^{1/2} \intertext{so}\notag
(\tau\OID +  F(a)^* \Gamma )^{-1} \Gamma^{-1}  & =  \Gamma^{-1/2} (\tau\OID + \Gamma^{1/2} F(a)^* 
\Gamma^{1/2} )^{-1}  \Gamma^{-1/2}. 
\end{align}
By Equation \eqref{e-herglotzF}, we have $\im F(a)^* = -\im F(a) \le \OZ$ and so we conclude that
$\im ( \Gamma^{1/2} F(a)^* \Gamma^{1/2}  ) \le \OZ$. We know that the operators 
$\tau\OID +  \Gamma F(a) $ and so the operators $\tau\OID +  F(a)^* \Gamma$ are invertible for all 
$\tau\in \R\setminus \{0\}$, so by similarity relation \eqref{e:sim01} the same holds for 
$\tau\OID + \Gamma^{1/2} F(a)^* \Gamma^{1/2} $.

By Lemma \ref{l:cycl triv kernel} $\Ker F(z) = \{\OZ\}$ for all $z\in \C \setminus \R$. Since
$F(a)^* = F(\overline a)$, we conclude that operators $F(a)^*$, and so 
$\Gamma^{1/2} F(a)^*\Gamma^{1/2} $ also have trivial kernels.

\begin{defn}
A bounded operator $R$ is called \emph{dissipative} if $\im R\ge \bO$. 
\end{defn}

It is well known and easy to show that for a dissipative $R$
\[
\sigma(R) \subset \C_+\cup \R. 
\]

Theorem  \ref{t-disintegration1} follows immediately from the lemma below, applied to a dissipative 
operator $R= - \Gamma^{1/2} F(a)^* \Gamma^{1/2}$. 
\begin{lem}\label{l-integraleval}
Given a dissipative  operator $R$ such that $\sigma(R)\subset \C_+\cup \{0\},$ and \linebreak
$\Ker R = \{\OZ\},$  we have
\begin{align}\label{e-integraleval}
\frac{1}{\pi}\int_\R \im \left[ (s\OID - R)^{-1}  \right] \dd s = \I.  
\end{align}
\end{lem}

\subsubsection{Proof of Lemma \ref{l-integraleval} for invertible $R$ }
In the case when $R$ is invertible, so $\sigma(R) \subset \C_+$ the proof of Lemma 
\ref{l-integraleval} is a standard calculation using Riesz functional calculus. It is included for 
the  reader's convenience.

Recall that for $s\in \R$ 
\begin{align*}
\im \left[ (s\OID - R)^{-1}  \right] = \frac{1}{2i}\left( (s\OID - R)^{-1} - (s\OID - R^*)^{-1}  
\right).  
\end{align*}
Denoting for $z\in \C\setminus\R$
\begin{align}\label{def:p-oper}
\bp\ci{R} (z) := \frac{1}{2\pi i}\left( (z\OID - R)^{-1} - (z\OID - R^*)^{-1}  \right)
\end{align}
we see that to prove the lemma it suffices to show that 
\begin{align*}
\int_\R \bp\ci{R}(s) \dd s = \OID. 
\end{align*}

Consider the closed contour $\gamma_r$, $r>\|R\|$,  consisting of the semicircle
\[
S_r=\{z\in \C_+: 
|z|=r\}
\]
in the upper half-plane (traversed in counter-clockwise direction) and the interval 
from $-r$ to $r$ along the real line.

If $R$ is invertible and so $\sigma(R)\subset \C_+$, we get that $\sigma(R^*)\subset \C_-$, so
$\bp\ci{R} (z) $ is invertible on $\gamma_r$ and the integral $\int_{\gamma_r} \bp\ci{R} (z) \dd z$
is well defined.

It is easy to see that $\|\bp\ci R (z)\| = \mathcal{O}(r^{-2})$ for $|z|=r$, so  
\begin{align*}
\int\ci{S_r}\bp\ci R (s) \dd s \to 0 \qquad \text{as} \quad r\to \infty.
\end{align*}
Therefore we have
\begin{align*}
 \int_\R \bp\ci{R} (s)  \dd s =  \lim_{r\to+\infty}  \int_{-r}^r \bp\ci{R} (s)  \dd s 
=  \lim_{r\to+\infty}  \int_{\gamma_r} \bp\ci{R} (z) \dd z.
\end{align*}

Recalling that $\sigma(R^*)\subset \C_-$ we see that $(s\OID - R^*)^{-1}$ is analytic for $s\in
\C_+$,  so the integral of this term over $\gamma_r$ evaluates to the zero operator $\OZ.$ Therefore
for $r>\|R\|$
\begin{align*}
\int_{\gamma_r} \bp\ci{R} (z) \dd z = \frac{1}{2\pi i} \int_{\gamma_r} (z\OID - R)^{-1}  
\dd z = \OID;
\end{align*}
the last equality follows from the Riesz functional calculus, applied to the function $f(z) \equiv 
1$. It can also be easily obtained by changing the contour to the circle of radius $r$ and using 
the power series decomposition 
\begin{align*}
(z\OID - R)^{-1} = \sum_{k=0}^\infty \frac{R^k}{z^{k+1}}, \qquad |z|> \|R\|. 
\end{align*}
This concludes the proof of Theorem \ref{t-disintegration1} for Poisson kernels.
\hfill\qed

%%%%%%%%%%%%%%%%%%%%%%%%%%%%%%%
\subsubsection{Proof of Lemma \ref{l-integraleval} in the general case}
Let us recall some definitions.

\begin{defn}
A dissipative operator $R$ ($\im R\ge\OZ$) is called \emph{completely non-self-adjoint} if there is 
no reducing subspace $E$ of $R$ such that $R\bigm|_{E}$ is self-adjoint. 
\end{defn}

It is easy to see that $R$ from Lemma \ref{l-integraleval} is completely non-self-adjoint. Indeed,
we have $\sigma(R)\cap \R\subset\{0\}$,  and $0$ cannot be an eigenvalue of $R$, because $\Ker R$ is
trivial.

\begin{lem}\label{l-CNSA}
Given a completely non-self-adjoint dissipative operator $R$ and a vector $\bx\in\cH$ there exists 
a \tup(unique\tup) non-negative function $w=w_\bx\in L^1(\R)$ such that for $z\in \C_-$ 
\begin{align}\label{e:sp-dens}
\left\langle (z\I- R )^{-1} \bx,\bx  \right\rangle_\cH = \int_\R \frac{w(t)}{z-t} \dd t.
\end{align}
\end{lem}
The proof of this lemma will be presented below in Section \ref{s:CNSA}. 

Let us show that this lemma implies Lemma \ref{l-integraleval}. For $z\in\C_-$ slightly abusing the 
notation let us denote by $w(z)$ the Poisson extension of $w$ to the point $z$, 
\begin{align*}
w(z) = \frac1\pi \im \int_\R  \frac{w(t)}{z-t} \dd t 
= \frac1\pi \im \int_\R  \frac{w(t)}{t - \overline z} \dd t   
= \int_\R  w(t) p_{\overline z} (t) \dd t.
\end{align*}
Then \eqref{e:sp-dens} implies that 
\begin{align} \label{e:PoissonExt}
\Jap{\im \left[(z\OID - R)^{-1}  \right] \bx, \bx }_\cH = \im \jap{(z\bI - R)^{-1} \bx, \bx} = \pi 
w(z). 
\end{align}
Since $\im w(z)\ge 0$ for $z\in\C_-$, we see that for a.a.~$s\in\R$
\begin{align}\label{e:bv01}
w(s) = \lim_{\e\to 0^+} w(s-i\e), 
\end{align}
and since the left-hand side of \eqref{e:sp-dens} is analytic on $\R\setminus\{0\}$ we can assume
(take a representative for an element $w\in L^1(\R)$) that \eqref{e:bv01} holds for all
$s\in\R\setminus\{0\}$, and
\begin{align} \label{e:w(s)}
w(s) = \frac1\pi \im  \jap{(s\bI - R)^{-1} \bx, \bx}, \qquad \forall s\in\R\setminus\{0\}. 
\end{align}
Since $R+i\e\bI$ is an invertible dissipative operator and 
\begin{align*}
w(s-i\e) = \im  \jap{(s\bI - (R + i\e\bI))^{-1} \bx, \bx}, \qquad s \in \R
\end{align*}
we get from Lemma \ref{l-integraleval} for invertible $R$ (which we already proved) that 
\begin{align*}
\int_\R w(s-i\e) \dd s = \|\bx\|^2. 
\end{align*}
But the Poisson extension of a (positive) measure on the real line $\R$ to the line $\R-i\e$ 
preserves the total mass, so 
\begin{align*}
\int_\R w(s) \dd s = \|\bx\|^2. 
\end{align*}
Using the identity \eqref{e:w(s)} we get that 
\begin{align*}
\frac1\pi \int_\R \Jap{ \im \left[ (s\bI - R)^{-1}    \right] \bx, \bx  } \dd s =\|\bx\|^2, 
\end{align*}
which is exactly the conclusion of Lemma \ref{l-integraleval}.\hfill\qed

Thus, Lemma \ref{l-integraleval}, and so Theorem \ref{t-disintegration1} for Poisson kernels are 
proven modulo Lemma \ref{l-CNSA}. 

\subsection{Proof of Lemma \ref{l-CNSA}}\label{s:CNSA}
Note that the statement of Lemma \ref{l-CNSA} is reminiscent of the statement of the spectral
theorem, with $w(t) \dd t$ being the spectral measure. But the operator $R$ is not self-adjoint, so
how that is possible?

A reader familiar with the Sz.-Nagy--Foias functional model for contractions and its analogue for
dissipative operators could probably guess where this ``spectral measure'' comes from. Namely, that
it comes from the self-adjoint dilation of a dissipative operators.

Below we present a rigorous proof of the lemma: we will keep all presentation self-contained, with 
no knowledge of the functional model required. 

It is easier for us to work with contractions, so the first step is to transform the dissipative 
operator $R$ ($\im R\ge \bO$) to a contraction $T$ ($\|T\|\le 1$) using Cayley 
transform. 

Let us recall that an operator $T$ is called  a contraction if $\|T\|\le1$. 

Let us also recall that the Cayley transform $\Cay$, 
\begin{align*}
\Cay(z) = \frac{z-i}{z+i}
\end{align*}
conformally  maps the upper half-plane $\C_+$ to the unit disc $\D$.  It is also well known and 
easy to show that for a \emph{bounded} operator $R$
\begin{enumerate}
\item  $R=R^*$ if and only if $\Cay(R)$ is unitary; 
\item $R$ is a disspative operator ($\im R\ge \bO$) if and only if $T=\Cay(R)$ is a contraction 
(i.e.~$\|T\|\le1$). 
\end{enumerate}
For statement \cond2  see for example \cite[Theorem 1.4.2]{Nik-book-v2}. Statement \cond1 is 
elementary and can be found in many standard textbooks, cf.~\cite[Ch.~4, Section 3, Theorem 
2]{{BirmanSol-book_1987}}.

\begin{defn}
A contraction  $T$ ($\|T\|\le1$) is called \emph{completely non-unitary}  if there is no reducing 
subspace $E$ for $T$ such that $T\big|_{E}$ is unitary 
\end{defn}
Combining the above statements \cond1, \cond2, and noting that the Cayley transform of the 
restriction to a reducing subspace is just the restriction of the original Cayley transform, we 
deduce that 
\begin{enumerate}\setcounter{enumi}{2}
\item $R$ is a completely non-self-adjoint dissipative operator if and only if $T=\Cay(R)$ is a 
completely non-unitary contraction. 
\end{enumerate}

\subsubsection{Preliminaries: minimal unitary dilation of a contraction and its spectral type}

We need the following simple fact, see \cite[Ch.~I, Theorem 4.2]{SzNF2010}.

\begin{theo}\label{t:dilation}
For a contraction $T$ \tup($\|T\|\le1$\tup) in a Hilbert space $\cH$ there exists a \emph{unitary
dilation}, i.e.~a unitary operator $U$ on a space $\cK\supset \cH$ such that
\begin{align*}
T^n = P\ci{\cH} U^n \big|_{\cH}\qquad \forall n\in\N. 
\end{align*}
Moreover, one can also find a \emph{minimal} unitary dilation, meaning that 
\begin{align*}
\cspn \{U^n\cH: n\in\Z  \} = \cK. 
\end{align*}
\end{theo}

We need the following simple result, see \cite[Ch.~II, Theorem 6.4]{SzNF2010}.

\begin{theo}\label{t:DilSpType}
Let $T$ be a completely non-unitary contraction and $U$ be its \emph{minimal} unitary dilation. 
Then $U$ has purely absolutely continuous spectrum. 
\end{theo}

In fact, \cite[Ch.~II, Theorem 6.4]{SzNF2010} states a bit more, but for our purposes the above 
simpler statement is sufficient. 

Let us briefly explain its proof. The standard construction of a unitary dilation gives us for a 
completely non-unitary contraction a unitary dilation $U$ which is unitary equivalent to the 
bilateral (possibly multiple) shift, i.e.~the multiplication by the independent variable $\xi$ in 
the vector-valued space $L^2(\T, |\dd \xi| ;E)$, where $E$ is an auxiliary Hilbert space. This 
operator clearly has  purely absolutely continuous spectral measure. 

To get a minimal unitary dilation one needs to restrict this operator to an appropriate reducing 
subspace; of course the  spectral measure of this restriction will still be purely absolutely 
continuous.

\subsubsection{Conclusion of the proof of Lemma \ref{l-CNSA}}

Consider a completely non-unitary contraction $T$, and let $U$ be its minimal unitary dilation. By 
the above Theorem \ref{t:DilSpType} the spectral measure $\mu\ci{U,\bx}$ of $U$ with respect to a 
vector $\bx\in\cH$ is purely absolutely continuous,  $ \mu\ci{U,\bx} (\dd\xi) = w\ci{U,\bx}(\xi) 
|\dd\xi|$. 

Then for any analytic polynomial $f$ and any $\bx\in\cH$
\begin{align}\notag
\Jap{f(T)\bx,\bx} = \Jap{f(U)\bx,\bx}  & = \int_\T f(\xi)  \mu\ci{U, \bx} (\dd \xi)\\ 
\label{e:SpMeas02}
& = \int_\T f(\xi)  w \ci{U, \bx} (\xi) |\dd \xi|.
\end{align}
Standard approximation reasoning allows us to extend this identity to functions holomorphic on 
$\overline \D$, i.e.~on some open neighborhood of $\overline \D$.%
\footnote{In fact this identity is true for any $f\in H^\infty(\D)$, but for our purposes $f$ 
holomorphic on $\overline\D$ is sufficient.}

Let now $R$ be a completely non-self-adjoint dissipative operator, $T=\Cay(R)$ and let $U$ be a 
minimal unitary dilation of $T$; note that $T$ is a completely non-unitary contraction. 

Therefore for $\bx\in\cH$ we can write using \eqref{e:SpMeas02} 
\begin{align}\label{e:SpMeas03}
\Jap{(z\bI - R)^{-1}\bx,\bx} &= \Jap{(z\bI - \Cay^{-1}(T))^{-1}\bx,\bx} = \int_\T f(\xi)  w \ci{U, 
\bx} (\xi) |\dd \xi| , 
\intertext{where} \notag
f(\xi) &= (z - \Cay^{-1}(\xi))^{-1}. 
\end{align}
Note that for $z\in\C_-$ the function $f$ is well-defined and analytic in a neighborhood of the
closed unit disc. Making the change of variables $\xi =\Cay(t)$ in the integral in
\eqref{e:SpMeas03} we get
\begin{align*}
\Jap{(z\bI - R)^{-1}\bx,\bx} = \int_\R \frac{1}{z-t} w\ci{U,\bx}(\Cay(t)) |\Cay'(t)| \dd t, 
\end{align*}
so \eqref{e:sp-dens} holds with $w(t) = w\ci{U,\bx}(\Cay(t)) |\Cay'(t)|$. \hfill\qed

This concludes the proof of the Aleksandrov disintegration theorem, Theorem \ref{t-disintegration1},
for Poisson kernels.

\subsection{General case of the Aleksandrov disintegration theorem }\label{s-AleksGenProof}
So far, we have proven the Aleksandrov disintegration theorem, Theorem \ref{t-disintegration1}, for 
Poisson kernels $f(s) =p_a(s) = \frac 1\pi \im \frac{1}{s-a} $, $a\in \C_+.$ Here we show how a 
monotone class argument and the repeated application of Convergence Theorems can be used to show 
the formula for general $L^1$ functions, completing the proof of Theorem \ref{t-disintegration1}.

%%%%%%%%%%%%%%%%%%%%%%%%%%%%%%%%%%%%%
\subsubsection{The case of continuous compactly supported functions}
Let $f\in C\ti c(\R)$. By virtue of Lemma \ref{le:poissonappro} below we let $f_n$ be a sequence of
finite linear combinations of Poisson kernels so that
\begin{align}\label{e-poisson}
\sup_{s\in\R}(1+s^2)|f_n(s)-f(s)|\to 0.
\end{align}

On the right-hand side of \eqref{Disintegration 011} this yields
\[
\left|\int\ci\R (f_n(s)- f(s)) \dd s\right|
\le 
\left(\int\ci\R \frac{\dd s}{1+s^2}\right)
 \sup_{s\in\R}(1+s^2)|f_n(s)-f(s)|\stackrel{n\to\infty}{\longrightarrow}0.
\] 

Let us focus on the left hand side. From \eqref{dM 01} we immediately obtain that 
\begin{align*}
\|\bM^\Gamma(\R)\| \le \|\bB\|^2  . 
\end{align*}
{If $\varphi$ is bounded and Borel measurable, then}
\begin{align*}
\left|\left\langle
\int \varphi(s)\bM(\dd s)\bx,\bx
\right\rangle\right|
=&
\left|\int \varphi(s)
\left\langle
\bM(\dd s)\bx,\bx
\right\rangle\right|
\\
=&
\left|\int \varphi(s)
\mu_\bx(\dd s)\right|
\\
\le&\|\varphi\|_\infty \mu_\bx(\R)
\\
\le&  \|\varphi\|_\infty \|\bM(\R)\|\|\bx\|^2.
\end{align*}
So the quadratic form of the operator  $\int\varphi(s)\bM(\dd s)$ is bounded by 
$\|\varphi\|_\infty \|\bM(\R)\|$. If $\varphi$ is real valued, the integral gives us a self-adjoint 
operator, and since for self-adjoint operators the norm coincides with the norm of its quadratic 
form, we can  conclude that
\[
\left\|\int\varphi(s)\bM(\dd s)\right\|
\le
\|\varphi\|_\infty \|\bM(\R)\|.
\]
By treating real and imaginary part of $\varphi$ separately, we obtain that for complex-valued 
$\varphi$ 
\[
\left\|\int\varphi(s)\bM(\dd s)\right\|
\le
2\|\varphi\|_\infty \|\bM(\R)\|.
\]
So, with $\|f_n-f\|_\infty\le\sup_{s\in\R}(1+s^2)|f_n(s)-f(s)|,$ we see that 
\begin{align}\label{e-uniformContComp}
\left\| \int\ci\R f_n(s) \bM^{\Gamma}(\dd s)
-  
\int\ci\R f(s) \bM^{\Gamma}(\dd s) \right\| \to 0
\quad\text{uniformly with respect to $\Gamma$}.
\end{align}

From the definition \eqref{e-DefFGamma} of  $\bM^{\Gamma}$, and   formula \eqref{d-poisson} of 
$p_a$, we see   using the formula   \eqref{F_Gamma} for $F\ci\Gamma$, that for $a\in\C_+$
\begin{align}\label{e-PaFG}
2\pi i\int\ci\R
p_a(s)\bM^{\Gamma}(\dd s)
&=
F\ci\Gamma(a)-F\ci\Gamma(\bar a) \\ \notag
&=
F(a)(\OID + \Gamma F(a))^{-1}
-F(\bar a)(\OID + \Gamma F(\bar a))^{-1}. \notag
\end{align}
So for any $a\in\C_+$ the function $\Gamma\mapsto \int\ci\R
p_a(s)\bM^{\Gamma}(\dd s)$ is continuous in the norm topology of $B(\cK)$.  

Therefore the functions $\Gamma\mapsto \int f_n(s) \bM^\Gamma (\dd s)$ are continuous, 
so the function
\begin{align}\label{e-cont}
\Gamma\longmapsto \int\ci\R f(s)\bM^{\Gamma}(\dd s)\quad\text{is continuous}
\end{align}
as a uniform limit (see Equation \eqref{e-uniformContComp}) of continuous functions.

We know that \eqref{Disintegration 011} holds for the functions $f_n$, so we need to show that 
\begin{align*}
\lim_{n\to \infty} \int_\R\left(\int_\R f_n(s) \bM^{t\Gamma}(\dd s) \right)\dd t   = 
\int_\R\left(\int_\R f(s) \bM^{t\Gamma}(\dd s) \right)\dd t ;
\end{align*}
for our purposes it suffices to have the limit in the weak operator topology. 

To interchange the limit and the outer integral we can  apply the Dominated Convergence 
Theorem  with the majorant $C/(1+t^2)$ for the value of the inner integral; the 
interchange with the inner integral is already proved in \eqref{e-uniformContComp}.  

To get the majorant we first notice that since $f_n$ are uniformly bounded and
\[
\|\bM^{t\Gamma}(\R)\|\le \|\bB\|^2,
\]
we only need to show that
\begin{align}\label{e-decay}
\left\|\int\frac{\bM^{t\Gamma}(\dd s)}{1+s^2}\right\| = \mathcal{O}(t^{-2}) \qquad t\to \infty.
\end{align}
To arrive at \eqref{e-decay}, we first recall that 
$\int_{\R} (1+s^2)^{-1}\bM^{t\Gamma}(\dd s) = \im F\ci{t\Gamma}(i)$, 
see \eqref{e-herglotzF} with $\Gamma$ replaced by $t\Gamma$. We then use \eqref{e-AKIm2} with $z=i$ 
and $\Gamma$ replaced by $t\Gamma$, together with the estimate
\[
\|(\OID + t\Gamma F(\pm i))^{-1}\|= \mathcal{O}(|t|^{-1}).
\]
This last estimate follows immediately by noting that both  $\Gamma$ and $F(\pm i)$ are  invertible.

At this point, we have proved the disintegration theorem, Theorem \ref{t-disintegration1}, for  
functions $f\in C\ti{c}(\R)$.

\subsubsection{The case of Borel measurable functions}
To extend to non-negative bounded measurable functions, we apply  the standard Monotone Class
argument, which is essentially the same that was used in~\cite[Section 9.4]{cimaross}. For
matrix-valued measures it was done in  \cite{JST}. But, as we are dealing with the operator-valued
measures, we include the complete argument to assure that we did not miss any details.

To that end, recall that a collection $\cT$ of subsets is called a $\pi$-system, if it is closed
under finite intersections. We denote by $\sigma(\cT)$ the sigma-algebra generated by $\cT$.

We need the following well-known theorem, see \cite[Section 3.14]{Williams_Prob-mart}. 
 \begin{theo}
\label{t:monotone class}
Let $\cS$ be a set of bounded functions $g:X\to\R$, and $\cT$ be a $\pi$-system such that
\begin{enumerate}
	\item $\cS$ is a real vector space;
	\item the constant function $\1$ belongs to $\cS$;
	\item if $(g_n)\ci{n\ge 1} $ is an increasing sequence of nonnegative functions in $\cS$ such 
	that its limit $g$ 
	\[
	g(x)=\lim_{n\to\infty} g_n(x)
	\]
	is bounded, then $g\in\cS$;
	\item  $\cS$ contains all indicator functions $\1\ci I$, $I\in\cT$.
\end{enumerate}
Then $\cS$ contains all bounded $\sigma(\cT)$-measurable functions. 
 \end{theo}

Fix a non-negative invertible $\Gamma$. Let the $\pi$-system $\cT$ be   the collection of all  open
intervals of the real axis, so $\sigma(\cT)$ is the Borel sigma algebra. And let $\cS$ be the class
of bounded measurable real functions $g$ on $\R$ such that
\begin{enumerate}
	\item[(a)] 
		the function  $\Phi_g =\Phi_g^{\Gamma, \bx}$
	\[
	\Phi_g(t) :=  \left\langle\int_\R \frac{g(s)}{1+s^2}\bM^{t\Gamma}(\dd s)\bx,\bx\right\rangle
    =
    \int_\R \frac{g(s)}{1+s^2}\left\langle\bM^{t\Gamma}(\dd s)\bx,\bx\right\rangle
	\]
	is Borel measurable for all $\bx\in \cK$. 
	\item[(b)] the distintegration identity \eqref{Disintegration 011} (with integrals being 
	finite) 
	holds for $f$, $f(s)=g(s)/(1+s^2)$.
\end{enumerate}

Notice that all continuous functions with compact support belong to the class $\cS$ defined this
way. Indeed, condition (a) follows from the continuity \eqref{e-cont}, and condition (b) follows
from the fact (just proved before) that the disintegration identity \eqref{Disintegration 011} holds
for $f\in C\ti c(\R).$

Let us now show that this class $\cS$ satisfies the assumptions of Theorem \ref{t:monotone class}. 
The above class $\cS$ is clearly a real vector space. 

The condition (iii) follows immediately from the Monotone Convergence Theorem. Namely, let 
$g_n\in\cS$, $g_n(s) \nearrow g(s)$. 

We next show that $g$ satisfies properties (a) and (b), starting with  the property (b) first. The 
identity \eqref{Disintegration 011} holds for $f_n$, $f_n(s) := g_n(s) /(1+s^2)$, by property (b) 
of 
$\cS$, and applying the Monotone Convergence Theorem for both sides of \eqref{Disintegration 011} 
(twice on the left-hand side), we will get that it holds for $f$, $f(s) = g(s)/(1+s^2)$. Notice 
that since the integral $\int f(s) \bM^{t\Gamma}(\dd s)$ is understood in the weak sense, we only 
need the classical scalar-valued Monotone Convergence Theorem. Notice that for a bounded $g$ the 
right-hand side of \eqref{Disintegration 011} is finite for all  $f$, $f(s) = g(s)/(1+s^2)$.  

Let us now verify property (a) for $g$. By property (a) we know that  the functions $\Phi_{g_n}$ 
are Borel measurable, and since by the Monotone Convergence Theorem $\Phi_{g_n}(t) \to \Phi_g(t)$ 
for all $t$ we conclude that $\Phi_g$ is Borel measurable as well. 

To verify (ii) and (iv) we take a sequence of functions $g_n\in C\ti c$, $g_n\nearrow \1$ or 
$g_n\nearrow \1\ci I$, and apply property (iii) that we already verified. 

Thus we have proved that our class $\cS$ contains all bounded Borel measurable functions, so the 
disintegration identity \eqref{Disintegration 011} holds for all Borel measurable functions $f\ge0$ 
such 
that the function $s\mapsto (1+s^2) f(s)$ is bounded. In particular it holds for compactly 
supported bounded Borel measurable $f\ge 0$. Taking a nonnegative $f\in L^1(\R)$ we can construct 
an increasing sequence of bounded, compactly supported Borel measurable functions $f_n$, $f_n(s) 
\nearrow f(s)$, and again apply the Monotone Convergence Theorem to both sides of  
\eqref{Disintegration 011} (we again need to apply it twice to the left hand side). 

Thus we have proved \eqref{Disintegration 011} for non-negative $f\in L^1(\R)$, and by linearity it 
extends to all $L^1(\R).$
\hfill$\qed$

%%%%%%%%%%%%%%%%%%%%%%%%%%%%%%%%
\section{Aleksandrov disintegration implies Theorem \ref{t:GA-D_01}
}\label{s:Proof a.e. except}
%%%%%%%%%%%%%%%%%%%%%%%%%%%%%%%%%%%%%%%%%%
In fact, we can prove a bit stronger  statement than Theorem \ref{t:GA-D_01}, namely the following 
theorem. 
\begin{theo}\label{t:mu_t(Omega)=0}
Let $A_t= A+tK$ with $K\ge \bO$ with $\Ran K$ being cyclic for $A$, and let  $\Omega$ be  Borel set
of Lebesgue measure $0$. For $t\in\R$ let $\mu_t$ be a scalar spectral measure of maximal type of 
$A_t$.

Then  $\mu_t(\Omega)=0$ for Lebesgue almost all $t\in\R$ \tup(the statement clearly does not depend
on the choice of the  measures $\mu_t$\tup).
\end{theo}

To get Theorem \ref{t:GA-D_01} from Theorem \ref{t:mu_t(Omega)=0} we take $\Omega$ to be a 
\emph{carrier} of the singular measure $\sigma$, i.e.~a Borel set of Lebesgue measure $0$ such that 
$\sigma(\R\setminus \Omega)=0$: Theorem \ref{t:GA-D_01} follows immediately. 

\begin{proof}[Proof of Theorem \ref{t:mu_t(Omega)=0}]
First  using Lemma \ref{l:factorization} let us represent  $K$ as  $K = \bB\Gamma\bB^*$ with
$\Ker\bB=\{\bO\}$ and $\Gamma\in B(\cK)$ being invertible. Then $\cRan \bB=\cRan K$, so the
cyclicity of $\Ran K$ for $A$ implies the same for $\Ran \bB$.

Since $K\ge \bO$ we conclude that $\Gamma\ge0$, so we can apply Theorem \ref{t:AlAver01}. Therefore
we have
\begin{align*}
\int_\R \bB^* \bE_t(\Omega) \bB \dd t = \bO . 
\end{align*}
The above integral is understood in the weak sense, so denoting $\bM_t:=\bB^* \bE_t\bB$ we see 
that  for any $\be\in \cH$
\[
\int_\R \Jap{\bM_t(\Omega) \be,\be} \dd t =\int_\R \Jap{\bB^*\bE_t(\Omega)\bB \be,\be} \dd t = 0.  
\]
If $\{\be_k\}_{k\ge1}$ is an orthonormal basis and $\alpha_k>0$ are such that $\sum_k 
\alpha_k<\infty$, we trivially have that 
\[
 \int_\R \left(\sum_k \alpha_k\Jap{\bM_t(\Omega) \be_k,\be_k} \right) \dd t =0. 
\]
Denoting by $\mu_t$ the scalar measure given by 
\[
\mu_t(E) = \sum_k \alpha_k\Jap{\bM_t(E) \be_k,\be_k}
\]
we see that $\mu_t(\Omega) = 0$ for Lebesgue almost all $t\in\R$, i.e.~the measures $\mu_t$ are 
mutually singular with $\sigma$ for Lebesgue almost all $t$. 

Since $\Ran\bB$ is cyclic for $A$, we see that by Lemma \ref{l-cycAGamma} it is cyclic for all 
$A_t= A + t\bB\Gamma\bB^*$. So by Corollary \ref{c:max type} for any $t\in\R$ the measure $\mu_t$ 
is the spectral measure of maximal type of $A_t$. 
\end{proof}

%%%%%%%%%%%%%%%%%%%%%%%%%%%%%%%
\section{Proof of Theorems \ref{t-two} and  \ref{t-many}}\label{s-parameters}
%%%%%%%%%%%%%%%%%%%%%%%%%%%%%%%
Theorem \ref{t-two} follows trivially from Theorem \ref{t-many}. To prove Theorem \ref{t-many} we 
again as in Section \ref{s:Proof a.e. except} prove a bit stronger statement. 

\begin{theo}\label{t:many 01}
Let a family of operators $A_{\vec t\,}$ satisfy the assumptions of Theorem \ref{t-many}, and  let
$\Omega\subset\R$ be a Borel measurable set of Lebesgue measure $0$. Denote by $\mu_{\vec t\,}$ a
scalar spectral measure of maximal type of the operator $A_{\vec t}$.

Then $\mu_{\vec t\,}(\Omega)=0$ for \tup(Lebesgue\tup) almost all $\vec t\,\in \R^n$.
\end{theo}

\begin{proof}[Proof of Theorem \ref{t:many 01}]
While the statement of the theorem does not depend on the choice of measures $\mu_{\vec t}$, in the 
proof it is convenient to make a concrete choice. Namely, take an orthonormal basis 
$\{\bx_k\}_{k\in\N}$ in $\cH$ 
and define 
\begin{align*}
\mu_{\vec t\,}(\fdot):= \sum_{k\in\N} \alpha_k\Jap{\bE_{\vec t\,}(\fdot)\bx_k, \bx_k}, 
\end{align*}
where $\alpha_k>0$, $\sum_{k\in\N}\alpha_k<\infty$, and $\bE_{\vec t\,} $ denotes the 
projection-valued spectral measure of $A_{\vec t}$. Then by Corollary \ref{c:max type} with 
$\bB=\bI$ and $A=A_{\vec t\,}$ we get that $\mu_{\vec t\,}$ is a spectral measure of maximal type 
of $A_{\vec t\,}$.

Denote by $E$ the exceptional set, 
\begin{align*}
E:=\{\vec t\in\R^n: \mu_{\vec t\,}(\Omega)>0\}. 
\end{align*}
We will show later that the function $\vec t\mapsto \mu_{\vec t\,}(\Omega)$ is Borel measurable, 
so the set  $E$ is Borel measurable. 

Let $\vec\tau$ be the vector from Theorem \ref{t-many}. Since $\cRan K_{\vec \tau} $ is cyclic, we
conclude that $K_{\vec\tau}\ne\bO$, so the vector $\vec\tau$ must be non-zero. By our assumptions,
Theorem \ref{t:mu_t(Omega)=0} now applies with $K=K_{\vec \tau}$, and we deduce  that for any $\vec
t\in (\vec \tau\,)^\perp\subset \R^n$ the set (cross-section)
\begin{align*}
E_{\vec t\,}:= \{s\in\R: \vec t\, +  s\vec\tau\, \in E \}
\end{align*}
has Lebesgue measure $0$  (measurability part of the Tonelli's Theorem implies that the sets
$E_{\vec t\,}$ are Borel measurable). Therefore applying the full Tonelli's Theorem to $\1\ci E$ we
get that $E$ has Lebesgue measure $0$.

To show that the function $\vec t\mapsto \mu_{\vec t\,}(\Omega)$ is Borel measurable we again, as in
Section \ref{s-AleksGenProof} use the Monotone class theorem (Theorem \ref{t:monotone class}).
Denote by $\cS$ the set of all bounded functions $g:\R\to \R$ such that the function
\begin{align*}
\vec t\mapsto \int_\R g(s) \mu_{\vec t\,}(\dd s)
\end{align*}
is measurable. 

Clearly $\cS$ is a real vector space, i.e.~assumption \cond1 of Theorem \ref{t:monotone class} 
is satisfied. 

Let us check  assumption \cond3. Take an increasing sequence of non-negative functions 
$g_k\in\cS$ such that the limit function $g$, $g(s)=\lim_{k\to\infty} g_k(s)$ is bounded. By the 
Monotone Convergence Theorem 
\begin{align*}
\lim_{k\to \infty } \int_\R g_k \mu_{\vec t\,}(\dd s) = \int_\R g  \mu_{\vec t\,}(\dd s) \qquad 
\forall \vec t\,\in\R^n, 
\end{align*}
so the function $\vec t\,\mapsto \int_\R g  \mu_{\vec t\,}(\dd s)$ is measurable as a pointwise 
limit of measurable functions. Thus assumption \cond3 is satisfied. 

To prove that the remaining assumptions \cond2 and \cond4 are also satisfied we first show that the 
Poisson kernels $p_a$, $p_a(s):=\frac1\pi \im \frac{1}{s-a}$, $a\in\C_+$ belong to $\cS$. 

We will need the following simple lemma. 
\begin{lem}\label{l:cont mu_t}
Let $A$, $K_1, \dots K_n$ be self-adjoint operators, and let $A_{\vec t\,} :=A+\sum_{k=1}^n t_k 
K_k$, $\vec t= (t_1, \dots, t_n)$ \tup(no additional assumptions on $K_k$\tup). Let 
$\{\bx_k\}_{k\in\N}$ be a complete system of vectors and let $\alpha_k\ge0$ be such that $\sum_k 
\alpha_k \|\bx_k\|^2<\infty$. 

Let measures $\mu_{\vec t\,}$ be defined by 
\begin{align*}
\mu_{\vec t\,}(E) := \sum_k \alpha_k \Jap{\bE_{\vec t\,}(E)\bx_k, \bx_k}
\end{align*}
where $\bE_{\vec t\,}$ is the projection-valued spectral measure of $A_{\vec t\,}$. 

Then the function 
\begin{align*}
(a, \vec t\,) \mapsto \int_\R p_a (s) \mu_{\vec t\,}(\dd s) ,\qquad  a\in\C_+, \ \vec 
t\,\in\R^n
\end{align*}
is continuous. 
\end{lem}

\begin{proof}
Since $p_a(s):=\frac1\pi \im \frac{1}{s-s}$ we get from the definition of the spectral measure 
that  
\begin{align*}
\int_\R p_a(s) \mu_{\vec t\,}(\dd s) =\frac{1}{\pi} \sum_{k\in\N} \alpha_k \im 
\Jap{(A_{\vec t\,} - a \bI)^{-1}\bx_k, \bx_k}
\end{align*} 
Using the standard resolvent identity   
\begin{align*}
(A_{\vec t\,} - a\bI)^{-1} = (A - a\bI)^{-1}\left(\bI + K_{\vec t\,} (A - a\bI)^{-1}   \right)^{-1}
\end{align*}
(formally it is Lemma \ref{l-AK} with $\bB=\bI$ and $\Gamma=K_{\vec t\,}$) we see that the function
$(a,\vec t\,)\mapsto (A_{\vec t\,} - a\bI)^{-1}$ is a continuous function from $\C_+\times \R^n$ to
$B(\cH)$ equipped with the operator norm. Therefore, the function
\begin{align}\label{e-mucont}
(a,\vec t\,) \mapsto \int_\R p_a(s) \mu_{\vec t\,} (\dd s) = \frac{1}{\pi} \sum_{k\in \N} \alpha_k 
\im 
 \Jap{(A_{\vec t \,} - a\bI )^{-1} \bx_k, \bx_k  }
\end{align}
is continuous. 
\end{proof}
Let us continue with the proof of the theorem. By Lemma \ref{l:cont mu_t} the function $\vec t
\mapsto \int_\R p_a(s) \mu_{\vec t\,} (\dd s)$ is continuous and so it is measurable.  Thus
$p_a\in\cS$ for all $a\in\C_+$.

Now take $g\in C\ti c(\R)$. By Lemma \ref{le:poissonappro} there exist finite linear combinations 
$g_k$ of Poisson kernels such that $g_k \rightarrow g$ uniformly. Therefore 
\begin{align*}
\int_\R g(s) \mu_{\vec t\,}(\dd s) =\lim_{k\to\infty} \int_\R g_k(s) \mu_{\vec t\,}(\dd s) \qquad 
\forall \vec t\,\in \R^n
\end{align*}
(recall that $\mu_{\vec t\,}$ are finite measures), so the function $\vec t\,\mapsto \int_\R g(s) 
\mu_{\vec t\,}(\dd s)$  is measurable as a pointwise limit of measurable functions. Therefore 
$C\ti c(\R)\subset \cS$. 

Since for any open interval $I\subset \R$ (including $I=\R$) there exists an increasing sequence of 
non-negative functions $g_k \nearrow \1\ci{I}$ the (already proved) assumption \cond3 implies that 
$\1\ci{I}\in\cS$, so all assumptions of Theorem \ref{t:monotone class} are satisfied (collection of 
open intervals is definitely a $\pi$-system). 

Therefore $\cS$ contains all bounded Borel measurable functions, and in particular the function 
$\1\ci\Omega$. 
\end{proof}

%%%%%%%%%%%%%%%%%%%%%%%%%%%%%%%%%%%%%%%%%%%%%%%%%%%
\section{Representation theorem and the operator $\bA_2$ condition}\label{s-Repr}
%%%%%%%%%%%%%%%%%%%%%%%%%%%%%%%%%%%%%%%%%%%%%%%%%%%%
Ideas behind the representation theorem (Theorem \ref{t-repr}) go back to \cite{JFA2009} and were 
further developed in \cite{JST}.

Consider the perturbation problem \eqref{e-AGamma}, where $\Ran \bB$ is cyclic for $A$ (without 
loss of generality), and let $\bM = \bB^* \bE \bB$ be the compressed spectral measure defined by 
\eqref{dM 01} and \eqref{d-M}.

By the generalized spectral theorem (Theorem \ref{t:SpThm-OVM}), the canonical unitary operator
$U:L^2(\bM)\to \cH$ such that $AU =U\cM_{\id}$ is given on functions of form $f\be$, where
$\be\in\cK$ and $f$ is a scalar bounded Borel measurable function, by
\[
U(f\be) = f(A)\bB\be.
\]
It then extends by linearity to finite linear combinations of such functions, and then by continuity
to whole $L^2(\bM)$. We will call the above unitary operator $U$ the \emph{canonical spectral
representation} of $A.$

It is not hard to see that it is uniquely defined by the conditions
\begin{align}
AU =U\cM_{\id} , \qquad U(\1\be) = \1\bB\be ,
\end{align}
where $\1$ stands for the scalar function identically equal to 1.

To simplify the notation let us use the above unitary operator $U$ to transfer everything to the
space $L^2(\bM)$. So, we assume that $\cH=L^2(\bM)$, and $A$ is the multiplication $\cM_{\id} =
\cM_{\id}^\bM$ by the independent variable in $L^2(\bM)$. In this representation, operator
$\bB:\cK\to L^2(\bM)$ is given by
\begin{align}
    \label{e-mapBe}
    \bB \be = \textbf{1}  \be.
\end{align}
Here $\be\in\cK$, and we write $\1\be$ in the right-hand side to emphasize that this is a 
\emph{function} (an element of $L^2(\bM)$) identically equal to $\be$ for all $s\in\R$. 

The adjoint operator $\bB^*$ is given on bounded  functions $f$ with finite-dimensional range by 
\begin{align}\label{e-B*}
\bB^* f := \int_\R \bM(\dd t) f(t) . 
\end{align}

So, now the perturbed operator $A\ci\Gamma = A + \bB\Gamma\bB^* = \cM_{\id}^\bM + \bB\Gamma\bB^*$
acts on $L^2(\bM)$. Recall that by Lemma \ref{l-cycAGamma} $\Ran \bB$ is cyclic for $A\ci\Gamma.$
Therefore, we can apply the spectral theorem for a second time, now to obtain a spectral
representation of $A\ci\Gamma.$ This means that for every $\Gamma$ there is a unique unitary
operator
\begin{align}\label{e-UGamma1}
    V\ci\Gamma:L^2(\bM) \to L^2(\bM^\Gamma)
\end{align}
such that 
\begin{align}\label{e-UGamma2}
    V\ci\Gamma A\ci\Gamma  &= \cM_{\id}^{\bM_\Gamma} V\ci\Gamma,\\
    \label{e-UGamma3}
    V\ci\Gamma (\textbf{1}  \be)
    &= \textbf{1}  \be
\end{align}
for all $\be\in \cK$. Note that in \eqref{e-UGamma3} $\1\be$ in the left-hand side is interpreted as
an element of $L^2(\bM)$, while in the right-hand side it is an element of $L^2(\bM^\Gamma).$

For $z\in \C\setminus \R$ define functions $K_z$ on $\R$ by
\[
K_z(s) := (s-z)^{-1}, \quad s\in\R.
\]

Note that  in this notation  \eqref{e-UGamma2} is trivially equivalent to 
\begin{align}\label{e-UGamma4}
V\ci\Gamma (A\ci\Gamma - z\bI)^{-1} = \cM\ci{K_z} V\ci\Gamma, \qquad \forall z\in\C\setminus\R.  
\end{align}

The focus of this section is the following representation theorem for the spectral representation 
of $A\ci\Gamma$.

\begin{theo}\label{t-repr}
For finite linear combinations $f = \sum_{k=1}^{N}K_{z_k}\be_k$, where $z_k\in \C\setminus \R$, 
$\be_k \in \cK$,  the operator $V\ci{\Gamma}$ is given by 
\begin{align}\label{e:repr01}
(V\ci\Gamma f)(s)
& = f(s) - \Gamma \int\ci\R \bM (\dd t)\frac{f(t)-f(s)}{t-s} \\ \label{e:repr02}
& = \sum_k K_{z_k} (s)  \left( \be_k + \Gamma \int_\R  \bM(\dd t) K_{z_k}(t) \be_k \right).
\end{align}
\end{theo}

The following proof is similar to the proof of \cite[Theorem 5.1]{JST}. We include it for the 
convenience of the reader. 

%%%%%%%%%%%%%%%%%%%%%%%%%%%%%%%%%%%%%%%%%%%%%%%%
\begin{proof}
Let us prove it for functions of form $K_z\be$, the rest will follow by linearity of the formulas
\eqref{e:repr01}, \eqref{e:repr02}. Recall the resolvent identity \eqref{e-ResId}
\[
(A - z\OID)^{-1} - (A\ci\Gamma - z\OID)^{-1} = 
(A\ci\Gamma - z\OID)^{-1}\bB\Gamma\bB^*(A - z\OID)^{-1}, \qquad z\in\C\setminus \R. 
\]
Applying this identity to the vector $\bB\be$ with $\be\in \cK$, and then left multiplying  by 
$V\ci\Gamma$ we  obtain
\begin{align}\label{e-VResolvent}
V\ci\Gamma(A - z\OID)^{-1}\bB\be = V\ci\Gamma(A\ci\Gamma - z\OID)^{-1}\bB\be +
V\ci\Gamma(A\ci\Gamma - z\OID)^{-1}\bB\Gamma\bB^*(A - z\OID)^{-1}\bB\be.
\end{align}

Using the fact that $\bB\be =\1\be$, and recalling that $A$ is the multiplication $\cM_{\id}$  by 
the independent variable in $L^2(\bM)$, we see that 
\begin{align}\label{e:SpRepr_03}
 (A - z\OID)^{-1}\bB\be  = \cM\ci{\cK_z} \be,
\end{align}
so the left-hand side of \eqref{e-VResolvent} reads as
\[
 V\ci\Gamma(A - z\OID)^{-1}\bB\be   = 
 V\ci\Gamma( K_z  \be ),
\]
where, recall  $K_z(s):=(s-z)^{-1} .$

For the first term on the right-hand side of \eqref{e-VResolvent} we note that by 
\eqref{e-UGamma4}, \eqref{e-mapBe}, and \eqref{e-UGamma3}  we have  
\begin{align*}
 V\ci\Gamma(A\ci\Gamma-z\OID)^{-1}\bB\be =  \cM_{K_z}V\ci\Gamma \bB\be = 
 \cM_{K_z}V\ci\Gamma(\textbf{1}  \be) 
 =  K_z  \1 \be =K_z  \be .
\end{align*}  
For the second term on the right-hand side of \eqref{e-VResolvent} we first compute using
\eqref{e:SpRepr_03} and the formula \eqref{e-B*} for $\bB^*$
\begin{align}\label{e:e_z}
\be_z : = \bB^*(A - z\OID)^{-1}\bB\be =
\int\ci\R \bM(\dd t) K_z(t) \be. 
\end{align}

Note that $\Gamma\be_z\in\cK$, so using \eqref{e:SpRepr_03} with $\Gamma\be_z$ instead of $\be$ 
we rewrite the second term on the right-hand side of \eqref{e-VResolvent} as 
\begin{align*}
 V\ci\Gamma(A\ci\Gamma-z\OID)^{-1}\bB\Gamma\be_z =K_z \Gamma\be_z . 
\end{align*}
Equation \eqref{e-VResolvent} then  reads
\begin{align}\label{e:V k_z e}
V\ci\Gamma (K_z  \be ) =K_z   \be + K_z \Gamma\be_z,    
\end{align}
which is exactly \eqref{e:repr02} for $f=K_z \be$. 

Recalling formula \eqref{e:e_z} for $\be_z$ we see that 
\begin{align*}
K_z(s) \be_z = K_z(s) \int\ci\R   \bM(\dd t) K_z(t) \be = \int\ci\R \bM(\dd t) K_z(s) K_z(t) 
\be , 
\end{align*}
and the simple identity 
\[
K_z(s)K_z(t)
=
-\frac{K_z(t) - K_z(s)}{t-s}
\]
implies that
\[
K_z(s) \Gamma \be_z
=
-\Gamma \int\ci\R \bM(\dd t) \frac{K_z(t) - K_z(s)}{t-s}\be.
\]
Therefore \eqref{e:V k_z e}  is exactly the identity \eqref{e:repr01} for $f = K_z\be$,  so the 
representation theorem holds for such functions.  
\end{proof}

\subsection{The representation theorem implies an operator $\bA_2$ condition}\label{s-reptoa2}
Given an operator-valued measure $\bM$ on $\R$, let us (slightly abusing notation) denote by
$\bM(z)$, $z\in \C\setminus\R$ the Poisson extension of $\bM$ to the point $z$,
\begin{align*}
\bM(z) := \frac1\pi \int_\R \frac{|\im z|}{| s - z |^2} \bM(\dd s).
\end{align*}
Note that $\bM(z)$ is a positive operator, and let us denote by $\bM(z)^{1/2}$ its positive square 
root. 

\begin{defn}\label{d:A_2}
Let $\bM$ and $\bN$ be operator-valued measures on $\R$ (both with values in $B_+(\cK)$). 
We say that the pair $\bM$, $\bN$ satisfies the joint (Poisson) operator $\bA_2$ condition (and 
write $(\bM, \bN) \in (\bA_2) $) if 
\begin{align}\label{e:A_2}
\sup_{z\in\C\setminus\R}\| \bM(z)^{1/2} \bN(z)^{1/2} \| =: [\bM,\bN]\ci{\bA_2} <\infty.
\end{align}
The quantity $ [\bM,\bN]\ci{\bA_2} $ is called the (joint Poisson) $\bA_2$ condition of the 
pair $\bM$, $\bN$. In this definition, the supremum is taken over the operator norm $\|\cdot\|$ on 
$L^2(\cK)$.
\end{defn}

The Poisson $\bA_2$ condition was introduced in another connection in \cite{TreilVolberg97}. There,
among other things, it was shown to be equivalent to a more standard $\bA_2$ condition written in
terms of averages over intervals, whence in what follows we will skip the word ``Poisson''.

Let $\bM$ and $\bM^\Gamma$ be the compressed spectral measures of the operators $A$ and $A\ci\Gamma 
= A + \bB \Gamma\bB^*$ as described in the beginning of this section. 

\begin{theo}\label{t:A_2}
The measures $\bM$ and $\bN:=\Gamma \bM^\Gamma \Gamma$ satisfy the operator $\bA_2$ condition 
\eqref{e:A_2} with $[ \bM, \Gamma \bM^\Gamma \Gamma]\ci{\bA_2} \le 1/\pi$. 
\end{theo}

\begin{rem}
In light of Lemma \ref{l:A_2 scaling} below  the above theorem can be restated as 
$(\Gamma \bM \Gamma , \bM^\Gamma )\in(\bA_2)$ with 
$[\Gamma \bM \Gamma , \bM^\Gamma ]\ci{\bA_2} \le 1/\pi$. 
\end{rem}

\begin{rem}
The constant $1/\pi$ is an improvement over the estimate of the $\bA_2$ characteristic  
$8/\pi$ obtained in \cite{JST}.
\end{rem}

\begin{proof}[Proof of Theorem \ref{t:A_2}]
For $ z \in \C_+$ consider the Blaschke factor $b_z$, $b_z(s) = \frac{s -z}{s -\bar z}$, and let
$\cM\ci{b_{z}}$ be the multiplication operator by ${b_{z}}$. We begin by showing that for $f$ being
a linear combination of functions $K_{ z_k}$ with vector coefficients,
\begin{align*}
f = \sum_{k=1}^N K_{ z_k}\be_k, \qquad  z_k\in \C\setminus \left(\R \cup 
\left\{ z\right\}\right)
\end{align*}
we have 
\begin{align}\label{e-difference}
\left(V\ci\Gamma  - \cM\ci{b_{z}} V\ci\Gamma \cM\ci{b_{\bar{z}}}\right) f(s)
=
-\frac{2\im {z}}{s-\bar{z}}\Gamma\int \bM(\dd t)\frac{f(t)}{t-{z}}.
\end{align}
Here, observe that for ${z}\in \C_+,$ the term $t-{z}$ in the denominator of the 
integral stays bounded away from zero, so we are not dealing with the complication of a 
singular integral.

To obtain \eqref{e-difference}, first note that by partial fraction decomposition the function 
$\cM\ci{b_{\bar{z}}}f =b_{\bar{z}}f$ also a linear combination  of functions $K_{{z}_k}$ and 
$K_{{z}}$ with some vector coefficients. So, the expression for $V\ci\Gamma$ in $V\ci\Gamma 
\cM\ci{b_{\bar{z}}}f$ is covered 
by the spectral representation formula \eqref{e:repr01} in Theorem \ref{t-repr}. Next, using that 
$b_{z}(s)b_{\bar{z}}(s) = 1$ it is not hard to see that in the difference \eqref{e-difference}, the 
terms not containing $f(t)$ will cancel. In the remaining terms, i.e.~the ones that contain $f(t)$, 
we then simplify
\[
\left(V\ci\Gamma  - \cM\ci{b_{z}} V\ci\Gamma \cM\ci{b_{\bar{z}}}\right) f(s)=-\Gamma\int \bM(\dd 
t)\frac{f(t) }{t-s}  \left[1-  \frac{s -{z}}{s -\bar{z}} \frac{t -\bar{z}}{t -{z}} \right]
\]
using the identity
\[
 1-\frac{s -{z}}{s -\bar{z}} \frac{t -\bar{z}}{t -{z}} 
 =
 \frac{2\im {z} (t-s)}{(s-\bar{z})(t-{z})}.
 \]
We have arrived at \eqref{e-difference}.

Applying \eqref{e-difference} to a function $g(t) =  K_{\bar{z}}(t) \be$, $\be\in \cK$ and noticing 
that for ${z}\in \C_+$
\begin{align}\label{e:K_la norm}
 \im {z} \int_\R \bM(\dd t) \frac{1}{t-{z}} \frac{1}{t-\bar{z}} = \int_\R \bM(\dd t) 
 \frac{\im{z}}{|t-{z}|^2} = \pi  %\sign(\im{z})
 \bM({z})\,   
\end{align}
we see that 
\begin{align}\label{e:V-BVB 02}
\left(V\ci\Gamma  - \cM\ci{b_{z}} V\ci\Gamma \cM\ci{b_{\bar{z}}}\right) g = -2 K_{\bar{z}} \Gamma 
\pi\bM({z}) \be  .  
\end{align}
Therefore using the second equality in \eqref{e:K_la norm}, we see that for any $\bx\in \cK$
\[
\|K_{z} \bx\|\ci{L^2(\bM)}^2
=
\frac{\pi}{\im {z}}\|\bM({z})^{1/2}\bx\|^2 .
\]
Applying the above identity with $\bM$ replaced by $\bM^\Gamma$ and $\bx = \Gamma\bM({z}) \be$ to 
\eqref{e:V-BVB 02} we get 
\begin{align} \notag
\ \qquad \left\|\left(V\ci\Gamma  - \cM\ci{b_{z}} V\ci\Gamma \cM\ci{b_{\bar{z}}}\right) 
g\right\|^2\ci{L^2(\bM^\Gamma)}
& =
\frac{4\pi^{3}}{\im{z} }\langle
\bM^\Gamma({z})\Gamma\, \bM({z})\be\, , \, \Gamma \bM({z})\be\rangle\ci{\cK}
\\  \label{e-normdifference}
& =
\frac{4\pi^{3}}{\im {z}} \langle
\left(\Gamma\bM^\Gamma({z})\Gamma\right)
\bM({z})\be\, , \,
\bM({z})\be \rangle\ci{\cK} 
\\  \notag
& =
\frac{4\pi^{3}}{\im {z}} \left\|
\left( \bM^\Gamma \, ({z})\Gamma\right)^{1/2}
\left( \bM({z})\right)^{1/2}\wt\be
\right\|^2\ci{\cK}    
\end{align}
    where $\wt\be:=\bM({z})^{1/2}\be.$

Notice that 
\[
\left\|g\right\|\ci{L^2(\bM)}
=
\left(\frac{\pi}{\im {z}}\right)^{1/2}\|\bM({z})^{1/2}\be\|\ci\cK 
=
\left(\frac{\pi}{\im {z}}\right)^{1/2}\|\wt\be\|\ci\cK.
\]
Since $V\ci\Gamma$, $\cM_{b_{z}}$, $\cM_{b_{\bar{z}}}$ are unitary, we have
\begin{align*}
\left\|\left(V\ci\Gamma  - \cM\ci{b_{z}} V\ci\Gamma \cM\ci{b_{\bar{z}}}\right) 
g\right\|\ci{L^2(\bM^\Gamma)} 
\le 2    \left\|g\right\|\ci{L^2(\bM)}  = 2 \left(\frac{\pi}{\im {z}}\right)^{1/2}\|\wt\be\|\ci\cK 
. 
\end{align*}
Thus equation \eqref{e-normdifference} implies that 
\begin{align*}
\frac{2\pi^{3/2}}{(\im {z})^{1/2}} \left\| \left(\Gamma \bM^\Gamma({z})\Gamma\right)^{1/2} \left( 
\bM({z})\right)^{1/2}\wt\be\right\|\ci{\cK}   
\le 
2 \left(\frac{\pi}{\im {z}}\right)^{1/2}\|\wt\be\|\ci\cK
\end{align*}  
for all $\wt\be \in \Ran \bM({z})^{1/2}$. Since $\Ran \bM({z})^{1/2}$ is dense in $(\Ker
\bM({z})^{1/2})^\perp$, this implies
\begin{align*}
\pi \left\| \left(\Gamma \bM^\Gamma({z})\Gamma\right)^{1/2} \left( 
\bM({z})\right)^{1/2}\right\|\ci{\cK} \le   1, 
\end{align*}
which is the conclusion of the theorem. 
\end{proof}

For a simple reformulation, we use the following result. \begin{lem}\label{l-operatorswap}
Let $S$, $T$ and $\Gamma$ be self-adjoint operators with $S$ and $T$ positive semi-definite. 
Then    
\begin{align*}
\|S^{1/2} (\Gamma T\Gamma )^{1/2}\| = (\Gamma S \Gamma)^{1/2} T ^{1/2}\|.
\end{align*}
\end{lem}

\begin{proof}
From the general operator identity
\begin{align}\label{e-opid}
\|A\|^2 = \|AA^*\| = \|A^*A\|
\end{align}
we see that
\begin{align*}
\|S^{1/2} (\Gamma T\Gamma )^{1/2}\|^2
&=
\|S^{1/2} \Gamma T\Gamma S^{1/2}\|
\\&=
\|S^{1/2} \Gamma T^{1/2}\|^2  =  \|T^{1/2} \Gamma S^{1/2}\|^2   .
\end{align*}
Thus   $\|S^{1/2} (\Gamma T\Gamma )^{1/2}\|$  is symmetric in $S$ and $T$ and the lemma follows.
\end{proof}

With this and the definition of the $\bA_2$ condition, we see the following.
\begin{lem}\label{l:A_2 scaling}
Let $\bM$ and $\bN$ be $B_+(\cK)$-valued measures, and let $\Gamma=\Gamma^*\in B(\cK)$. If  
$(\bM, \Gamma\bN\Gamma ) \in (\bA_2)$ then $(\Gamma\bM \Gamma, \bN ) \in (\bA_2)$, and, moreover 
\begin{align*}
[ \bM, \Gamma\bN\Gamma ]\ci{\bA_2} = [ \Gamma\bM \Gamma, \bN ]\ci{\bA_2}.
\end{align*}
\end{lem}

\subsection{Monotonicity of the $\bA_2$ condition}

\begin{lem}\label{l:A_2 monotone}
Let $\bO\le \wt\bM\le\bM$, $\bO\le \wt\bN\le\bN$ be operator-valued measures. If $(\bM, \bN)\in 
(\bA_2)$, then $(\wt\bM, \wt\bN)\in (\bA_2)$ and 
\begin{align*}
[\wt\bM, \wt\bN]\ci{\bA_2} \le [\bM, \bN]\ci{\bA_2}  . 
\end{align*}
\end{lem}

\begin{proof}
For fixed $z\in \C\backslash \R$ let $\wt S = \wt\bM(z)$, $S = \bM(z)$, $\wt T = \wt\bN(z)$, and $T 
= \bN(z)$. Then by the assumption we have $\bO\le \wt S\le S$ and $\bO\le \wt T\le T$.

Let $\left\|S^{1/2}T^{1/2}\right\|\le C.$ Using the trivial fact  that $\|A\|\le C$ if and only if 
$A^*A\le C^2\bI$, we conclude that $\left\|S^{1/2}T^{1/2}\right\|\le C$ is equivalent to 
\[
S^{1/2} T S^{1/2} \le C^2 \bI. 
\]
Since $\wt T\le T$, we conclude that 
\[
S^{1/2} \wt T S^{1/2} \le C^2 \bI, 
\]
so $\left\|\wt T^{1/2} S^{1/2}\right\|\le C$. 

Taking the adjoint we get that $\left\| S^{1/2}\wt T^{1/2}\right\|\le C$, and applying the same 
reasoning to $S$, we get that $\left\|\wt S^{1/2}\wt T^{1/2}\right\|\le C$. 
\end{proof}

%%%%%%%%%%%%%%%%%%%%%%%%%%%%%%%%%%%%%%%%%%
\section{$\bA_2$ condition implies vector mutual singularity}\label{s-VMU}
%%%%%%%%%%%%%%%%%%%%%%%%%%%%%%%%%%%%%%%%%%
\subsection{Vector mutual singularity}
First let us notice that if an operator-valued measure $\bM$ with values in $B_+(\cK)$ is a weak 
indefinite integral, i.e.~$\bM(\dd s) = W(s) \mu(\dd s)$, then for any bounded  measurable 
operator-valued function $F$ with values in $B(\cK)$  the measure $F^*\bM F$ is well 
defined as the indefinite integral  $F^*W F \mu$ (since the product of measurable 
operator-valued functions is measurable). 

\begin{defn}
Let operator-valued measures $\bM$ and $\bN$ be both weak indefinite integrals. We say that $\bM$
and $\bN$ are \emph{vector mutually singular} (and write $\bM\perp\bN$) if there exists a measurable
operator-valued function $\Pi$ whose values are orthogonal projections in $\cK$ such that
\begin{align*}
\Pi \bM\Pi =\bO, \qquad (\bI-\Pi)\bN (\bI-\Pi) =\bO. 
\end{align*}
\end{defn}

Note, that this definition works only for measures that are indefinite integrals, since we did not 
define the measure $\Pi \bM\Pi$ for general operator-valued measures  (and it is probably 
impossible). 

\subsubsection{A restatement of the vector mutual singularity condition}\label{s:VMS density} 
Let $\bM=W\mu$ and $\bN=V\nu$. Then, denoting $\sigma:=\mu+\nu$ one can rewrite $\bM =\wt W\sigma$,
$\bN=\wt V\sigma$, where
\begin{align*}
\wt W = \left( \frac{\dd\mu}{\dd \sigma}  \right) W, \qquad 
\wt V = \left( \frac{\dd\nu}{\dd \sigma}  \right) V; 
\end{align*}
here $\dd \mu/\dd \sigma$ and $\dd \nu/\dd \sigma$ are the corresponding Radon--Nikodym derivatives.
Then it is easy to see that the condition $\bM\perp\bN$ is equivalent to the condition
\begin{align*}
\Ran \wt W(s) \perp \Ran \wt V(s) \qquad \sigma\text{-a.e.}
\end{align*}
We leave  the  details as an exercise for the reader. The critical fact here is the measurability of
the functions $s\mapsto P_{\cRan \wt W(s)}$, $s\mapsto P_{\cRan \wt V(s)}$, which follows
immediately from Lemma \ref{l:meas proj}

Of course, the above condition is independent of the choice of bases $\mu$ and $\nu$.

\subsection{Lebesgue decomposition of spectral measures}\label{s-decompo}
Let us recall the Lebesgue decomposition of a spectral measure into absolutely continuous and 
singular parts. 

For a self-adjoint operator $A = A^*$, let $\bE$ denote its projection-valued spectral measure. 
Take a \emph{scalar} spectral measure $\mu=\mu_{\bx}$  of maximal spectral type (as guaranteed by 
Corollary \ref{c:max type}), and  consider the Lebesgue decomposition of $\mu = 
\mu\ti{ac}+\mu\ti{s},$  where $\mu\ti{ac}$ and $\mu\ti{s}$ are absolutely continuous and singular 
parts of $\mu$ respectively.  

Take a set $S$  of zero Lebesgue measure such that $\mu\ti{s}(S^c) = 0.$
Then we  define the singular and absolutely continuous  parts of $\bE$ by 
\begin{align}
    \bE\ti{s} := \bE\mid\ci{S}
    \quad\text{and}\quad
    \bE\ti{ac} := \bE\mid\ci{S^c}.
\end{align}
so $\bE=\bE\ti{s} +\bE\ti{ac}$. It is not hard to see that the result does not depend on the choice 
of the measure $\mu_{\bx}$ and the set $S$.

For the compressed spectral measure $\bM=\bB^*\bE\bB$ we define $\bM\ti{s} = \bB^*\bE\ti{s}\bB$ as 
well as $\bM\ti{ac} = \bB^*\bE\ti{ac}\bB$ 
so that we obtain the decomposition 
\[
\bM=\bM\ti{ac}+\bM\ti{s}.
\]

\subsection{Vector mutual singularity for perturbations}

We are now prepared to state one of the key results of the paper. Its proof will be given in 
Section \ref{subse:proof}.

Consider operators $A$ and $A\ci\Gamma= A + \bB\Gamma\bB^*$. Let $\bE$ and $\bE\ci\Gamma$ be the 
projection-valued spectral measures for $A$ and $A\ci\Gamma$, and $\bM=\bB^*\bE\bB$ and 
$\bM^\Gamma= \bB^*\bE\ci\Gamma \bB$ be the corresponding compressed spectral measures.

\begin{theo}\label{t:VMS}
Let $A$ and $A\ci\Gamma = A + \bB\Gamma\bB^*$ be such that the corresponding compressed spectral 
measures $\bM $ and $\bM^\Gamma$ are weak indefinite integrals. Then their singular parts $\bM\ti 
s$ and $\bM^\Gamma\ti s$ satisfy the following vector mutual singularity condition
\begin{align*}
\bM\ti s\perp \Gamma \bM^\Gamma\ti s\Gamma\qquad\text{or, equivalently,} \qquad 
\Gamma \bM\ti s\Gamma \perp \bM^\Gamma\ti s .
\end{align*}
\end{theo}

\begin{rem}
Note that by Lemma \ref{l:factorization} any self-adjoint trace class operator $K$ can be 
represented as $K=\bB \Gamma\bB^*$ with invertible $\Gamma$ and $\bB$ being Hilbert--Schmidt. Then 
the operators $\bM(\R) = \bB^* \bE(\R) \bB$, $\bM^\Gamma(\R)=\bB^* \bE\ci\Gamma(\R)\bB$ belong to 
the trace class, so by Lemma \ref{l:TraceClass_IndInt} the measures $\bM$ and $\bM^\Gamma$ are weak 
indefinite integrals. Thus, Theorem \ref{t:VMS} can be interpreted as a generalization of the 
classical Aronszajn--Donoghue theorem with mutual singularity replaced by the \emph{vector} mutual 
singularity. 

Moreover, by Lemma \ref{l:factorization}, any bounded self-adjoint perturbation $K$ can be 
represented as $K=\bB \Gamma\bB^*$, so this representation is not really a restriction;  the 
requirement that the compressed spectral measures are weak indefinite integrals is the real 
restriction here. 
\end{rem}

\subsection{Main Lemma and proof of Theorem \ref{t:VMS}}\label{subse:proof}
\begin{lem}\label{l:VMS}
Let $\sigma$ be a singular measure, and let the measures $\bM=W\sigma$ and $\bN=V\sigma$ satisfy 
the operator $\bA_2$ condition \eqref{e:A_2}. Then $\bM\perp\bN$, or equivalently 
\begin{align*}
\Ran W(s)\perp\Ran V(s) \qquad \sigma\text{-a.e.}
\end{align*}
\end{lem}

This lemma immediately implies Theorem \ref{t:VMS}. Indeed, we assumed that the measures are 
indefinite integrals, meaning that $\bM=W\mu$, $\bM^\Gamma=W^\Gamma\nu$, where $\mu$ and $\nu$ are 
some scalar measures. Then, as we discussed above in Section \ref{s:VMS density}, denoting 
$\sigma:=\mu+\nu$ we can write $\bM=\wt W\sigma$, $\bM^\Gamma= \wt W^\Gamma\sigma$. 
And as was explained in Section \ref{s-decompo}, we have the decomposition 
$\bM\ti{s} = \wt W\sigma\ti{s}$, $\bM^\Gamma\ti{s} = \wt W^\Gamma\sigma\ti{s}$. Therefore we obtain
\begin{align*}
\bM\ti s \le\bM, \qquad \bM^\Gamma\ti s \le\bM^\Gamma, \qquad
\Gamma \bM^\Gamma\ti s \Gamma \le  \Gamma \bM^\Gamma \Gamma .
\end{align*}

We know, see Theorem  \ref{t:A_2}, that  the measures $\bM$ and $\Gamma \bM^\Gamma \Gamma$ satisfy 
the operator $\bA_2$ condition.  Therefore, by the monotonicity of the $\bA_2$ condition, see Lemma 
\ref{l:A_2 monotone}, we conclude that the measures $\bM\ti s$ and $\Gamma \bM^\Gamma\ti s \Gamma$ 
also satisfy the operator $\bA_2$ condition. 

Then by Lemma \ref{l:VMS} with $\sigma=\sigma\ti s$ the measures $\bM\ti s$ and $\Gamma
\bM^\Gamma\ti s \Gamma$ are vector mutually singular, which is exactly the conclusion of Theorem
\ref{t:VMS}. \hfill\qed

\subsection{Proof of Lemma \ref{l:VMS}}

\subsubsection{Reduction to the case of uniformly bounded densities}

Define $\wt\bM =\wt W\sigma$, $\wt N=\wt V\sigma$, where $\wt W:=(\bI + W)^{-1}W$, $\wt V:=(\bI + 
V)^{-1}V$; note that by Lemma \ref{lem:measurable} \cond3 the functions $\wt W$, $\wt V$ are 
measurable. Note that trivially 
\begin{align*}
\wt W = W^{1/2} (\bI + W)^{-1} W^{1/2} \le W
\end{align*}
and similarly $\wt V\le V$ (the inequalities understood as pointwise inequalities for operators). 
Therefore $\wt\bM\le\bM$ and $\wt\bN\le \bN$, and so by the monotonicity of the $\bA_2$ condition, 
see Lemma \ref{l:A_2 monotone}, the measures $\wt\bM$, $\wt\bN$ satisfy the $\bA_2$ condition.

Note that since $|t/(1+t)|\le 1$ for $t\ge 0$ we conclude from the spectral theorem that $\|\wt 
W(t)\|\le1$, $\|\wt V(t)\|\le1$.  So if Lemma \ref{l:VMS} holds for uniformly bounded weights, we 
conclude that 
\begin{align*}
\Ran\wt W(s)\perp\Ran\wt V(s) \qquad \sigma\text{-a.e.}
\end{align*}
But $\Ran W(s) = \Ran \wt W(s)$, $\Ran V(s) = \Ran \wt V(s)$, so conclusion holds for arbitrary 
weights. 

Thus, it is sufficient to prove Lemma \ref{l:VMS} only for uniformly bounded weights.

\subsubsection{Proof for the uniformly bounded densities}
We first show that 
\begin{align}\label{e:density}
\frac{\bM(z)}{\sigma(z)}   \underset{z\to s\sphericalangle}{\longrightarrow} W(s) , 
\qquad\frac{\bN(z)}{\sigma(z)}   \underset{z\to s\sphericalangle}{\longrightarrow} V(s) \qquad 
\sigma\text{-a.e.}
\end{align}
in strong operator topology.

Indeed, for $\bx\in\cK$ the vector $\bM(z) \bx$ is the Poisson extension of the vector-valued 
measure $\sigma W\bx$,  i.e.~$\bM(z) \bx = (\sigma W\bx)(z) $. Then by Lemma \ref{l:P-L differ} 
\begin{align*}
\frac{\bM(z)}{\sigma(z)} \bx  \underset{z\to s\sphericalangle}{\longrightarrow} W(s)\bx , 
\qquad 
\sigma\text{-a.e.}
\end{align*}
Trivially, the $\sigma$-a.e.\ convergence holds for all $\bx$ in some \emph{countable} dense set. 
Since the weight $W$ is uniformly bounded, the operators $\bM(z)/\sigma(z)$ are uniformly bounded, 
and so by $\e/3$ Theorem we conclude that the convergence holds for all $\bx\in\cK$.

The same reasoning applies to $\bN$. 

Applying Lemma \ref{lem:strong} \cond3 with $\f$ defined by 
\begin{align*}
\f(s):= \begin{cases}
\sqrt{s} \qquad& s\ge 0 \\
0 & s<0
\end{cases}
\end{align*}
we conclude from \eqref{e:density} that 
\begin{align}\label{e:W^1/2}
\frac{\bM(z)^{1/2}}{\sigma(z)^{1/2}}   \underset{z\to s\sphericalangle}{\longrightarrow} W(s)^{1/2} 
, \qquad\frac{\bN(z)^{1/2}}{\sigma(z)^{1/2}}   \underset{z\to s\sphericalangle}{\longrightarrow} 
V(s)^{1/2} \qquad 
\sigma\text{-a.e.}
\end{align}
in the strong operator topology. 

Since $\sigma$ is singular, 
\begin{align}\label{e:singularity}
\sigma(z)   \underset{z\to s\sphericalangle}{\longrightarrow} \infty ,  \qquad 
\text{for }\sigma\text{-a.a. }s\in \R. 
\end{align}
Take $s$ such that \eqref{e:density} and \eqref{e:W^1/2} hold. Then for such $s$ using 
\eqref{e:W^1/2} and the fact that in the strong operator topology limit of a product is 
product of limits, we get that 
\begin{align*}
\| W(s)^{1/2}V(s)^{1/2}\| &=   \left\| 
\left( \lim_{z\to s\sphericalangle} \sigma(z)^{-1/2}\bM(z)^{1/2}\right)
\left( \lim_{z\to s\sphericalangle} \sigma(z)^{-1/2}\bN(z)^{1/2}\right) \right\| \\
 &=   \left\| 
 \lim_{z\to s\sphericalangle} \sigma(z)^{-1}\bM(z)^{1/2} \bN(z)^{1/2} \right\|
\\
& \le\limsup_{z\to s\sphericalangle} \sigma(z)^{-1} \left\| 
\bM(z)^{1/2}\bN(z)^{1/2}  \right\| . 
\end{align*}
Since $\left\| \bM(z)^{1/2}\bN(z)^{1/2}  \right\|\le C<\infty$ uniformly in $z\in\C_+$, condition
\eqref{e:singularity} implies that $W(s)^{1/2}V(s)^{1/2} =\bO$. Hence we have $W(s)V(s) =\bO$ ,
which implies $\Ran V(s)\subset \Ker W(s) = (\Ran W(s))^\perp$.  Lemma \ref{l:VMS} is proved.
\hfill\qed

%%%%%%%%%%%%%%%%%%%%%%%
\section{Vector mutual singularity implies Theorems  \ref{t:GA-D_trace}  and \ref{t:GA-D_02}}
\label{s:proofs gen A-D trace}
%%%%%%%%%%%%%%%%%%%%%
Formally we will prove a bit stronger results. Namely, Theorem \ref{t:GA-D_trace} follows from the 
theorem below. 

\begin{theo}\label{t:GA-D_trace 01}
Consider a family of operators 
\begin{align*}
A_t := A + t \bB \Gamma \bB^* , 
\end{align*}
where $A=A^*$ and $\Gamma = \Gamma^*\ge\bO$ is invertible. As usual, assume  that $\Ran \bB$ is 
cyclic for $A$.

Let $\sigma$ be a singular measure on $\R$. If for each $t\in\R$ the compressed spectral measure
$\bM_t = \bB^*\bE_t\bB$ of the operator $A_t$ is  a weak indefinite integral, then for all $t\in\R$,
except countably many, the scalar spectral measure \tup(of maximal spectral type\tup) of $A_t$ is
mutually singular with $\sigma$.
\end{theo}

To get Theorem \ref{t:GA-D_trace} from the above Theorem \ref{t:GA-D_trace 01} we use Lemma
\ref{l:factorization} to factorize $K=\bB\Gamma\bB^*$, where $\bB:\cK\to\cH$ is Hilbert--Schmidt
with $\Ran\bB$ cyclic for $A$, and $\Gamma=\Gamma^*\ge\bO$ is an invertible operator in $\cK$. Since
$\bB$ is a Hilbert--Schmidt, the operators $\bM_t(\R) = \bB^*\bE_t(\R)\bB$ belong to the trace
class, so by Lemma \ref{l:TraceClass_IndInt} the measures $\bM_t$ are weak indefinite integrals. So
applying Theorem \ref{t:GA-D_trace 01} we immediately get the conclusion of Theorem
\ref{t:GA-D_trace}.

\subsection{Proof of Theorem \ref{t:GA-D_trace 01}} \label{s:proof gen A-D 01}

In the space $L^2(\sigma, \cK)$ of $\cK$-valued functions let us introduce a new inner product
\begin{align*}
\Jap{f,g}_{\sigma, \Gamma} := \int \Jap{\Gamma f(s), g(s)}\ci\cK \sigma(\dd s).
\end{align*}
Trivially the corresponding norm  is equivalent to the standard norm in $L^2(\sigma, \cK)$.

Denote $A_t = A+t\bB\Gamma\bB^*$, $t\in\R$, and let $\bE_t$ be the projection-valued 
spectral measure for $A_t$ and $\bM_t=\bB^* \bE_t \bB$ be the corresponding compressed spectral 
measure.

Fix a complete system $\{\bx_k\}_{k\in\N}$ in  $\cK$ and a sequence $\{\alpha_k\}_{k\in\N}$,
$\alpha_k>0$ such that  $\sum_k \alpha_k \|\bx_k\|^2 <\infty$.   By Lemma \ref{l-cycAGamma}
$\Ran\bB$ is cyclic for all operators $A_t$,  so by Corollary \ref{c:max type} for any $t\in\R$ the
measure $\mu^t$
\begin{align*}
\mu^t(\fdot) :=\sum_{k\in\N} \alpha_k \Jap{\bM_t(\fdot) \bx_k, \bx_k}
\end{align*}
is a spectral measure of maximal type for $A_t$. Therefore by Proposition \ref{p:max type base} the
measure $\mu^t$ is a base of the measure $\bM_t$, so we can write
\begin{align*}
\bM_t = W_t \mu^t. 
\end{align*}

Let $E$ be the exceptional set, i.e.\ the collection of all points $t\in\R$ such that the measures
$\mu^t$ and $\sigma$ are not mutually singular. Then for any $t\in E$  the Radon--Nikodym derivative
$w^t=\dd\mu^t/\dd \sigma$ is a non-zero element of  $L^1(\sigma)$ (i.e.~it is positive on a set $E$,
$\sigma(E)>0$).

Take $t_1, t_2\in E$, $t_1 < t_2$. We can write  $A_{t_2} = A_{t_1} + (t_2-t_1) \bB \Gamma \bB^*$,
so Theorem \ref{t:A_2} implies that the measures $\bM_{t_1}$ and $\Gamma \bM_{t_2} \Gamma$ satisfy
the operator $\bA_2$ condition.

Denote by $ \bM_t^\sigma$ the measure $W_t w^t\sigma$; clearly $\bM_t^\sigma \ne\bO$ if and only if
$t\in E$. It is immediate from the definition that $\bM_t^\sigma\le \bM_t$. Therefore, by
monotonicity of the $\bA_2$ condition (Lemma \ref{l:A_2 monotone}) the measures $\bM_{t_1}^\sigma$
and $\Gamma \bM_{t_2}^\sigma \Gamma$ also satisfy the operator $\bA_2$ condition. Therefore by Lemma
\ref{l:VMS} (recall that $\sigma$ is a singular measure)
\begin{align*}
\Ran\left( w^{t_1}(s) W_{t_1}(s)\right) \perp  \Ran \left(w^{t_2}(s) \Gamma W_{t_2}(s) 
\Gamma\right) \qquad \sigma\text{-a.e.}
\end{align*}
Using the symbol $\perp\ci\Gamma$ for the orthogonality in the inner product $\Jap{\fdot, 
\fdot}\ci\Gamma$ on $\cK$, 
\begin{align*}
\Jap{\bx, \by}\ci\Gamma = \Jap{\Gamma \bx, \by}
\end{align*}
we can rewrite this as 
\begin{align}\label{e:loc orthog 01}
\Ran\left( w^{t_1}(s) W_{t_1}(s)\right) \perp\ci\Gamma  \Ran \left(w^{t_2}(s)  W_{t_2}(s) 
\right) \qquad \sigma\text{-a.e.}
\end{align}

For each $t\in E$ we can find a non-zero $f_t\in L^2(\sigma, \cK)$ such that 
\begin{align}\label{e:mu^t}
f_t(s) \in \cRan \left( w^t(s) W_t(s)\right)\qquad \sigma\text{-a.e.}
\end{align}
Indeed, at least one of the measures $\Jap{\bM(\fdot)\bx_k, \bx_k}\ci\cK$ from \eqref{e:mu^t} must 
have a non-trivial absolutely continuous part with respect to $\sigma$, so $w^t(s)\Jap{W_t(s)\bx_k, 
\bx_k}\ci\cK>0$ on a set $E$, $\sigma(E)>0$, so considering the function $s\mapsto W_t(s)\bx_k$ and 
multiplying it by a suitable non-vanishing scalar function does the trick. 

Normalizing, we can always assume without loss of generality that $\| f_t\|\ci{\sigma, \Gamma}=1$. 
Take  $t_1, t_2\in E$, $t_1\ne t_2$.  The local orthogonality condition \eqref{e:loc orthog 01} 
implies that $\Jap{f_{t_1},f_{t_2}}\ci{\sigma, \Gamma}=0$, so 
\begin{align}\label{e:dist 01}
\| f_{t_1} -  f_{t_2} \|\ci{\sigma, \Gamma} =\sqrt2 \qquad \forall t_1, t_2\in E, \ t_1\ne t_2.
\end{align}
The space $L^2(\sigma, \cK)$ is separable, so the condition \eqref{e:dist 01} implies that $E$ must 
be countable (finite or infinite). \hfill\qed

\subsection{Proof of Theorem \ref{t:GA-D_02}}

Again, we prove a bit stronger result. Namely, Theorem \ref{t:GA-D_02} is a consequence of the 
following result
\begin{theo}\label{t:GA-D_02 01}
Let $A_t= A+ \bB \Gamma(t)\bB^*$, where  $t\mapsto \Gamma(t)$ is a $C^1$ map such that $\Gamma'(t)$
is positive definite and invertible for all $t\in I$. Assume that for each $t$ the compressed
spectral measure $\bM_t = \bB^* \bE_t \bB$ is an indefinite integral, and let $\sigma$ be an
arbitrary singular measure.

If $\Ran \bB$ is cyclic for $A$, then for all $t\in I$ except maybe countably many, the spectral
measure of $A_t$ is mutually singular with $\sigma$.
\end{theo}
As in Theorem \ref{t:GA-D_02}, the spectral measure of $A_t$ means the \emph{scalar spectral 
measure of maximal spectral type}. 

To prove this theorem it is sufficient to prove its \emph{local} version, i.e.\ that for every 
$t_0\in I$ there exists an open neighborhood $U\subset I$ of $t_0$ such that the spectral measure 
of $A_t$ is mutually singular  with $\sigma$ for all $t\in U$ except maybe countably many.  

Fix $t_0\in I$, and denote $R_0:= \Gamma'(t_0)$. The function $\Phi$, $\Phi(t) := 
R_0^{-1/2}\Gamma(t) R_0^{-1/2}$ is continuously differentiable, $\Phi'(t_0)=\bI$, so there exist an 
open neighborhood $U\ni t_0$ such that 
\begin{align}\label{e:Phi'-I}
\|R_0^{-1/2}\Gamma'(t)R_0^{-1/2} -\bI\| <1/2 \qquad \forall t\in U. 
\end{align}

Take $t_1, t_2\in U$, $t_1<t_2$, Denoting $R:=\Gamma(t_2)- \Gamma(t_1)$, $\tau:=t_2-t_1$ and
applying the mean value estimates to the function $t\mapsto \Psi(t) := R_0^{-1/2}\Gamma
(t)R_0^{-1/2} - t\bI $, we get from the above estimate \eqref{e:Phi'-I}
\begin{align*}
\|R_0^{-1/2} R R_0^{-1/2} - \tau\bI\| <\tau/2.
\end{align*}
Substituting vector $R_0^{1/2}\bx$, $\bx\in\cK$ into the quadratic form of the operator 
$R_0^{-1/2} R R_0^{-1/2} - \tau \bI$  we get that 
\begin{align*}
\left|\Jap{R\bx,\bx} - \tau\Jap{R_0\bx,\bx}  \right| 
\le \frac12 \tau \Jap{R_0 \bx, \bx} .
\end{align*}
Therefore 
\begin{align}\label{e:norm equiv}
\frac12 \Jap{\tau R_0 \bx, \bx} \le \Jap{ R  \bx, \bx} \le \frac32 \Jap{\tau R_0 \bx, \bx} ;
\end{align}
in particular this always implies that $R=\Gamma(t_2)- \Gamma(t_1)$, $t_1<t_2$ is a non-negative 
invertible operator. 

The rest of the proof now follows the proof of Theorem \ref{t:GA-D_trace 01}. For an invertible 
non-negative operator $R$ on $\cK$ we define the inner products $\Jap{\fdot, \fdot}\ci R$ and 
$\Jap{\fdot, \fdot}\ci{\sigma, R}$, on $\cK$ and $L^2(\sigma, \cK)$ respectively, 
\begin{align*}
\Jap{\bx, \by}\ci R & := \Jap{R\bx, \by}, \\
\Jap{f,g}\ci{\sigma, R} & := \int_\R \Jap{R f(s), g(s)}\sigma(\dd s).
\end{align*}
Note that the corresponding norms are equivalent to the standard norms in the respective spaces.

As in the proof of Theorem \ref{t:GA-D_trace 01} denote by  $E$ the exceptional set, i.e.\ the
collection of all points $t\in\R$ such that the measures $\mu^t$ and $\sigma$ are not mutually
singular. As it was discussed above for any $t\in E$  the Radon--Nikodym derivative
$w^t=\dd\mu^t/\dd \sigma$ is a non-zero element of  $L^1(\sigma)$ (i.e.~positive on a set $E$,
$\sigma(E)>0$).

Let now $t_1, t_2\in E$, $t_1<t_2$ be fixed. Denote $\tau:=t_2-t_1$, $R:= \Gamma(t_2) -\Gamma(t_1)$.
Then $A_{t_2} = A_{t_1} + \bB R \bB^*$, so by Theorem \ref{t:A_2}   the measures $\bM_{t_1}$ and
$R\bM_{t_2} R$ satisfy the operator $\bA_2$ condition.

Again, for $t\in E$ denote by $\bM_t^\sigma$ the measure $w^t W_t\sigma$. Clearly 
$\bM_t^\sigma\le \bM_t$, so by the monotonicity  of $\bA_2$ condition (Lemma \ref{l:A_2 monotone})
the measures $\bM_{t_1}^\sigma$ and $R \bM_{t_2}^\sigma R$ also satisfy the operator $\bA_2$
condition. Therefore by Lemma \ref{l:VMS} (recall that $\sigma$ is a singular measure)
\begin{align*}
\Ran\left( w^{t_1}(s) W_{t_1}(s)\right) \perp  \Ran \left(w^{t_2}(s) R W_{t_2}(s) 
R\right) \qquad \sigma\text{-a.e.}
\end{align*}
Denoting by $\perp\ci{R}$ the orthogonality in the inner product $\Jap{\fdot, \fdot}\ci{R}$, we can 
rewrite this as
\begin{align}\label{e:orthog 02}
\Ran\left( w^{t_1}(s) W_{t_1}(s)\right) \perp\ci{R}  \Ran \left(w^{t_2}(s)  W_{t_2}(s) 
\right) \qquad \sigma\text{-a.e.}
\end{align}

As before, see explanation after \eqref{e:mu^t}, for each $t\in E$ we can find a non-zero $f_t\in 
L^2(\sigma, \cK)$ such that 
\begin{align*}
f_t(s) \in \Ran \left( w^t(s) W_t(s)\right)\qquad \sigma\text{-a.e.}
\end{align*}
Normalizing, we can always assume without loss of generality that $\| f_t\|\ci{\sigma, R_0}=1$.

For our  $t_1, t_2\in E$ the local orthogonality condition \eqref{e:orthog 02} implies that 
$\Jap{f_{t_2}, f_{t_2}}\ci{R}=0$, so 
\begin{align*}
\| f_{t_2} - f_{t_2}\|\ci{\sigma, R}^2 =  \| f_{t_1}\|\ci{\sigma, R}^2 + \| f_{t_2}\|\ci{\sigma, 
R}^2. 
\end{align*}
Using the equivalence of norms \eqref{e:norm equiv} we get that 
\begin{align*}
 \|f_{t_1}- f_{t_2}\|^2\ci{\sigma,R_0} &\ge \frac2{3\tau} \| f_{t_1} - f_{t_2}\|\ci{\sigma, R}^2 = 
\frac2{3\tau}\left( \| f_{t_1}\|\ci{\sigma, R}^2 + \| f_{t_2}\|\ci{\sigma, R}^2\right) \\
&\ge \frac{\tau}{2} \frac2{3\tau} \left( \| f_{t_1}\|\ci{\sigma, R_0}^2 + \| f_{t_2}\|\ci{\sigma, 
R_0}^2\right)
=   \frac23 , 
\end{align*}
so separability of $L^2(\sigma, \cK)$ implies that $E$ is at most countable. \hfill\qed

%%%%%%%%%%%%%%%%%%%%%%%%%%%%%%%%
\section{Exceptional sets: proof of Theorem \ref {t:many except dim}}
\label{s:exceptional}
%%%%%%%%%%%%%%%%%%%%%%%%%%%%%%%%
In fact we prove the following more general statement:

\begin{theo}\label{t:many except dim 01}
Let $A$, $K_1, \dots, K_n$, be self-adjoint operators. Consider a family of perturbations 
\begin{align}\label{e:pert 01}
A_{\vec t\,} := A+ K_{\vec t\,} , \qquad K_{\vec t\,}:=\sum_{k=1}^n t_k K_k, \qquad 
\vec t:= (t_1, \dots, t_n)\in \R^n . 
\end{align}

Let $\sigma$ be a singular Radon  on $\R$. Assume that there exists a set of directions $D =
D_\sigma\subset S^{n-1}\subset \R^n$ of positive surface measure, such that for any $\vec t_0\in 
\R^n$ and any $\vec \tau\in D$ the  singular parts
of spectral measures of operators
\begin{align*}
A_{\vec t_0 + s \vec \tau\,}, \qquad s\in \R
\end{align*}
are mutually singular with respect to $\sigma$ for all $s\in \R$ except maybe countably many.

Then  the spectral measure of $A_{\vec t}$ is mutually singular with $\sigma$ for all $\vec t\in
\R^n\setminus E$, where the exceptional set $E=E_\sigma$ is Borel and has Hausdorff dimension at 
most $n-1$.
\end{theo}

Let us explain why this theorem implies Theorem \ref{t:many except dim}. By the assumption of 
Theorem \ref{t:many except dim} there exists a subset $D\subset S^{n-1}$ of positive surface measure
such that for any $\vec \tau \in D$ the operator $K_{\vec \tau}$ is a non-negative trace class 
operator with cyclic range. By Theorem \ref{t:GA-D_trace} for any $\vec t_0\in \R^n$ and $s\in\R$ 
the singular part of the spectrum of the operator 
\begin{align*}
A_{\vec t_0 + s\vec \tau\,}  =A_{\vec t_0} + s K_{\vec \tau}
\end{align*}
is mutually singular with the singular measure measure $\sigma$  for all $s\in \R$ except maybe 
countably many. 
Thus the assumptions of Theorem \ref{t:many except dim 01} are satisfied.

\subsection{Exceptional set}
Denote by $\bE_{\vec t\,}$ the projection-valued spectral measure of $A_{\vec t\,}$. 

As in Section \ref{s-parameters}, fix a complete system $\{\bx_k\}_{k\in\N}$ in  $\cH$ and a
sequence $\{\alpha_k\}_{k\in\N}$, $\alpha_k>0$ such that  $\sum_k \alpha_k \|\bx_k\|^2 <\infty$.  By
Corollary \ref{c:max type} with $\bB=\bI$ and $A=A_{\vec t\,}$ we get that for any $\vec t\in\R^n$
the measure $\mu^{\vec t}$
\begin{align*}
\mu^{\vec t}(\fdot) :=\sum_{k\in\N} \alpha_k \Jap{\bE_{\vec t}\,(\fdot) \bx_k, \bx_k}
\end{align*}
is a scalar spectral measure of maximal type for $A_{\vec t}$. 

\begin{defn}
Let $A_{\vec t\,}$, $\vec t\,\in \R^n$ be some family of self-adjoint operators, and let $\mu^{\vec 
t\,}$ be the corresponding spectral measures as defined above. 
  
Given a singular measure $\mu$ define  the \emph{exceptional} set  $E=E_\sigma$ to be  the
collection of all $\vec t \in\R^n$ such that the measure $\mu^{\vec t}$ is not mutually singular
with $\sigma$.
\end{defn}

We will use a result from \cite{Mattila1975}, see Lemma \ref{l:Mattila--Marstrand 02} below. Before
stating it, let us introduce some notation. We will use $\cH^s$ for the $s$-dimensional Hausdorff
measure, $\dim $ is the Hausdorff dimension, $G(n,m)$ denotes the set of all $m$-dimensional
subspaces of $\R^n$, and $\gamma_{n,m}$ is the invariant measure on $G(n,m).$

While it was not stated explicitly, the following statement  was proved in \cite{Mattila1975}, see 
Lemma 6.4  and discussion in the beginning of p.~228 there. 

\begin{lem}\label{l:Mattila--Marstrand 02}  
Let $E$ be an analytic \tup(Suslin\tup) subset of $\R^n$ such that 
$\cH^s(E)>0$. Then 
\[
\dim( E\cap (V+x)) \ge s+m-n
\]
for $\cH^s\times\gamma_{n,m}$ almost all $(x,V)\in E\times G(n,m)$.
\end{lem}

We will not be giving the definition of analytic (Suslin) sets; for our purposes it is sufficient 
to know that every Borel set in $\R^n$ is analytic, cf.~\cite[Sec.~6.6]{Bogachev2007vol2}. 
Thus Theorem \ref{t:many except dim 01} follows from the next result.

\begin{theo}\label{t:Borel}
Consider the family \eqref{e:pert 01} of perturbations, assuming only $K_j=K_j^*$. 
Then the exceptional set $E$ defined above is Borel.  
\end{theo} 

Note that application of the above theorem significantly simplifies the proof of the main result of 
\cite{IMRN}.

\begin{rem}
As one can see from the proof of Theorem \ref{t:Borel}, the above theorem holds for more general
families of operators, not just for the affine ones given by \eqref{e:pert 01}.

For example, it holds for any family such that the map $\vec t\,\mapsto A_{\vec t\,}=A_{\vec t\,}^*$
is a continuous map from $\R^n$ (or its Borel subset) to $B(\cH)$ with strong operator topology.
\end{rem}

\subsection{Theorem \ref{t:Borel} implies Theorem \ref{t:many except dim 01}}

Assume that $\dim E>n-1$. So, there exists $s>n-1$ such that $\cH^s( E)>0$ (possibly $=\infty$).
Applying Lemma \ref{l:Mattila--Marstrand 02} with $m=1$ we get that for $\cH^s\times \gamma_{n,1}$
almost all lines
\[
\{ \vec t_0 + s\vec t: s\in\R  \}
\]
(i.e.~for $\cH^s\times \gamma_{n,1}$ almost all $(\vec t_0, \vec t \,) \in   E\times S^{n-1}$) their
intersection with $E$ has  Hausdorff dimension at least $s+1-n>0$. By Theorem \ref{t:Borel} for any
line with $\vec t_0\in   E$ and $\vec t \in D\subset S^{n-1}$ the intersection is countable (so has
dimension $0$). But  $\gamma_{n,1}(D)>0$ ($\gamma_{n,1}$ is just the normalized surface measure on
$S^{n-1}$), so
\[
\cH^s\times 
\gamma_{n,1}(  E\times U)>0,
\]
and we arrived at a contradiction. \hfill \qed

\subsection{Proof of Theorem \ref{t:Borel}}
For a measure $\mu$  let us, slightly abusing notation, denote by  $\mu(z)$ the Poisson extension of
$\mu$ to the point $z\in\C_+$,
\begin{align*}
\mu(z) =\int_{\R} p_z(s) \mu(\dd s), \qquad p_z(s) = \frac1\pi \im \frac{1}{s-z}. 
\end{align*}

It is well known (see e.g.~\cite[Part (ii) of Theorem 3.4]{JST}) that for a finite non-negative 
measure $\mu$ on $\R$ its singular part $\mu\ti s$ is \emph{carried} by the set $S_\mu$ of all 
$x\in\R$ for which 
\begin{align}\label{e:carrier}
\lim_{z\to x\sphericalangle} \mu(z) = + \infty;
\end{align}
here by limit we understand the non-tangential limit. The term \emph{carried} here means that
$\mu\ti s(\R\setminus S_\mu) =0$.

Note, that if we pick a \emph{reasonable} sequence  $y_n \downarrow 0$, for example $y_n = 2^{-n}$ 
(or $y_n= 1/n$), then it follows easily from Harnack's inequality that for any aperture of the 
approach region 
\[
\lim_{z\to x\sphericalangle} \mu(z) = + \infty \quad \text{if and only if } \quad \lim_{n\to 
\infty} \mu(x+i y_n) = + \infty;
\]
this equivalence holds for \emph{all} $x\in \R$. In particular, this means that the set $S_\mu$ is 
always a Borel set. 

Let us fix such a ``reasonable'' sequence $y_n$. Let $S_\sigma$ be the \emph{carrier} of the 
singular measure $\sigma$, namely the set of all $x\in\R$ such that \eqref{e:carrier} (with 
$\sigma$ instead $\mu$)  holds. Note that $|S_\sigma|=0$. 

Note that the measures $\sigma$ and $\mu^{\vec t\,}$ are not mutually singular if and only if in 
the Lebesgue--Radon--Nikodym decomposition 
\begin{align*}
\mu^{\vec t\,} = f_{\vec t\,}\sigma + \mu^{\vec t\,}\ti{s}, \qquad \mu^{\vec t\,}\ti{s}\perp \sigma
\end{align*}
the absolutely continuous part $f_{\vec t\,}\sigma$ is a non-zero measure. Note that $f_{\vec t\,}$ 
is the Radon--Nikodym derivative,  $f_{\vec t\,}=\frac{\dd \mu^{\vec t\,}}{\dd \sigma}$. Thus, 
$\vec t$ belongs to the exceptional set $E$ if and only if 
\begin{align}\label{e:except 03}
\int_{S_\sigma} f_{\vec t\,} \sigma = 
\int_{S_\sigma} \frac{\dd \mu^{\vec t\,}}{\dd \sigma} \sigma >0 .
\end{align}

Define the functions $\f_n, \psi_n, \Phi_n:\R^n \times  S_\sigma \to \R$,  
\begin{align*}
\f_n(\vec t, x):= \mu^{\vec t} (x+iy_n), \qquad \psi_n(\vec t, x):= \sigma (x+iy_n) , \qquad
\Phi_n =\f_n/\psi_n. 
\end{align*}
By Lemma  \ref{l:cont mu_t} the map $(\vec t, z) \mapsto \mu^{\vec t}(z)$, $(\vec t, z)\in \R^n 
\times \C_+$ is continuous, therefore the functions $\f_n$ are continuous. The functions $\psi_n$ 
are continuous and non-vanishing, so the functions $\Phi_n$ are continuous,  and so Borel 
measurable. 

Therefore the function $\Phi:\R^n\times S_\sigma\to\R$, 
\begin{align*}
\Phi(\vec t, x):= \lim_{n\to \infty} \Phi_n(\vec t\,, x+iy_n)
\end{align*}
(and defined to be $0$ if limit does not exist), is Borel measurable. 

By Lemma \ref{l:weak type 01} for any $\vec t\,\in\R^n$
\begin{align*}
\Phi(\vec t\, , x)= \lim_{n\to\infty} \frac{\mu^{\vec t\,} (x+iy_n)}{\sigma(x+iy_n)} 
= f_{\vec t\,},   \qquad \sigma\text{-a.e.}    
\end{align*}
As we discussed above, $\mu^{\vec t\,}\perp\sigma$ if and only if \eqref{e:except 03} holds, so the 
exceptional set $E$ is exactly the set of all $\vec t\,\in\R^n$ such that 
\begin{align*}
\int_{S_\sigma} \Phi(\vec t\, , x) \sigma(\dd x) >0 .
\end{align*}
By Tonelli's theorem this set is measurable. \hfill\qed

\appendix

%%%%%%%%%%%%%%%%%%%%%%%%%%%
\section{Approximation by Poisson kernels}\label{s-apppoiss}
%%%%%%%%%%%%%%%%%%%%%%%%%%%
We use upper half-plane limits and Poisson kernels $P_\delta(s) = 
\frac{1}{\pi}\frac{\delta}{s^2+\delta^2} = \frac{1}{\pi}\im\frac{1}{s-i\delta}$ to approximate 
functions $f$ with respect to the \emph{Poisson norm}
\[
\|f\|\ti{Poiss}
:=
\sup_{s\in\R}(1+s^2)\left|f_n(s)-f(s)\right|\to 0.
\]

\begin{lem}\label{le:poissonappro}
Assume that $f\in C_c(\R)$. Then there is a sequence $f_n$ of finite linear combinations of Poisson 
kernels, 
\[
f_n(s)
=
\sum_{k=1}^{N} a_k^n P_\delta(s-t_k^n),
\quad
N\in\N\text{ and }a_k^n, t_k^n\in\R 
\]
such that $\|f_n-f\|\ti{Poiss}\to 0$ as $n\to\infty$. 

Moreover, the linear combination is non-negative if $f$ is non-negative.
\end{lem}

\begin{proof}
Fix $f\in C_c(\R)$. By scaling  we can assume without loss of generality that $\supp f\subset 
[-1,1]$ and that $|f(s)|\le 1$ on $\R$. Take arbitrary $\e>0.$ To prove the lemma it suffices to 
find a linear combination $f_\e$ of Poisson kernels such that for all $s\in\R$
\begin{align}\label{e:f_e-f}
|f_\e(s) - f(s) |\cdot (1+s^2) \le \e . 
\end{align}

Trivially on $[-2,2]$ we have $P_\delta * f \rightrightarrows f$ as $\delta\searrow 0$, so   
\begin{align}\label{e:Poisson appr}
\|P_\delta*f-f\|\ci{C([-2,2])}<\e/10      
\end{align}
for all sufficiently small $\delta$. 
This immediately implies that for all $s\in[-2,2]$
\begin{align*}
|P_\delta*f (s) - f(s) |\cdot (1+s^2) \le \e/2. 
\end{align*}
Approximating  in uniform norm the convolution $P_\delta*f$ by linear combinations of Poisson 
kernels (the details are provided below) we get the estimate \eqref{e:Poisson appr} for 
$s\in[-2,2]$. 

To get the estimate for $|s|\ge 2$ we will use the fact that for any complex-valued measure $\tau$
supported on $[-1,1]$ and $|s|\ge 2$ we have
\begin{align}\label{e-estconv}
|P_\delta*\tau(s)|
\le \frac{1}{\pi}\cdot \frac{\delta\|\tau\|}{(|s|-1)^2}\le 
\frac{1}{\pi}\cdot\frac{5\delta\|\tau\|}{1+s^2} 
< \frac{2\delta\|\tau\|}{1+s^2}\,;
\end{align}
here $\|\tau\|$ is the total variation of the measure $\tau$.   

To continue with the proof let us  fix $\delta<\e/8$ such that \eqref{e:Poisson appr} holds.

The approximation of $P_\delta*f$ by a linear combination of Poisson kernels is pretty obvious: we
just approximate the function $f$ by a finite linear combination of $\delta$-measures. Let us work
out the details.

The derivative of $P_\delta$ is uniformly bounded (remember that we fixed $\delta$), and so 
$P_\delta$ is uniformly continuous.  Therefore there exists $\alpha>0$ so that for $s,t\in\R$
\begin{align}\label{e-unifconv}
|s-t|<\alpha  \implies|P_\delta(s)-P_\delta(t)|<\e/20  .
\end{align}

Let us split the interval $[-2,2]$ into finitely many of disjoint intervals $I_k$, $|I_k|<\alpha$, 
and pick  points $t_k\in I_k$. Define the linear combination $f_\e$ of Poisson kernels as 
\begin{align*}
f_\e(s) = \sum_k a_k P_\delta (s-t_k), \qquad a_k:= \int_{I_k} f(s)\dd s. 
\end{align*}
Then for all $s\in\R$
\begin{align*}
 P_\delta*f (s)-f_\e(s)   & = \sum_k   \int_{I_k}  \bigl(  P_\delta (s-t) -P_\delta (t_k) \bigr) 
 f(t) \dd t
\intertext{so}
| P_\delta*f (s)-f_\e(s) |  & \le \sum_k   \int_{I_k}  \bigl|  P_\delta (s-t) -P_\delta (t_k) 
\bigr| \, |f(t)| \dd t \\
& \le \sum_k (\e/20 ) |I_k| =2\e/20 =\e/10.
\end{align*}
Combining this with \eqref{e:Poisson appr} we get that for all $s\in[-2,2]$
\begin{align*}
|f_\e(s) - f(s)| (1+s^2) \le 5\cdot(\e/10+\e/10) =\e. 
\end{align*}
For $|s|>2$ we will use estimate \eqref{e-estconv}. Define the measure $\tau$ by $\tau:=\sum_k 
a_k\delta_{t_k}$. 
Clearly 
\begin{align}\label{e-tau}
\|\tau\|\le \sum_k \int_{I_k} |f(t)|\dd t = \int_{-2}^2 |f(t)|\dd t \le 2,  
\end{align}
and $f_\e=P_\delta* \tau$. 
Since $\supp f\subset [-1,1]$ we get that for $|s|\ge 2$
\begin{align*}
|f_\e(s) - f(s)| =|f_\e(s)| =|P_\delta*\tau (s)|\le
\frac{2\delta \|\tau\|}{1+s^2}  \le \frac{4\delta }{1+s^2}
<\frac{\e}{2}\cdot \frac{1}{1+s^2} ;
\end{align*}
the first inequality here follows from \eqref{e-estconv}, the second one from \eqref{e-tau}, and 
the last one because we choose $\delta<\e/8$. 
\end{proof}

%%%%%%%%%%%%%%%%%%%%%%%%%%%%%%%%%%
\section{Strong continuity and measurability}\label{ap:measurability}
%%%%%%%%%%%%%%%%%%%%%%%%%%%%%%%%%%
We start by showing that  sequential convergence in strong operator topology is preserved after
taking square roots. Our proof relies on part (iii)  of the following lemma. Below $A, A_n, B,
B_n,$:s are operators on a fixed separable (complex) Hilbert space $\cH$ and $P_n$ is the orthogonal
projection onto the span of first $n$ vectors in a fixed orthonormal basis of $\cH$. The symbol
$A_n\stackrel{\rm s}{\longrightarrow} A$ means convergence in the strong operator topology.
 
\begin{lem}\label{lem:strong} \  
\begin{enumerate}
\item If $A_n\stackrel{\textup{s}}{\longrightarrow} A$ and $B_n\stackrel{\textup s}{\longrightarrow}
B$, then $A_nB_n \stackrel{\textup{s}}{\longrightarrow} AB$.

\item For any $A$ we have $P_nAP_n \stackrel{\textup{s}}{\longrightarrow} A$ as  $n\to\infty$.

\item If $A_n=A_n^*$ and $A_n\stackrel{\textup{s}}{\longrightarrow} A$, then for any continuous 
function 
$\f:\R\to \C$
\[
\f(A_n) \stackrel{\textup{s}}{\longrightarrow} \f(A). 
\]
\end{enumerate}
\end{lem}

\begin{proof}
Statement (i) follows by using the uniform bounds (from the Uniform Boundedness Principle) for
$\|A_n\|$, $\|B_n\|$, and decomposing $A_nB_n -AB = A_n(B_n-B)+(A_n-A)B.$ Since $P_n
\stackrel{\textup{s}}{\longrightarrow} \bI$,  statement (ii) follows from part (i).

To prove \cond3 we first notice that statement \cond1 implies that for any polynomial $p$ and any 
$A_n \stackrel{\textup{s}}{\longrightarrow} A$ ($A_n$s are not necessarily self-adjoint here) 
\begin{align}\label{e:sc 01}
p(A_n) \stackrel{\textup{s}}{\longrightarrow} p(A).   
\end{align}
The rest of the proof follows by the standard $\e/3$ reasoning. Namely,  by the Uniform Boundedness 
Principle we have that $\|A_n\|, \|A\|\le M<\infty$ for all $n$, so we get the inclusion of the 
spectra, $\sigma(A_n), \sigma(A) \subset [-M,M]$. Denote by $I$ the interval $[-M,M]$. Then for any 
$f\in C(I)$
\begin{align*}
\| f(A_n)\|, \|f(A)\|\le \|f\|\ci{C(I)}
\end{align*}
(norm of a normal operator $=$ spectral radius). 

Pick arbitrary $\e>0$ and $\bx\in \cH$, $\|\bx\|=1$. By the Stone--Weierstrass Theorem we can find 
a polynomial $p$ such that $\|\f-p\|\ci{C(I)}<\e/3$. Using triangle inequality we can estimate
\begin{align*}
\| \f(A_n)\bx - \f(A)\bx\| 
&\le \| \left(\f(A_n) - p(A_n)\right)\bx\| + \|  \left( p(A_n) -p(A)  \right)\bx \|  \\ 
&\qquad + \| \left(p(A) - \f(A)\right)\bx \|  \\
& < 2\e/3 +  \|  \left( p(A_n) -p(A)  \right)\bx \|  . 
\end{align*}
From \eqref{e:sc 01} we get that there exist $N=N(\bx)<\infty$ such that for all $n>N$ there holds  
$\|  \left( p(A_n) -p(A)  \right)\bx \|<\e/3$. Thus 
\begin{align*}
\| \f(A_n)\bx - \f(A)\bx\| < \e\qquad \forall n>N
\end{align*}
so the statement \cond3 is proved. 
\end{proof}

Aside we mention that, in statements \cond1, \cond3 of the above Lemma \ref{lem:strong}, we do not
need to assume that the Hilbert space $\cH$ is separable; separability only appears in statement
\cond2, since there we assumed existence or (countable) orthonormal basis.

Let $(\fX, \cA)$ be a measurable space, and let $X$, $Y$, $Z$ be Banach spaces. 

\begin{lem}\label{l:prod meas funct}
Let $F:\fX\to B(Y;Z)$, $G:\fX\to B(X;Y)$, $g:\fX\to Y$. 
\begin{enumerate}
\item If $F$ and $g$ are strongly measurable, the the product $Fg$ \tup(i.e.~the function $s\mapsto 
F(s)g(s)$\tup) is strongly measurable
\item If $F$, $G$ are strongly measurable, then the product $FG$ \tup(i.e.~the function $s\mapsto 
F(s)G(s)$\tup) is strongly measurable. 
\end{enumerate}
\end{lem}

\begin{proof}
For a strongly measurable $g$ let $g_n$, $n\ge 1$ be a sequence of simple functions (i.e.~of the 
measurable functions taking finitely many values) converging to $g$ pointwise. Then the functions 
$s\mapsto F(s) g_n(s)$ are clearly measurable. Since $F(s)$ is a bounded operator  $F(s)f_n(s)\to 
F(s)$ in norm, so $f$ is a measurable function as a pointwise limit of measurable functions.  
Statement \cond1 is proved. 

To prove statement \cond2 take arbitrary $\bx \in X$. Since $G$ is strongly measurable, the 
function $s\mapsto G(s)\bx$ is strongly measurable by definition. Applying statement \cond1 to this 
function, we get that the function $s\mapsto F(s)G(s)\bx$ is strongly measurable. Since $\bx$ is 
arbitrary, statement \cond2 is proved. 
\end{proof}

In the rest of this section all Hilbert spaces are separable, so for functions with values in 
$B(\cH)$ weak and strong measurability coincide. So we will simply  call such  functions 
measurable.

\begin{lem}\label{lem:measurable}
Let $F:\fX\to B(\cH)$, $F(s)=F(s)^*$, be $\cA$-measurable \tup(weakly, or equivalently strongly,
since $\cH$ is separable\tup), and let $\f:\R\to \C$ be a locally bounded \tup(i.e.~bounded on any
finite interval\tup) Borel measurable function.

Then the function $s\mapsto \f(F(s))$ is measurable. 
\end{lem}

\begin{proof}
By Lemma \ref{l:prod meas funct} the conclusion of the  lemma trivially holds when $\f$ is a 
polynomial. 

Let now $\f\in C(\R)$. By the Weierstrass Approximation Theorem there exists a sequence of 
polynomials $p_n$ such that 
\begin{align*}
\|\f -p_n \|\ci{C([-n,n])} < 1/n. 
\end{align*}
Therefore for any $s\in \fX$ we have $\|\f(F(s)) -p_n (F(s))\|\to 0$ as $n\to \infty$, so 
\begin{align*}
p_n (F(s)) \stackrel{\textup{s}}{\longrightarrow} \f(F(s)) \qquad \text{as }n\to \infty. 
\end{align*}
Again, since the pointwise limit of strongly measurable functions is strongly measurable, see 
\cite[Corollary 1.1.9]{Hytoenen-book2016}, we conclude that the function $s\mapsto \f(F(s))$ is 
strongly measurable. 

To extend the result to Borel measurable functions, it is sufficient to prove it for real-valued 
functions: considering $\re\f$ and $\im \f$ separately, we get the general case. For a real-valued 
$\f$  we will use the Monotone class theorem (Theorem \ref{t:monotone class}). 

Fix a measurable function $F:\fX\to B(\cH)$,  and define the class $\cS$ to be a collection of all
Borel measurable functions $\f:\R\to \R$ such that $s\mapsto \f(F(s))$ is $\cA$-measurable.  Note
that we had shown before that continuous functions belong to $\cS$.  Define also the $\pi$-system
$\cT$ to be the collection of all open bounded intervals  $I\subset \R$.

Trivially $\cS$ satisfies conditions \cond1 and \cond2 of Theorem \ref{t:monotone class}.  

Let us check condition \cond3. Take functions $g_n\in \cS$, $0\le g_n  \nearrow g $, where $g$ is a 
bounded function. Recall that the functions in $\cS$ are Borel measurable by definition, and so $g$ 
is also Borel measurable as pointwise limit of measurable functions. 

We can see from the Monotone 
Convergence Theorem that  for arbitrary self-adjoint operator $A$ and a vector $\bx\in\cH$ 
\begin{align}\label{e:monotone quad}
\Jap{g_n(A) \bx, \bx}=\int g_n\dd \mu_\bx  \nearrow \int g\dd \mu_\bx  = \Jap{g(A) \bx, \bx}  ;
\end{align}
here $\mu_\bx$ is the scalar spectral measure of the operator $A$ corresponding to the vector $\bx$,
see Section  \ref{s:Sp Thm}. Note that the boundedness of $g$ guarantee that the right hand side of
\eqref{e:monotone quad} is bounded.

Therefore for any $\bx\in\cH$ the function $s\mapsto \Jap{g(F(s))\bx, \bx}$ is measurable as a limit
of measurable functions, so via polarization the function $s\mapsto g (F(s))$ is weakly (and so
strongly) measurable.

Finally, for any open interval $I$ one can find an increasing sequence of $g_n\in C\ti{c}$, $0\le 
g_n \nearrow \1\ci{I} $, so $\1\ci{I}\in\cS$. Thus condition \cond4 of Theorem \ref{t:monotone 
class} is satisfied.

Since the $\sigma$-algebra generated by bounded open intervals is exactly the Borel 
$\sigma$-algebra, we conclude using Theorem \ref{t:monotone class} that $\cS$ contains all bounded 
Borel measurable functions. 

Let now $\f$ be a locally bounded Borel measurable function. Define 
\[
\f_n(t):= \1\ci{[-n,n]}(t) 
\f(t).
\]
Clearly $\f_n$ are bounded Borel measurable functions, so $\f_n\in\cS$. Since for any self-adjoint 
$A$ 
\begin{align*}
\|\f(A) -\f_n(A)\|\to 0\qquad \text{as } n\to\infty
\end{align*}
(the sequence $\f_n(A)$ stabilizes). 

Therefore the function $\f(F(\fdot))$ is a pointwise limit (in the norm topology of $B(\cH)$, and 
therefore in the strong operator topology) of measurable functions $\f_n(F(\fdot))$, so 
$\f(F(\fdot))$ is measurable. 
\end{proof}

%%%%%%%%%%%%%%%%%%%%%%%%%%%%%%%%
\section{Poisson Lebesgue differentiation theorem for general measures}
%%%%%%%%%%%%%%%%%%%%%%%%%%%%%%%%
Lemma \ref{l:weak type 01} below is sometimes called the Relative Fatou's  Theorem, and was proved
in \cite{Doob1959}. We needed also needed its (a bit weaker) vector-valued version, see Lemma
\ref{l:P-L differ} below, but we could not find a reference, so we present it here.

We say that a measure $\mu$ in $\R^n$ is Poisson finite if the measure $\wt\mu$, 
\begin{align}\label{e:PoissFinite}
\wt\mu (\dd s)=(1+|s|)^{-(n+1)}\mu(\dd s)
\end{align}
is finite. 

Let $\R^{n+1}_+:= \R^n\times (0,\infty)$. For a (Poisson finite) measure $\mu$ on $\R^n$ and 
$z\in\R^{n+1}_+$ slightly abusing notation denote by $\mu(z)$ the Poisson extension of $\mu$ to the 
point $z$. If $f$ is a Bochner integrable (with respect to $\mu$) function with values in a Banach 
space, we denote by $(f\mu)(z)$ the Poisson extension of vector-valued measure $f\mu$ to the point 
$z\in\R^{n+1}_+$, and similarly for a scalar-valued function. 

For $x\in \R^n$ denote by $R_x=R_{x, \alpha} $ the non-tangential approach region (cone) in 
$\R^{n+1}_+$, 
\begin{align*}
R_{x} =\left\{ (x',t)\in \R^{n+1}_+ : | x-x'|\le t \tan\alpha   \right\};
\end{align*}
here  $\alpha\in (0,\pi/2)$ is fixed (the resuts do not depend on $\alpha$) and $|\cdot|$ stands 
for the standard Euclidean norm in $\R^n$.

We say that $z\in \R^{n+1}_+$ approaches $x\in\R^n$ non-tangentially and write $z\to 
x\sphericalangle$ if $z\to x$ while remaining in $ R_x$.

\begin{lem}\label{l:P-L differ}
Let $\mu$ be a Poisson finite measure in $\R^n$, and let $\wt\mu$ be defined by
\eqref{e:PoissFinite}. If $f$ is a Bochner integrable function with respect to the measure 
$\wt \mu$ with values in a Banach space, then
\begin{align*}
\frac{(f\mu)(z)}{\mu(z)} \underset{z\to x\sphericalangle}{\longrightarrow} f(x) \qquad 
\mu\text{-a.e.}
\end{align*}
\end{lem}

\begin{proof}
Trivially, it is sufficient to prove the lemma for finite compactly supported measures, so let us 
assume that. 

The statement of the theorem is trivial for continuous functions. In the general case the crucial
step is weak type estimates of the nontangential maximal function, see Lemma \ref{l:weak type}
below.

Knowing this fact, we follow the classical textbook route:  approximating the function $f$ in 
$L^1(\mu)$ by continuous functions $f_n$ and using the weak type estimates for the scalar-valued 
functions $x\mapsto \|f(x)-f_n(x)\|$ we show that for arbitrary 
$\e>0$ the measure $\mu$ of the set of $x\in\R^n$ where 
\begin{align*}
\limsup_{{(x,t)\to x\sphericalangle}}  \int P_t(x-x') \| f(x') - f(x)\| \mu(\dd x') \ge\e
\end{align*}
is arbitrarily small; here $P_t(x) $ is the Poisson kernel. 

This immediately implies the lemma.
\end{proof}

For a Poisson finite measure $\mu$ in $\R^n$ denote by $M\ut{P,nt}_\mu$ the nontangential Poisson 
maximal functions, 
\begin{align*}
M\ut{P,nt}_\mu f(x) = \sup\left\{ \frac{(|f|\mu)(z)}{\mu(z)}: z\in R_{x}\right\},
\end{align*}
where $f\in L^1(\wt\mu)$ is a scalar-valued function.

\begin{lem}\label{l:weak type}
Let $\mu$ be a Poisson finite measure in $\R^n$. Then the maximal function $M\ut{P,nt}_\mu$ has the 
weak type $1$-$1$, i.e.~there exists $C<\infty$ such that for all $\lambda>0$
\begin{align*}
\mu\left(   \left\{  x\in\R^n: M\ut{P,nt}_\mu f(x)  > \lambda \right\}    \right) \le 
C\lambda^{-1} \|f\|\ci{L^1(\mu)}.
\end{align*}
\end{lem}

\begin{proof}
Consider the \emph{vertical} Poisson maximal function $M\ut{P}_\mu$
\begin{align*}
M\ut{P}_\mu f(x) = \sup\left\{ \frac{(|f|\mu)(x,t)}{\mu(x,t)}: t>0\right\}. 
\end{align*}
By the standard Harnack estimates there exists $C=C_{\alpha,n}$ such that 
$M\ut{P,nt}_\mu f(x) \le C M\ut{P}_\mu f(x)$. Therefore it is sufficient to prove the weak type 
estimates for $M\ut{P}_\mu$.

We will use the fact that the centered maximal function $M_\mu$, 
\begin{align*}
M_\mu f(x):= \sup_{r>0} \mu(B(x,r))^{-1}\int_{B(x,r)} |f|\mu
\end{align*}
is of weak type $1$-$1$. This fact is proved using the Besicovitch covering theorem, although for 
$n=1$ (which we need) an elementary proof is possible.  

To prove the weak type estimates for $M\ut{P}$ take $\lambda>0$ and let $x\in\R^n$ be such that 
$M\ut{P}_\mu f(x)>\lambda$. This means that there exists $t>0$ such that 
\begin{align*}
\int P_t(x-x') |f(x')| \mu(\dd x') > \lambda \int P_t(x-x')  \mu(\dd x'). 
\end{align*}
Evaluating the above integrals via distribution functions (with respect to the measures $|f|\mu$ 
and $\mu$) and recalling that the level sets of the function $x'\mapsto P_t(x-x')$ are balls 
centered at $x$, we can find $r>0$ such that 
\begin{align*}
\int_{B(x,r)} |f|\mu >\lambda \mu(B(x,r)), 
\end{align*}
which means that $M_\mu f(x) >\lambda$. This means 
\begin{align*}
\left\{ x\in\R^n: M\ut{P}_\mu f(x) >\lambda \right\} \subset 
\left\{ x\in\R^n: M_\mu f(x) >\lambda \right\}
\end{align*}
which together with weak type 1-1 estimates for $M_\mu$ proves the lemma. 
\end{proof}

\begin{lem}\label{l:weak type 01}
Let $\mu$ and $\nu$ be \tup(positive\tup) Poisson finite measures on $\R^n$. Then 
\begin{align*}
\frac{\nu(z)}{\mu(z)} \underset{z\to x\sphericalangle}{\longrightarrow} \frac{\dd \nu}{\dd\mu}(x) 
\qquad 
\mu\text{-a.e.},
\end{align*}
where $\frac{\dd \nu}{\dd\mu}$ is the Radon--Nikodym derivative.
\end{lem}
\begin{proof}
Previous Lemma \ref{l:P-L differ} covers the case when $\nu$ is absolutely continuous with respect 
to $\mu$, so it remain to show that if $\mu\perp\nu$ then 
\begin{align*}
\frac{\nu(z)}{\mu(z)} \underset{z\to x\sphericalangle}{\longrightarrow} 0 
\qquad 
\mu\text{-a.e.}
\end{align*}
The proof below is essentially the proof of Lemma 1.3 in \cite{NONTAN}, and is presented only for 
the convenience of the reader. 

Since $\mu\perp\nu$, there exists a Borel set $E\subset\R^n$ such that $\mu(E\ut{c}) = 
\varnothing$, $\nu(E) = \varnothing$, i.e.~that $\mu$ and $\nu$ are carried by $E$ and $E\ut{c}$ 
respectively. Applying Lemma \ref{l:P-L differ} to the measure $\sigma:=\mu+\nu$ and the function 
$f:=\1\ci{E}$ we get that
\begin{align*}
\frac{\mu(z)}{\mu(z)+\nu(z)} = \frac{\left(f\sigma\right)(z)}{\sigma(z)} 
\underset{z\to x\sphericalangle}{\longrightarrow}   1
\qquad \sigma\text{-a.e.~on }E  .
\end{align*}
Since $\mu$ is carried by $E$ and $\mu\perp\nu$, the qualifier ``$\sigma$-a.e.~on $E$'' is 
equivalent to ``$\mu$-a.e.'' Therefore, taking the reciprocal, 
\begin{align*}
1+ \frac{\nu(z)}{\mu(z)} = \frac{\mu(z)+\nu(z)}{\mu(z)} \underset{z\to 
x\sphericalangle}{\longrightarrow}   1 
\qquad \mu\text{-a.e.}, 
\end{align*}
so 
\begin{align*}
\frac{\nu(z)}{\mu(z)} \underset{z\to x\sphericalangle}{\longrightarrow} 0 \qquad \mu\text{-a.e.}
\end{align*}
\end{proof}

\providecommand{\bysame}{\leavevmode\hbox to3em{\hrulefill}\thinspace}
\providecommand{\MR}{\relax\ifhmode\unskip\space\fi MR }
\providecommand{\MRhref}[2]{%
  \href{http://www.ams.org/mathscinet-getitem?mr=#1}{#2}
}
\providecommand{\href}[2]{#2}

\end{document}